\documentclass{article}
\usepackage{amsmath}
\usepackage{amsthm}

\usepackage{textcomp}
\usepackage{times}
\usepackage{xcolor}
\usepackage{enumitem}
\usepackage{xparse,etoolbox}
\usepackage{caption}
\usepackage{chngcntr}
\counterwithin{figure}{section}
\counterwithin{table}{section}
\usepackage{tikz}
\usepackage{graphicx}
\usepackage[mathscr]{eucal}
\usepackage{scrextend}
\usepackage{wasysym}
\usepackage{etoolbox}
\makeatletter

\usepackage[all]{xy}

\usepackage{etoolbox}
\makeatletter
\patchcmd{\@makechapterhead}{50\p@}{\chapheadtopskip}{}{}

\patchcmd{\@makeschapterhead}{50\p@}{\chapheadtopskip}{}{}
\makeatother
\newlength{\chapheadtopskip}
\usepackage[parfill]{parskip} 
\usepackage{amssymb}
\usepackage{imakeidx}
\usepackage{enumitem}
\usepackage[mathlines]{lineno}
\DeclareMathAlphabet{\mathpzc}{OT1}{pzc}{m}{it}
\makeatletter
\newcommand{\mylabel}[2]{#2\def\@currentlabel{#2}\label{#1}}
\makeatother
\usepackage{fancyhdr}
\usepackage{enumitem}
\usepackage{xpatch}
\newtheorem{theorem}{Theorem}[section]
\newtheorem{lemma}[theorem]{Lemma}

\newtheorem{defn}[theorem]{Definition}

\newtheorem{cor}[theorem]{Corollary}

\newtheorem{Claim}[theorem]{Claim}

\newtheorem{subclaim}{Subclaim}[theorem]
\newcounter{claimlevel}[theorem]

\ExplSyntaxOn
\NewDocumentEnvironment{claim}{O{=}}
 {
  \str_case:nn { #1 }
   {
    {=}  { }
    {+}  { \stepcounter{claimlevel} }
    {-}  { \addtocounter{claimlevel}{-1} }
   }
  \begin{ Claim \int_to_Roman:n { \value{claimlevel} } }
 }
 {
  \end{ Claim \int_to_Roman:n { \value{claimlevel} } }
 }
\ExplSyntaxOff

\newenvironment{claimproof}[1]{\par\noindent\underline{Proof:}\space#1}{\hspace{1mm}$\blacksquare$}

\usepackage[left=2.54cm,top=2.54cm,right=2.54cm,bottom=2.54cm]{geometry}
\usepackage{setspace}
\newlength\FHoffset
\fancyheadoffset{\FHoffset}
\newlength\FHright
\makeatletter
\usepackage{framed}

 \newtheoremstyle{TheoremNum}
        {\topsep}{\topsep}              
        {\itshape}                      
        {}                              
        {\bfseries}                     
        {.}                             
        { }                             
        {\thmname{#1}\thmnote{ \bfseries #3}}
    \theoremstyle{TheoremNum}
    \newtheorem{thmn}{Theorem}

\newtheoremstyle{PropNum}
        {\topsep}{\topsep}              
        {\itshape}                      
        {}                              
        {\bfseries}                     
        {.}                             
        { }                             
        {\thmname{#1}\thmnote{ \bfseries #3}}
    \theoremstyle{PropNum}

\newtheoremstyle{LemmaNum}
        {\topsep}{\topsep}              
        {\itshape}                      
        {}                              
        {\bfseries}                     
        {.}                             
        { }                             
        {\thmname{#1}\thmnote{ \bfseries #3}}
    \theoremstyle{LemmaNum}
    
\makeatletter
\renewcommand\subitem{\@idxitem\nobreak\hspace*{20\p@}}
\renewcommand\subsubitem{\@idxitem\nobreak\hspace*{20\p@}}
\makeatother

\date{}
\makeatother
\title{Extensions of Thomassen's Theorem to Paths of Length At Most Four: Part III}
\author{Joshua Nevin}
\begin{document}
\maketitle

\begin{center}\textbf{Abstract}\end{center} Let $G$ be a planar embedding with list-assignment $L$ and outer cycle $C$, and let $P$ be a path of length at most four on $C$, where each vertex of $G\setminus C$ has a list of size at least five and each vertex of $C\setminus P$ has a list of size at least three. This is the final paper in a sequence of three papers in which we prove some results about partial $L$-colorings $\phi$ of $C$ with the property that any extension of $\phi$ to an $L$-coloring of $\textnormal{dom}(\phi)\cup V(P)$ extends to $L$-color all of $G$, and, in particular, some useful results about the special case in which $\textnormal{dom}(\phi)$ consists only of the endpoints of $P$. We also prove some results about the other special case in which $\phi$ is allowed to color some vertices of $C\setminus\mathring{P}$ but we avoid taking too many colors away from the leftover vertices of $\mathring{P}\setminus\textnormal{dom}(\phi)$. We use these results in a later sequence of papers to prove some results about list-colorings of high-representativity embeddings on surfaces. 

\section{Introduction}\label{IntroMotivSec}

Given a graph $G$, a \emph{list-assignment} for $G$ is a family of sets $\{L(v): v\in V(G)\}$ indexed by the vertices of $G$, such that $L(v)$ is a finite subset of $\mathbb{N}$ for each $v\in V(G)$. The elements of $L(v)$ are called \emph{colors}. A function $\phi:V(G)\rightarrow\bigcup_{v\in V(G)}L(v)$ is called an \emph{$L$-coloring of} $G$ if $\phi(v)\in L(v)$ for each $v\in V(G)$, and, for each pair of vertices $x,y\in V(G)$ such that $xy\in E(G)$, we have $\phi(x)\neq\phi(y)$. 

Given a set $S\subseteq V(G)$ and a function $\phi: S\rightarrow\bigcup_{v\in S}L(v)$, we call $\phi$ an \emph{ $L$-coloring of $S$} if $\phi(v)\in L(v)$ for each $v\in S$ and $\phi$ is an $L$-coloring of the induced graph $G[S]$. A \emph{partial} $L$-coloring of $G$ is a function of the form $\phi:S\rightarrow\bigcup_{v\in S}L(v)$, where $S$ is a subset of $V(G)$ and $\phi$ is an $L$-coloring of $S$. Likewise, given a set $S\subseteq V(G)$, a \emph{partial $L$-coloring} of $S$ is a function $\phi:S'\rightarrow\bigcup_{v\in S'}L(v)$, where $S'\subseteq S$ and $\phi$ is an $L$-coloring of $S'$. Given an integer $k\geq 1$, a graph $G$ is called \emph{$k$-choosable} if, for every list-assignment $L$ for $G$ such that $|L(v)|\geq k$ for all $v\in V(G)$, $G$ is $L$-colorable.

This paper is the third in a sequence of three papers. The motivation for this sequence of papers is outlined in the introduction of Paper I (\cite{JNHolepunchPaperICitation}). We use the main result of this sequence of papers in a later sequence of papers to prove a result about high-representativty embeddings on surfaces with distant precolored components.  We now recall some notation from Paper I. Given a graph $G$ with list-assignment $L$, we very frequently analyze the situation where we begin with a partial $L$-coloring $\phi$ of a subgraph of $G$, and then delete some or all of the vertices of $\textnormal{dom}(\phi)$ and remove the colors of the deleted vertices from the lists of their neighbors in $G\setminus\textnormal{dom}(\phi)$. We thus make the following definition. 

\begin{defn}\emph{Let $G$ be a graph, let $\phi$ be a partial $L$-coloring of $G$, and let $S\subseteq V(G)$. We define a list-assignment $L^S_{\phi}$ for $G\setminus (\textnormal{dom}(\phi)\setminus S)$ as follows.}
$$L^S_{\phi}(v):=\begin{cases} \{\phi(v)\}\ \textnormal{if}\ v\in\textnormal{dom}(\phi)\cap S\\ L(v)\setminus\{\phi(w): w\in N(v)\cap (\textnormal{dom}(\phi)\setminus S)\}\ \textnormal{if}\ v\in V(G)\setminus \textnormal{dom}(\phi) \end{cases}$$ \end{defn}

If $S=\varnothing$, then $L^{\varnothing}_{\phi}$ is a list-assignment for $G\setminus\textnormal{dom}(\phi)$ in which the colors of the vertices in $\textnormal{dom}(\phi)$ have been deleted from the lists of their neighbors in $G\setminus\textnormal{dom}(\phi)$. The situation where $S=\varnothing$ arises so frequently that, in this case, we simply drop the superscript and let $L_{\phi}$ denote the list-assignment $L^{\varnothing}_{\phi}$ for $G\setminus\textnormal{dom}(\phi)$. In some cases, we specify a subgraph $H$ of $G$ rather than a vertex-set $S$. In this case, to avoid clutter, we write $L^H_{\phi}$ to mean $L^{V(H)}_{\phi}$. Our main result for this paper, which we prove in the final section of this paper, is Theorem \ref{MainHolepunchPaperResulThm} below.

\begin{theorem}\label{MainHolepunchPaperResulThm} (Holepunch Theorem) Let $G$ be a planar embedding with outer cycle $C$. Let $P:=p_0q_0zq_1p_1$ be a subpath of $C$ whose three internal vertices have no common neighbor in $C\setminus P$, and let $L$ be a list-assignment for $V(G)$ such that the following hold.
\begin{enumerate}[label=\arabic*)] 
\itemsep-0.1em
\item $|L(p_0)|+|L(p_1)|\geq 4$ and each of $L(p_0)$ and $L(p_1)$ is nonempty; AND
\item For each $v\in V(C\setminus P)$, $|L(v)|\geq 3$; AND
\item For each $v\in\{q_0, z, q_1\}\cup V(G\setminus C)$, $|L(v)|\geq 5$.
\end{enumerate}
Then there is a partial $L$-coloring $\phi$ of $V(C)\setminus\{q_0, q_1\}$, where $p_0, p_1, z\in\textnormal{dom}(\phi)$, such that each of $q_0, q_1$ has an $L_{\phi}$-list of size at least three, and furthermore, any extension of $\phi$ to an $L$-coloring of $\textnormal{dom}(\phi)\cup\{q_0, q_1\}$ also extends to $L$-color all of $G$. 
 \end{theorem}

This paper consists entirely of the proof of Theorem \ref{MainHolepunchPaperResulThm}. In order to prove this, we need some intermediate facts about paths of lengths 2, 3, and 4 in facial cycles of planar graphs, which we proved in Papers I-II and use as black boxes in this paper. In Sections \ref{NotationFromPISec}-\ref{PIBlackBoxSec}, we recall the notation that we introduced in Paper I and the statements of the results that we proved in Papers I. In Section \ref{PIITwoBlackBoxSec}, we recall the two results we proved in Paper II. It is not necessary to have read either of Papers I or II in order to read this paper (more generally, any of the three papers can be read independently of the other two). All of the notation introduced in those papers is re-introduced below, and we restate the results from Papers I-II in the same language. 

\section{Background}\label{BackgroundSect}

In 1994, Thomassen demonstrated in \cite{AllPlanar5ThomPap} that all planar graphs are 5-choosable, settling a problem that had been posed in the 1970's. Actually, Thomassen proved something stronger. 

\begin{theorem}\label{thomassen5ChooseThm}
Let $G$ be a planar graph with facial cycle $C$ and let $xy\in E(C)$. Let $L$ be a list assignment for $G$, where  vertex of $G\setminus C$ has a list of size at least five and each vertex of $V(C)\setminus\{x,y\}$ has a list of size at least three, where $xy$ is $L$-colorable. Then $G$ is $L$-colorable.
\end{theorem}

Theorem \ref{thomassen5ChooseThm} has the following two useful corollaries, which we use frequently. 

\begin{cor}\label{CycleLen4CorToThom} Let $G$ be a planar graph with outer cycle $C$ and let $L$ be a list-assignment for $G$ where each vertex of $G\setminus C$ has a list of size at least five.  If $|V(C)|\leq 4$ then any $L$-coloring of $V(C)$ extends to an $L$-coloring of $G$. \end{cor}

\begin{cor}\label{2ListsNextToPrecEdgeCor}
Let $G$ be a planar graph with facial cycle $C$ and $p_0p_1\in E(C)$ and, for each $i=0,1$, let $u_i$ be the unique neighbor of $p_i$ on the path $C\setminus\{p_0, p_1\}$. Let $L$ be a list assignment for $G$, where each vertex of $G\setminus C$ has a list of size at least five and each vertex of $C\setminus\{p_0, p_1, u_0, u_1\}$ has a list of size at least three.  Let $\phi$ be an $L$-coloring of $p_0p_1$, where $|L(u_i)\setminus\{\phi(p_i)\}|\geq 2$ for each $i=0,1$. Then $\phi$ extends to an $L$-coloring of $G$. \end{cor}

Because of Corollary \ref{CycleLen4CorToThom}, planar embeddings which have no separating cycles of length 3 or 4 play a special role in our analysis, so we give them a name.

\begin{defn} \emph{Given a planar graph $G$, we say that $G$ is \emph{short-separation-free} if, for any cycle $F\subseteq G$ with $|V(F)|\leq 4$, either $V(\textnormal{Int}_G(F))=V(F)$ or $V(\textnormal{Ext}_G(F))=V(F)$.} \end{defn}

Lastly, we rely on the following very simple result, which is a useful consequence of Theorem 7 from \cite{lKthForGoBoHm6} which characterizes the obstructions to extending a precoloring of a cycle of length at most six in a planar graph. 

\begin{theorem}\label{BohmePaper5CycleCorList} Let $G$ be a short-separation-free graph with facial cycle $C$. Let $L$ be a list-assignment for $G$, where $|L(v)|\geq 5$ for all $v\in V(G\setminus C)$. Suppose that $|V(C)|\leq 6$ and $V(C)$ is $L$-colorable, but $G$ is not $L$-colorable. Then $5\leq |V(C)|\leq 6$, and the following hold.
\begin{enumerate}[label=\arabic*)]
\itemsep-0.1em
\item If $|V(C)|=5$, then $G\setminus C$ consists of a lone vertex which is adjacent to all five vertices of $C$; AND
\item If $|V(C)|=6$, then $G\setminus C$ consists of at most three vertices, each of which has at least three neighbors in $G\setminus C$. Furthermore, $G\setminus C$ consists of one of the following.
\begin{enumerate}[label=\roman*)]
\itemsep-0.1em
\item A lone vertex adjacent to at least five vertices of $C$; OR
\item An edge $x_1x_2$ such that, for each $i=1,2$, $G[N(x_i)\cap V(C)]$ is a path of length three; OR
\item A triangle $x_1x_2x_3$ such that, for each $i=1,2,3$, $G[N(x_i)\cap V(C)]$ is a path of length two.
\end{enumerate}
\end{enumerate}
\end{theorem}

\section{Notation from Paper I}\label{NotationFromPISec}

In this section, we restate some terminology introduced in Paper I. We first have the following definition, which is our main object of study both for the results of Papers I-II and the result of this paper. 

\begin{defn} \emph{A \emph{rainbow} is a tuple $(G, C, P, L)$, where $G$ is a planar graph with outer cycle $C$, $P$ is a path on $C$ of length at least two, and $L$ is a list-assignment for $V(G)$ such that $|L(v)|\geq 3$ for each $v\in V(C\setminus P)$ and $|L(v)|\geq 5$ for each $v\in V(G\setminus C)$. We say that a rainbow is \emph{end-linked} if, letting $p ,p^*$ be the endpoints of $P$, each of $L(p)$ and $L(p^*)$ is nonempty and $|L(p)|+|L(p^*)|\geq 4$. }
 \end{defn}

We also recall the following definitions. 

\begin{defn} \emph{Given a graph $G$ with list-assignment $L$, a subgraph $H$ of $G$, and a partial $L$-coloring $\phi$ of $G$, we say that $\phi$ is \emph{$(H,G)$-sufficient} if any extension of $\phi$ to an $L$-coloring of $\textnormal{dom}(\phi)\cup V(H)$ extends to $L$-color all of $G$.} \end{defn}

\begin{defn}\label{GeneralAugCrownNotForLink} \emph{Let $G$ be a planar graph with outer cycle $C$, let $L$ be a list-assignment for $G$, and let $P$ be a path in $C$ with $|V(P)|\geq 3$. Let $pq$ and $p'q'$ be the terminal edges of $P$, where $p, p'$ are the endpoints of $P$. We let $\textnormal{Crown}_{L}(P, G)$ be the set of partial $L$-colorings $\phi$ of $V(C)\setminus\{q, q'\}$ such that}
\begin{enumerate}[label=\arabic*)] 
\itemsep-0.1em
\item $V(P)\setminus\{q, q'\}\subseteq\textnormal{dom}(\phi)$ and, for each $x\in\{q, q'\}$, $|L_{\phi}(x)|\geq |L(x)|-2$; AND
\item\emph{$\phi$ is $(P, G)$-sufficient}
\end{enumerate}
 \end{defn}

With the definitions above in hand, we have the following compact restatement of Theorem \ref{MainHolepunchPaperResulThm}.

\begin{thmn}[\ref{MainHolepunchPaperResulThm}] Let $(G, C, P, L)$ be an end-linked rainbow, where $P$ is a path of length four whose three internal vertices have no common neighbor in $C\setminus P$. Suppose further that each internal vertex of $P$ has an $L$-list of size at least five. Then $\textnormal{Crown}_L(P, G)\neq\varnothing$. \end{thmn} 

We also recall the following notation from Paper I and the standard notation of Definition \ref{StandQD}.

\begin{defn} \emph{Given a planar graph $G$ with outer face $C$ and an $H\subseteq G$, we let $C^H$ denote the outer face of $H$.} \end{defn}

\begin{defn}\label{StandQD} \emph{Given a path $Q$ in a graph $G$, we let $\mathring{Q}$ denote the subpath of $Q$ consisting of the internal vertices of $Q$. In particular, if $|E(Q)|\leq 2$, then $\mathring{Q}=\varnothing$. Furthermore, for any $x,y\in V(Q)$, we let $xQy$ denote the unique subpath of $Q$ with endpoints $x$ and $y$. } \end{defn}

As we frequently deal with colorings of paths, we also use the following natural notation. 

\begin{defn} \emph{Let $G$ be a graph with list-assignment $L$. Given an integer $n\geq 1$, a path $P:=p_1\ldots p_n$, and a partial $L$-coloring $\phi$ of $G$ with $V(P)\subseteq\textnormal{dom}(\phi)$, we denote the $L$-coloring $\phi|_P$ of $P$ as $(\phi(p_1), \phi(p_2),\ldots, \phi(p_n))$.} \end{defn}

\begin{defn}\label{DefnForColorSetsonVertex} \emph{Let $G$ be a graph with list-assignment $L$. Given a set $\mathcal{F}$ of partial $L$-colorings of $G$ and a vertex $x\in V(G)$ with $x\in\textnormal{dom}(\phi)$ for each $\phi\in\mathcal{F}$, we define $\textnormal{Col}(\mathcal{F}\mid x)=\{\phi(x): \phi\in\mathcal{F}\}$.} \end{defn}

\section{Black Boxes from Paper I}\label{PIBlackBoxSec}

The first black box from Paper I that we have is a result about broken wheels.

\begin{defn}\emph{A \emph{broken wheel} is a graph $G$ with a vertex $p\in V(G)$ such that $G-p$ is a path $q_1\ldots q_n$ with $n\geq 2$, where $N(p)=\{q_1, \ldots, q_n\}$. The subpath $q_1pq_n$ of $G$ is called the \emph{principal path} of $G$.} \end{defn}

Note that, if $|V(G)|\leq 4$, then the above definition does not uniquely specify the principal path, although in practice, whenever we deal with broken wheels, we specify the principal path beforehand so that there is no ambiguity. 

\begin{defn} \emph{Let $G$ be a graph and let $L$ be a list-assignment for $V(G)$. Let $P:=p_1p_2p_3$ be a path of length two in $G$. For each $(c, c')\in L(p_1)\times L(p_3)$, we let $\Lambda_{G,L}^P(c, \bullet, c')$ be the set of $d\in L(p_2)$ such that there is an $L$-coloring of $G$ which uses $c,d,c'$ on the respective vertices $p_1, p_2, p_3$. Likewise,  given a pair $(c, c')$ in either $L(p_1)\times L(p_2)$ or $L(p_2)\times L(p_3)$ respectively, we define the sets $\Lambda_{G,L}^P(c, c', \bullet)$  and $\Lambda_{G,L}^P(\bullet, c, c')$  analogously.}  \end{defn}

In the setting above, we have $\Lambda_{G,L}^P(c, c, \bullet)=\varnothing$ for any $c\in L(p_1)\cap L(p_2)$. Likewise, $\Lambda_{G,L}(\bullet, d, d)=\varnothing$ for any $d\in L(p_2)\cap L(p_3)$. The use of the notation above always requires us to specify an ordering of the vertices of a given 2-path. That is, whenever we write $\Lambda_{G,L}^P(\cdot, \cdot, \cdot)$, where two of the coordinates are colors of two of the vertices of $P$ and one is a bullet denoting the remaining uncolored vertex of $P$, we have specified beforehand which vertices the first, second, and third coordinates correspond to. Sometimes we make this explicit by writing $\Lambda^{p_1p_2p_3}_{G,L}(\cdot, \cdot, \cdot)$. Whenever any of $P, G, L$ are clear from the context, we drop the respective super- or subscripts from the notation above. 

\begin{defn}\label{GUniversalDefinition} \emph{Let $G$ be a graph and let $P:= p_1p_2p_3$ be a subpath of $G$ of length two. Let $L$ be a list-assignment for $V(G)$. Given an $a\in L(p_3)$,}
\begin{enumerate}[label=\emph{\alph*)}] 
\itemsep-0.1em
\item\emph{we say that $a$ is \emph{$(G, P)$-universal} if, for each $b\in L(p_2)\setminus\{a\}$, we have $\Lambda_G^P(\bullet, b, a)=L(p_1)\setminus\{b\}$.}
\item\emph{We say that $a$ is \emph{almost $(G,P)$-universal}, if, for each $b\in L(p_2)\setminus\{a\}$, we have $|\Lambda_G^P(\bullet, b, a)|\geq |L(p_1)|-1$.}
\end{enumerate}
 \end{defn}

In the setting above, if $a$ is $(G,P)$-universal, then it is clearly also almost $(G,P)$-universal, and, if $p_1p_3\in E(G)$ and $L(p_3)\subseteq L(p_1)$, then there is no $(G, P)$-universal color in $L(p_3)$. That is, $a$ being a $(G, P)$-universal color of $L(p_3)$ is a stronger property than the property that any $L$-coloring of $V(P)$ using $a$ on $p_3$ extends to an $L$-coloring of $G$, unless either $a\not\in L(p_1)$ or $p_1p_3\not\in E(G)$. If the 2-path $P$ is clear from the context, then, given an $a\in L(p_3)$, we just say that $a$ is $G$-universal.  Our first and second black boxes from Paper I are the following two results.

\begin{theorem}\label{BWheelMainRevListThm2} Let $G$ be a broken wheel with principal path $P=pp'p''$ and let $L$ be a list-assignment for $V(G)$ in which each vertex of $V(G)\setminus\{p, p'\}$ has a list of size at least three. Let $G-p'=pu_1\ldots u_tp''$ for some $t\geq 0$. 
\begin{enumerate}[label=\arabic*)]
\item Let $\phi_0, \phi_1$ be a pair of distinct $L$-colorings of $pp'$. For each $i=0,1$, let $S_i:=\Lambda_G(\phi_i(p), \phi_i(p'), \bullet)$, and suppose that $|S_0|=|S_1|=1$. Then the following hold. 
\begin{enumerate}[label=\alph*)]
\itemsep-0.1em
\item If $\phi_0(p)=\phi_1(p)$ and $S_0=S_1$, then $|E(G-p')|$ is even; AND
\item If $\phi_0(p)=\phi_1(p)$ and $S_0\neq S_1$, then $|E(G-p')|$ is odd and, for each $i=0,1$, $S_i=\{\phi_{1-i}(p')\}$; AND
\item If $\phi_0(p)\neq\phi_1(p)$ and $S_0=S_1$, then $|E(G-p')|$ is odd and $(\phi_0(p), \phi_0(p'))=(\phi_1(p'), \phi_1(p))$
\end{enumerate}
\item Let $\mathcal{F}$ be a family of $L$-colorings of $pp'$ and let $q\in\{p, p'\}$. Suppose that $|\mathcal{F}|\geq 3$ and $\mathcal{F}$ is constant on $q$. Then there exists a $\phi\in\mathcal{F}$ such that $|\Lambda_G(\phi(p), \phi(p'), \bullet)|\geq 2$ and in particular, if $G$ is not a triangle, then $L(p'')\setminus\{\phi(p')\}\subseteq\Lambda_G(\phi(p), \phi(p'), \bullet)$.
\item If $|V(G)|>4$ and there is an $a\in L(p)$ with $L(u_1)\setminus\{a\}\not\subseteq L(u_2)$, then $a$ is $G$-universal. 
\item If $|V(G)|\geq 4$, then, letting $x$ be the unique vertex of distance two from $p$ on the path $G-p'$, the following holds: For any $a\in L(p)$ with $L(u_1)\setminus\{a\}\not\subseteq L(x)$, $a$ is almost $G$-universal. 
\end{enumerate}
 \end{theorem}

\begin{lemma}\label{PartialPathColoringExtCL0}
Let $(G, C, P, L)$ be a rainbow and let $\phi$ be a partial $L$-coloring of $V(P)$ which includes the endpoints of $P$ in its domain and does not extend to an $L$-coloring of $G$. Then at least one of the following holds.
\begin{enumerate}[label=\arabic*)]
\itemsep-0.1em
\item There is a chord of $C$ with one endpoint in $\textnormal{dom}(\phi)$ and the other endpoint in $C\setminus P$; OR
\item There is a $v\in V(G\setminus C)\cup (V(\mathring{P})\setminus\textnormal{dom}(\phi))$ with $|L_{\phi}(v)|\leq 2$.
\end{enumerate} 
\end{lemma}

Our third black box from Paper I is the following result and its corollary.  

\begin{theorem}\label{EitherBWheelOrAtMostOneColThm} Let $(G, C, P, L)$ be a rainbow, where $P=p_1p_2p_3$ is a 2-path. Suppose further that $G$ is short-separation-free and every chord of $C$ has $p_2$ as an endpoint. Then,
\begin{enumerate}[label=\arabic*)]
\itemsep-0.1em
\item either $G$ is a broken wheel with principal path $P$, or there is at most one $L$-coloring of $V(P)$ which does not extend to an $L$-coloring of $G$; AND
\item If $\phi$ is an $L$-coloring of $V(P)$ which does not extend to an $L$-coloring of $G$, then, for each $p\in\{p_1, p_3\}$, letting $u$ be the unique neighbor of $p$ on the path $C-p_2$, we have $L(u)\setminus\{\phi(p)\}|=2$.
\end{enumerate}
\end{theorem}

\begin{cor}\label{CorMainEitherBWheelAtM1ColCor} Let $(G, C, P, L)$ be a rainbow, where $P:=p_1p_2p_3$ is a 2-path. Suppose further that $G$ is short-separation-free and every chord of $C$ has $p_2$ as an endpoint. Then, 
\begin{enumerate}[label=\arabic*)]
\itemsep-0.1em
\item if $|V(C)|>4$ and $|L(p_3)|\geq 2$, then, letting $xyp_3$ be the unique 2-path of $C-p_2$ with endpoint $p_3$, either there is a $G$-universal $a\in L(p_3)$ or $G$ is a broken wheel with principal path $P$ such that $L(p_3)\subseteq L(x)\cap L(y)$; AND
\item  If $|L(p_3)|\geq 3$ and either $|V(C)|>3$ or $G=C$, then
\begin{enumerate}[label=\roman*)]
\itemsep-0.1em
\item Either $G$ is a broken wheel with principal path $P$, or $|\Lambda_G^P(\phi(p_1), \phi(p_2), \bullet)|>1$ for any $L$-coloring $\phi$ of $p_1p_2$; AND
\item For each $a\in L(p_1)$, there are at most two $b\in L(p_2)\setminus\{a\}$ such that $|\Lambda_G^P(a, b, \bullet)|=1$; AND
\end{enumerate}
\item If $\phi$ is an $L$-coloring of $\{p_1, p_3\}$ and $S$ is a subset of $L_{\phi}(p_2)$ with $|S|\geq 2$ and $S\cap\Lambda_G^P(\phi(p_1), \bullet, \phi(p_3))=\varnothing$, then
\begin{enumerate}[label=\roman*)]
\itemsep-0.1em
\item $|S|=2$ and $G$ is a broken wheel with principal path $P$, and $G-p_2$ is a path of odd length; AND
\item For any $L$-coloring $\psi$ of $V(P)$ which does not extend to an $L$-coloring of $G$, either $\psi(p_1)=\psi(p_3)=s$ for some $s\in S$, or $\psi$ restricts to the same coloring of $\{p_1, p_3\}$ as $\phi$. 
\end{enumerate}
\end{enumerate} \end{cor}

We now recall the following notation from Paper I. 

\begin{defn}\label{EndNotationColor} \emph{Let $G$ be a graph and let $L$ be a list-assignment for $G$. Let $P$ be a path in $G$ with $|V(P)|\geq 3$, let $H$ be a subgraph of $G$ and let $\{p, p'\}$ be the endpoints of $P$. We let $\textnormal{End}_L(H, P, G)$ be the set of $L$-colorings $\phi$ of $\{p, p'\}\cup V(H)$ such that $\phi$ is $(P,G)$-sufficient.} \end{defn}

We usually drop the subscript $L$ in the case where it is clear from the context. Furthemore, if $H=\varnothing$, then we just write $\textnormal{End}_L(P,G)$.That is, $\textnormal{End}_L(P, G)$ is the set of $L$-colorings $\phi$ of the endpoints of $P$ such that any extension of $\phi$ to an $L$-coloring of $V(P)$ extends to $L$-color all of $G$. Our fourth black box from Paper I is the following result, together with its corollary. 

\begin{theorem}\label{SumTo4For2PathColorEnds} Let $(G, C, P, L)$ be an end-linked rainbow, where $P$ is a 2-path. Then $\textnormal{End}_{L}(P,G)\neq\varnothing$.  \end{theorem}

\begin{cor}\label{GlueAugFromKHCor} Let $(G, C, P, L)$ be an end-linked rainbow and $pq$ be a terminal edge of $P$, where $p$ is an endpoint of $P$, and let $p'$ be the other endpoint of $P$. Let $qx$ be an edge with $x\not\in V(C-p')$. Let $H, K$ be subgraphs of $G$ bounded by respective outer faces $x(C\setminus\mathring{P})p'Pq$ and $p(C\setminus\mathring{P})xq$. Let $\mathcal{F}$ be a nonempty family of partial $L$-colorings of $H$, where each element of $\mathcal{F}$ has $x$ in its domain. Suppose further that either $x=p$ or $|\textnormal{Col}(\mathcal{F}\mid x)|\geq |L(p')|$. Then there is an $L$-coloring $\phi$ of $\{p, x\}$ and a $\psi\in\mathcal{F}$ such that $\phi(x)=\psi(x)$ and $\phi$ is $(q, K)$-sufficient. \end{cor}

Our fifth and sixth black boxes from Paper I are the following results. 

\begin{theorem}\label{ThmFirstLink3PathForUseInHolepunch} Let $(G, C, P, L)$ be an end-linked rainbow, where $P:=p_1p_2p_3p_4$ is a 3-path. Then
\begin{enumerate}[label=\arabic*)] 
\itemsep-0.1em
\item $\textnormal{Crown}_L(P, G)\neq\varnothing$. Actually, something stronger holds. There is a subgraph $H$ of $G$ with $V(H)\subseteq V(C\setminus P)\cap N(p_1)\cup N(p_4))$, where $|V(H)\cap N(p)|\leq 1$ for each endpoint $p$ of $P$ and $\textnormal{End}_L(H,P,G)\neq\varnothing$; AND
\item If there is no chord of $C$ incident to a vertex of $\mathring{P}$, then $\textnormal{End}_{L}(P, G)\neq\varnothing$. 
\end{enumerate}
 \end{theorem}

\begin{theorem}\label{3ChordVersionMainThm1} Let $(G, C, P, L)$ be an end-linked rainbow, where $P:=p_1p_2p_3p_4$ is a 3-path, and the following additional conditions are satisfied.
\begin{enumerate}[label=\arabic*)]
\itemsep-0.1em
\item $|L(p_1)|\geq 1$ and $|L(p_4)|\geq 3$; AND
\item $N(p_3)\cap V(C)=\{p_2, p_4\}$. 
\end{enumerate} 
Then $\textnormal{End}_{L}(P, G)\neq\varnothing$. \end{theorem}

\section{Black Boxes from Paper II}\label{PIITwoBlackBoxSec}

Paper II consists of two results, which we state below here as black boxes. Our first black box from Paper II is the following. 

\begin{theorem}\label{CornerColoringMainRes} Let $(G, C, P, L)$ be an end-linked rainbow, where $p_1p_2p_3p_4$ be a subpath of $C$ of length three. Suppose further that $|L(p_3)|\geq 5$. Then
\begin{enumerate}[label=\arabic*)] 
\item there is an $L$-coloring $\psi$ of $\{p_1, p_4\}$ which extends to $|L_{\psi}(p_3)|-2$ elements of $\textnormal{End}_{L}(p_3, P, G)$; AND
\item If $p_3$ is incident to no chord of $C$, then there is an $L$-coloring $\psi$ of $\{p_1, p_4\}$ which extends to $|L_{\psi}(p_3)|-1$ elements of $\textnormal{End}_{L}(p_3, P, G)$.
\end{enumerate}
\end{theorem}

Ou second and final black box from Paper II is the following.

\begin{lemma}\label{EndLinked4PathBoxLemmaState} Let $(G, C, P, L)$ be an end-linked rainbow, where $P=p_0q_0wq_1p_1$ is a path of length four and $G$ is a short-separation-free graph in which every face of $G$, except for $C$, is bounded by a triangle. Suppose that $|L(w)|\geq 5$ and no chord of $C$ is incident to $w$. Then, for each $i\in\{0,1\}$, there is a partial $L$-coloring $\phi$ of $V(C\setminus\mathring{P})$ such that
\begin{enumerate}[label=\arabic*)] 
\itemsep-0.1em
\item  $\{p_0, p_1\}\subseteq\textnormal{dom}(\phi)\subseteq N(q_0)\cup N(q_1)$, and furthermore $N(q_i)\cap\textnormal{dom}(\phi)\subseteq\{p_0, p_1\}$; AND
\item For any two (not necessarily distinct) extensions of $\phi$ to $L$-colorings $\psi, \psi'$ of $\textnormal{dom}(\phi)\cup\{q_0, q_1\}$, at least one of the following holds
\begin{enumerate}[label=\roman*)]
\itemsep-0.1em
\item For some $c\in L(w)$, each of $\psi, \psi'$ extends to an $L$-coloring of $G$ using $c$ on $w$; OR
\item There is at an element of $\{\psi, \psi'\}$ which extends to two $L$-colorings of $G$ using different colors on $w$; OR
\item There is an element of $\{\psi, \psi'\}$ which does not extend to an $L$-coloring of $G$. 
\end{enumerate}
\end{enumerate}
\end{lemma}

\section{The Proof of Theorem \ref{MainHolepunchPaperResulThm}}\label{4ChordAnalogueFor2ChSec}

In this section, we bring together the results stated above to prove the main result of this paper, Theorem \ref{MainHolepunchPaperResulThm}. We first note that, in the statement of Theorem \ref{MainHolepunchPaperResulThm}, we cannot drop the condition that $q_0, z, q_1$ have no common neighbor in $C\setminus P$. If we drop this condition, then we have the counterexample in Figure \ref{DropConditionHolepunchCounterFigM}. We show now that, with $G$ being the graph in Figure \ref{DropConditionHolepunchCounterFigM} we have $\textnormal{Crown}_L(P, G)=\varnothing$, where $L$ is the list-assignment indicated in red. Let $C$ be the outer cycle of $G$, let $P=C-u_1$, and let $\phi$ be a partial $L$-coloring of $V(C)\setminus\{q_0, q_1\}$, where $p_0, p_1, z\in\textnormal{dom}(\phi)$, and suppose toward a contradiction that $\phi\in\textnormal{Crown}_L(P, G)$.  If $u_1\in\textnormal{dom}(\phi)$, then, since $|L_{\phi}(q_0)|\geq 3$, and $\phi(u_1)\in\{b,c\}$, we have $\phi(z)=d$. But then, no matter which color of $L(p_1)$ we use on $p_1$, we have $|L_{\phi}(q_1)|=2$, contradicting our assumption that $\phi\in\textnormal{Crown}_L(P, G)$. Thus, $u_1\not\in\textnormal{dom}(\phi)$. Since $\{b,c\}\subseteq L(q_0)\cap L(q_1)$ and $q_0q_1\not\in E(G)$, it follows that, no matter what colors $\phi$ uses on $p_1, z$, there is an extension of $\phi$ to an $L$-coloring $\phi'$ of $V(P)$ such that both of $b,c$ appear among the colors used by $\phi'$ on the vertices of $P-p_0$, so $\phi'$ does not extend to an $L$-coloring of $G$. Thus, $\phi\not\in\textnormal{Crown}_{L}(P, G)$, a contradiction. 

\begin{center}
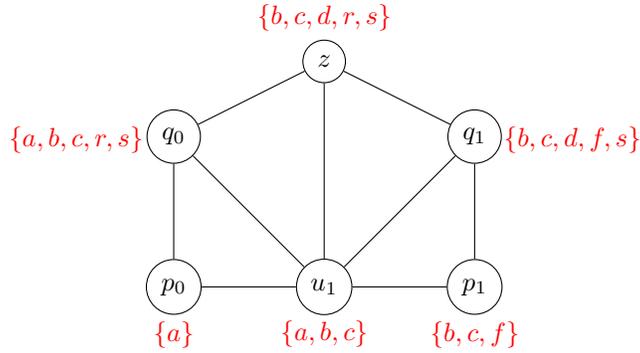
\begin{tikzpicture}
\node[shape=circle,draw=black] [label={[xshift=-0.0cm, yshift=-1.3cm]\textcolor{red}{$\{a\}$}}] (p0) at (0,0) {$p_0$};
\node[shape=circle,draw=black] [label={[xshift=-0.0cm, yshift=-1.3cm]\textcolor{red}{$\{a, b, c\}$}}] (u1) at (2, 0) {$u_1$};
\node[shape=circle,draw=black] [label={[xshift=-0.0cm, yshift=-1.3cm]\textcolor{red}{$\{b, c, f\}$}}] (p1) at (4, 0) {$p_1$};
\node[shape=circle,draw=black] [label={[xshift=-1.3cm, yshift=-0.7cm]\textcolor{red}{$\{a, b, c, r, s\}$}}] (q0) at (0,2) {$q_0$};
\node[shape=circle,draw=black] [label={[xshift=1.3cm, yshift=-0.7cm]\textcolor{red}{$\{b, c, d, f, s\}$}}] (q1) at (4,2) {$q_1$};
\node[shape=circle,draw=black] [label={[xshift=-0.0cm, yshift=0cm]\textcolor{red}{$\{b,c,d,r,s\}$}}] (z) at (2,3) {$z$};

 \draw[-] (p1) to (u1) to (p0) to (q0) to (z) to (q1) to (p1);
\draw[-] (z) to (u1);
\draw[] (q0) to (u1) to (q1);
\end{tikzpicture}\captionof{figure}{Theorem \ref{MainHolepunchPaperResulThm} is false if $q_0, q_1$ are allowed to have a common neighbor in $V(C\setminus\mathring{P})$}\label{DropConditionHolepunchCounterFigM}\end{center}

We now prove Theorem \ref{MainHolepunchPaperResulThm}, which makes up the remainder of this paper. 

\begin{thmn}[\ref{MainHolepunchPaperResulThm}] Let $(G, C, P, L)$ be an end-linked rainbow, where $P$ is a path of length four whose three internal vertices have no common neighbor in $C\setminus P$. Suppose further that each internal vertex of $P$ has an $L$-list of size at least five. Then $\textnormal{Crown}_L(P, G)\neq\varnothing$. \end{thmn} 

\begin{proof} Suppose not and let $G$ be a vertex-minimal counterexample to the claim. Let $C, P, L$ be as above, where $P:=p_0q_0zq_1p_1$. By adding edges to $G$ if necessary, we suppose further that every face of $G$, except possibly $C$, is bounded by a triangle. This is permissible since it is always possible to add edges until this holds without creating a common neighbor to $q_0, z, q_1$ in $V(C\setminus P)$. By removing some colors from some lists if necessary, we suppose that the inequalities of 1)-4) are equalities for each vertex other than $q_0, q_1$. That is, we suppose that $|L(p_0)|+|L(p_1)|=4$ and $|L(z)|=5$ and furthermore, for each $v\in V(C\setminus P)$, $|L(v)|=3$, and, for each $v\in V(G\setminus C)$, $|L(v)|=5$. We now introduce the following notation, which we use all throughout the remainder of the proof of Theorem \ref{MainHolepunchPaperResulThm}.

\begin{defn}\label{Defnu1.utKiTi}
\textcolor{white}{aaaaaaaaaaaaaaaaaa}
\begin{enumerate}[label=\emph{\arabic*)}] 
\itemsep-0.1em
\item \emph{Let $C\setminus\mathring{P}:=p_0u_1\ldots u_tp_1$ for some $t\geq 1$.}
\item\emph{For each $i\in\{0,1\}$, we define the following.}
\begin{enumerate}[label=\emph{\alph*)}] 
\itemsep-0.1em
\item\emph{We let $v_i$ be the unique vertex of $N(q_i)\cap V(C\setminus\mathring{P})$ which is farthest from $p_i$ on the path $C\setminus\mathring{P}$.}
\item\emph{We let $K_i$ be the subgraph of $G$ bounded by outer face $v_i\ldots p_iq_i$. That is, if $v_i=p_i$, then $K_i$ is an edge.}
\item\emph{We let $Q_0$ be the path $p_0q_0v_0$ and $Q_1$ be the path $v_1q_1p_1$. For each $i=0,1$, if $v_i=p_i$, then $K_i=Q_i$, and $Q_i$ is an edge.}
\item\emph{We say that $p_i$ is \emph{predictable} if, letting $y$ be the unique neighbor of $p_i$ on the path $C\setminus\mathring{P}$, $L(p_i)\subseteq L(y)$.}
\end{enumerate}
\end{enumerate} \end{defn}

Note that we have specified opposing orderings of the paths $Q_0$ and $Q_1$. That is, to avoid clutter, whenever we deal with a case in which $K_0$ is not just an edge, and we write $\Lambda_{K_0}(\cdot, \cdot, \cdot)$, where each $\cdot$ is a placeholder for either a color or the symbol $\bullet$, we are suppressing the superscript $p_0q_0v_0$ corresponding to the first, second, and third entries $\Lambda_{K_0}(\cdot, \cdot, \cdot)$ in that order. Likewise, whenever we deal with a case in which $K_1$ is not just an edge, and we write $\Lambda_{K_1}(\cdot, \cdot, \cdot)$, we are suppressing the superscript $v_1q_1p_1$ corresponding to the first, second, and third entries of $\Lambda_{K_1}(\cdot, \cdot, \cdot)$ in that order. In other words, for $i=0$, $p_i$ corresponds to the leftmost entry $\Lambda_{K_i}(\cdot, \cdot, \cdot)$ and, for $i=1$, $p_i$ corresponds to the rightmost entry of $\Lambda_{K_i}(\cdot, \cdot, \cdot)$. This is consistent with our diagrammatic conventions for this paper in which the cycle $C$ is always drawn with the edge $p_0q_0$ to the left of the edge $p_1q_1$.  

We break the proof of Theorem \ref{MainHolepunchPaperResulThm} into fourteen subsections of Section \ref{4ChordAnalogueFor2ChSec}, which are organized as follows.

\begin{enumerate}[label=\arabic*)] 
\item In Subsection \ref{FactGatherOpeningSubSecHolepunchMS}, we gather a few preliminary facts about $G$. In particular, we show that $G$ has no chord of $P$, except possibly $q_0q_1$.
\item In Subsections \ref{RuleOutCommonNbrQ0Q1OnPathFromU1toUtSubHolepunch}-\ref{RuleOutEdgeQ0Q1}, we show that $q_0q_1\not\in E(G)$.
\item In Subsection \ref{PrelimNotatSetup}, we first introduce some additional notation which we use throughout the remainder of the proof of Theorem \ref{MainHolepunchPaperResulThm} and then, in Subsection \ref{SubSecZEndPointP0P1CommNbr}, we show that, for each $i\in\{0,1\}$, $p_i, z$ have no common neighbor in $G\setminus C$ . 
\item In Subsections \ref{CommNbrCChordInterm}-\ref{Part22ChordSubsecR}, we show that, for each $i\in\{0,1\}$ and any $w\in V(G\setminus C)$ which is adjacent to both of $z, q_i$, $z$ has no neighbors in $C\setminus\mathring{P}$, except possibly $v_i$.
\item In Subsections \ref{FirstPrelimForChordsZ}-\ref{MTriangleInterSectAtLeast2X0X1}, we rule out the possibility that there is a chord of $C$ incident to $z$. 
\item In Subsection \ref{RuleOutCommNbrToV0V1InG-C}, we show that $v_0, v_1$ have no common neighbor in $G\setminus C$.
\item In Subsection \ref{CompSubSecTh}, we complete the proof of Theorem \ref{MainHolepunchPaperResulThm}.
\end{enumerate}

Subsections \ref{CommNbrCChordInterm}-\ref{MTriangleInterSectAtLeast2X0X1} make up the majority of the work of the proof of Theorem \ref{MainHolepunchPaperResulThm}. These content of these seven subsections is the proofs of Claims \ref{2ChordIncidentZVI} and \ref{noChordzOrAtLeast2Rule1Out}. For the purposes of the final two subsections of the proof of Theorem \ref{MainHolepunchPaperResulThm}, Subsections \ref{CommNbrCChordInterm}-\ref{MTriangleInterSectAtLeast2X0X1} just consists of these two results as black boxes. Everything else in Subsections \ref{CommNbrCChordInterm}-\ref{MTriangleInterSectAtLeast2X0X1} is an intermediate result which works towards the proofs of Claims \ref{2ChordIncidentZVI} and \ref{noChordzOrAtLeast2Rule1Out}. 

\makeatletter
\renewcommand{\thetheorem}{\thesubsection.\arabic{theorem}}
\@addtoreset{theorem}{subsection}
\makeatother

\subsection{Preliminary restrictions}\label{FactGatherOpeningSubSecHolepunchMS}

 Applying Theorem \ref{thomassen5ChooseThm} and Corollary \ref{CycleLen4CorToThom}, we immediately have the following by the minimality of $|V(G)|$. 

\begin{Claim}\label{MidPathPChordAllIst4Chord} $G$ is short-separation-free, and furthermore, every chord of $C$ is incident to one of $\{q_0, z, q_1\}$. \end{Claim}

Now we have the following. 

\begin{Claim}\label{NoP1P4P2P5inEGCL} There are no induced 4-cycles in $G$, and furthermore, $G$ contains neither of the edges $p_0q_1, p_1q_0$. \end{Claim}

\begin{claimproof} Since $G$ is short-separation-free, it follows from our triangulation conditions that there are no induced 4-cycles in $G$. Suppose toward a contradiction that $p_iq_{1-i}\in E(G)$ for some $i\in\{0,1\}$, say $i=0$. Let $G^*=G\setminus\{q_0, z\}$. Since $G$ is short-separation-free, $G^*$ is bounded by outer cycle $(C\setminus\{q_0, z\})+p_0q_1$. By Theorem \ref{SumTo4For2PathColorEnds}, there is an $L$-coloring $\psi$ of $\{p_0, p_1\}$ which is $(q_1, G^*)$-sufficient. Since $p_0q_1$ separates $q_0$ from $p_1$, we have $|L_{\psi}(q_0)|\geq 4$ and $|L_{\psi}(q_1)|\geq 3$. Possibly $p_0z\in E(G)$, but since $p_0q_1$ separates $z$ from $p_1$, we have $|L_{\psi}(z)|\geq 4$, so $\psi$ extends to an $L$-coloring $\psi'$ of $\{p_0, z, p_1\}$ such that each of $q_0, q_1$ has an $L_{\psi'}$-list of size at least three. Furthermore, $\psi'$ is $(P, G)$-sufficient, contradicting our assumption that $G$ is a counterexample. \end{claimproof}

\begin{Claim}\label{RuleOutEdgesp1p3p3p5CL} $G$ contains neither of the edges $p_0z, zp_1$. \end{Claim}

\begin{claimproof} Suppose toward a contradiction that $G$ contains one of the edges $p_0z, p_1z$, and suppose without loss of generality that $p_0z\in E(G)$. Let $G^*:=G-q_0$. Since $G$ is short-separation-free, $G^*$ is bounded by outer cycle $(C-q_0)+p_0z$. For any $L$-coloring $\psi$ of $\{p_0, p_1\}$, we have $|L_{\psi}(z)|\geq 3$. Since $|L(p_0)|+|L(p_1)|=4$ and each of $L(p_0), L(p_1)$ is nonempty, it follows from Theorem \ref{CornerColoringMainRes} that there is a $\psi\in\textnormal{End}_{L}(z, p_0zq_0p_1, G^*)$. Since $N(q_0)=\{p_0, z\}$, we have $|L_{\psi}(q_0)|\geq 3$. By Claim \ref{NoP1P4P2P5inEGCL}, $p_0q_1\not\in E(G)$. Since $\textnormal{dom}(\psi)=\{p_0, z, p_1\}$, we have $|L_{\psi}(q_1)|\geq 3$ as well. As $N(q_0)=\{p_0, z\}$ and $\psi\in\textnormal{End}_{L}(z, p_0zq_0p_1, G^*)$, it follows that $\psi$ is $(P, G)$-sufficient, contradicting our assumption that $G$ is a counterexample. \end{claimproof}

\begin{Claim}\label{VCoutFaceLenGr5} $V(C)\neq V(P)$. In particular, $t\geq 1$, and $G$ contains no chord of $P$, except possibly $q_0q_1$. \end{Claim}

\begin{claimproof} Suppose toward a contradiction that $V(C)=V(P)$. Thus, $C$ is a cycle of length five. Since $G$ is a counterexample, there is at least one $L$-coloring of $V(P)$ which does not extend to an $L$-coloring of $G$, and it follows from Theorem \ref{BohmePaper5CycleCorList} that there exists a $w\in V(G\setminus C)$ such that $V(G)=\{w\}\cup V(C)$, where $w$ is adjacent to all five vertices of $P$. Since $|L(z)|=5$ and $G$ contains neither of the edges $p_0z, p_1z$, there is an $L$-coloring $\psi$ of $\{p_0, z, p_1\}$ such that $|L_{\psi}(w)|\geq 3$, so $\psi$ is $(P,G)$-sufficient, contradicting our assumption that $G$ is a counterexample. We conclude that $V(C)\neq V(P)$. It follows from Claim \ref{NoP1P4P2P5inEGCL} that $p_0p_1\not\in E(G)$. Combining this with Claim \ref{RuleOutEdgesp1p3p3p5CL}, it follows that $G$ contains no chord of $P$, except possibly $q_0q_1$. \end{claimproof}

\begin{Claim}\label{AnyLColMidEndExtPCL11} For any colors $a_0\in L(p_0)$, $b\in L(z)$, and $a_1\in L(p_1)$, not necessarily distinct, there exists an $L$-coloring $\phi$ of $V(P)$ which does not extend to an $L$-coloring of $G$, where $\phi$ uses $a_0, b, a_1$ on the respective vertices $p_0, z, p_1$. \end{Claim}

\begin{claimproof} By Claim \ref{VCoutFaceLenGr5}, $V(C)\neq V(P)$, and since every chord of $C$ is incident to one of $q_0, z, q_1$, we have $p_0p_1\not\in E(G)$. By Claim \ref{RuleOutEdgesp1p3p3p5CL}, $z$ is not adjacent to either of $p_0, p_1$, so $\{p_0, z, p_1\}$ is an independent subset of $V(G)$. Thus, there is an $L$-coloring $\psi$ of $\{p_0, z, p_1\}$ using $a_0, b, a_1$ on the respective vertices $p_0, z, p_1$. By Claim \ref{NoP1P4P2P5inEGCL}, $p_0q_1, p_1q_0\not\in E(G)$, so each of $q_0, q_1$ has an $L_{\psi}$-list of size at least three. As $G$ is a counterexample, $\psi$ extends to an $L$-coloring $\phi$ of $V(P)$ which does not extend to an $L$-coloring of $G$.  \end{claimproof}

\begin{Claim}\label{KiBWheelSizeG} Let $i\in\{0,1\}$ and let $x$ be the unique vertex which is of distance two from $p_i$ on of the path $C\setminus\mathring{P}$. Then at least one of the following holds.
\begin{enumerate}[label=\arabic*)] 
\itemsep-0.1em
\item $|V(K_i)|\leq 3$; OR
\item $K_i$ is not a broken wheen with principal path $Q_i$; OR
\item $L(p_i)\not\subseteq L(x)$
\end{enumerate}
\end{Claim}

\begin{claimproof} By Claim \ref{VCoutFaceLenGr5}, $C\setminus\mathring{P}$ has length at least two, so $x$ is well-defined. Suppose without loss of generality that $i=0$ and that $K_0$ is a broken wheel with principal path $p_0q_0v_0$, where $|V(K_0)|>3$. We show that $L(p_0)\not\subseteq L(x)$. By assumption, $K_0-q_0$ is a path of length at least two, and since $p_0q_1\not\in E(G)$, we have $t\geq 2$ and $x=u_2$. Suppose toward a contradiction that $L(p_0)\subseteq L(u_2)$. Let $G^*=G\setminus\{p_0, u_1\}$. Since $K_0$ is a broken wheel with principal path $p_0q_0v_0$, $G^*$ is bounded by outer cycle $C^{G^*}=q_0zq_1p_1(u_t\ldots u_2)$, and $C^{G^*}$ contains the 4-path $R:=u_2q_0zq_1p_1$. Furthermore, since $q_0, z, q_1$ have no common neighbor $C\setminus P$, it follows that they have no common neighbor on $C^{G^*}\setminus R$. Since $|L(p_0)|+|L(p_1)|=4$ and $L(p_0)\subseteq L(u_2)$, we have $|L(p_0)\cap L(u_2)|+|L(p_1)|=4$. By the minimality of $G$, there is a partial $L$-coloring $\psi$ of $V(C^{G^*})\setminus\{q_0, q_1\}$ such that the following hold. 
\begin{enumerate}[label=\roman*)] 
\itemsep-0.1em
\item $\psi(u_2)\in L(p_0)\cap L(u_2)$ and $u_2, z, p_1\in\textnormal{dom}(\psi)$; \emph{AND}
\item Each of $q_0, q_1$ has an $L_{\psi}$-list of size at least three, and $\psi$ is $(P, G^*)$-sufficient. 
\end{enumerate}

Since $\psi(u_2)\in L(p_0)$, there is an extension of $\psi$ to an $L$-coloring $\psi'$ of $\textnormal{dom}(\psi)\cup\{p_0\}$, where $\psi'$ uses the same color on $p_0, u_2$. Since $|L_{\psi}(q_0)|\geq 3$ and the color $\psi(u_2)$ is already used by $\psi$ on a neighbor of $q_0$, we have $|L_{\psi'}(q_0)|\geq 3$. Since $|L_{\psi}(q_1)|\geq 3$ and $p_0q_1\not\in E(G)$, we have $|L_{\psi'}(q_1)|\geq 3$. Furthermore, $V(C^{G^*})\subseteq V(C)$, so $\textnormal{dom}(\psi')\subseteq V(C)\setminus\{q_0, q_1\}$, and $p_0, p_1, z\in\textnormal{dom}(\psi')$. Since $G$ is a counterexample, there is an extension of $\psi'$ to an $L$-coloring $\phi$ of $\textnormal{dom}(\psi')\cup\{q_0, q_1\}$, where $\phi$ not extend to an $L$-coloring of $G$. But since $N(p_0)=\{u_1, q_0\}$, $\phi$ extends to an $L$-coloring of $V(G^*)\cup\{p_0\}$ in which $p_0, u_2$ are colored with the same color, so there is a color left over for $u_1$. Thus, $\phi$ extends to an $L$-coloring of $G$, a contradiction.  \end{claimproof}

Recalling Definition \ref{GUniversalDefinition}, Claim \ref{KiBWheelSizeG} above has the following immediate consequence.

\begin{Claim}\label{KiSizeAtMost4UnivColorBlock} Let $i\in\{0,1\}$ with $|L(p_i)|\geq 2$ and $|V(K_i)|>2$. Then both of the following hold.
\begin{enumerate}[label=\arabic*)] 
\itemsep-0.1em
\item If there is no $K_i$-universal color of $L(p_i)$, then $K_i$ is a broken wheel with principal path $Q_i$, where $|V(K_i)|\leq 4$ and $p_i$ is predictable; AND
\item If there is no almost $K_i$-universal color of $L(p_i)$, then $K_i$ is a triangle and $p_i$ is predictable.
\end{enumerate}
 \end{Claim}

\begin{claimproof} Suppose without loss of generality that $i=1$ and that that there is no  $K_1$-universal color of $L(p_1)$. We first show that $K_1$ is a broken wheel with principal path $v_1q_1p_1$. Suppose not. Since $G$ is short-separation-free, $K_1$ is not a triangle, so we have $v_1\neq u_t$. If $v_1=u_{t-1}$, then, since $G$ is short-separation-free and has no chords of $C$ incident to $p_1$, we have $q_1u_t\in E(G)$ and $K_1$ is a broken wheel with principal path $v_1q_1p_1$, contradicting our assumption. Thus, the outer face of $K_1$ has at least five vertices, and, by 1) of Corollary \ref{CorMainEitherBWheelAtM1ColCor}, there is a $K_1$-universal color of $L(p_1)$, contradicting our assumption. Thus, $K_1$ is a broken wheel with principal path $v_1q_1p_1$. Since there is no $K_1$-universal color of $L(p_1)$, it immediately follows that $L(p_1)\subseteq L(u_t)$, so $p_1$ is predictable. We just need to show that $|V(K_1)|\leq 4$. Suppose not. Thus, by Claim \ref{KiBWheelSizeG}, $L(p_1)\not\subseteq L(u_{t-1})$. Since $K_1$ is outerplanar, the outer face of $K_1$ has at least five vertices. Since $|L(p_1)|\geq 2$ by assumption, it follows from 1) of Corollary \ref{CorMainEitherBWheelAtM1ColCor} that there is a $K_1$-universal color of $L(p_1)$, contradicting our assumption. This proves 1) of Claim \ref{KiSizeAtMost4UnivColorBlock}. Now we prove 2). Suppose that there is no almost $K_1$-universal color of $L(p_1)$. Thus, there is no $K_1$-universal color of $L(p_1)$, so, by 1), $K_1$ is a broken wheel with principal path $v_1q_1p_1$ and furthermore, $L(p_1)\subseteq L(u_t)$ and $|V(K_1)|\leq 4$. We just need to show that $|V(K_1)|\neq 4$. Suppose that $|V(K_1)|=4$, so $v_1=u_{t-1}$. By Claim \ref{KiBWheelSizeG}, $L(p_1)\not\subseteq L(u_{t-1})$. Since $|L(p_1)|\geq 2$ and $L(p_1)\subseteq L(u_t)$, there is an $a\in L(p_1)$ such that $L(u_t)\setminus\{a\}\not\subseteq L(u_{t-1})$, contradicting 4) of Theorem \ref{BWheelMainRevListThm2}. \end{claimproof}

\subsection{Ruling out a common neighbor to $q_0, q_1$ in $C-z$}\label{RuleOutCommonNbrQ0Q1OnPathFromU1toUtSubHolepunch}

In Subsection \ref{RuleOutEdgeQ0Q1}, we show that $q_0q_1\not\in E(G)$. To prove this, we first prove in this subsection the intermediate result that $q_0, q_1$ have no common neigbor in $C-z$. We break up the proof of this into two claims.

\begin{Claim}\label{Q0Q1CmNbrK0K1Triangle} If $q_0, q_1$ have a common neighbor $u$ in $C-z$, then the following hold.
\begin{enumerate}[label=\Alph*)] 
\itemsep-0.1em
\item $u\in V(C\setminus P)$ and $u=v_0=v_1$. Furthermore, $q_0q_1\in E(G)$ and $G-z=(K_0+K_1)$; AND
\item For each $i\in\{0,1\}$ with $|L(p_i)|\geq 2$, $K_i$ is a triangle and $L(p_i)\subseteq L(u)$. 
\end{enumerate}
\end{Claim}

\begin{claimproof} By Claim \ref{NoP1P4P2P5inEGCL}, $u$ is not an endpoint of $P$, so $u\in\{u_1, \ldots, u_t\}$, and, by the conditions of Theorem \ref{MainHolepunchPaperResulThm}, $uz\not\in E(G)$. Since there are no induced 4-cycles in $G$, we have $q_0q_1\in E(G)$. In particular, in the notation of Definition \ref{Defnu1.utKiTi}, we have $u=v_0=v_1$ and, since $G$ is short-separation-free, $N(z)=\{q_0, q_1\}$ and $G-z=(K_0\cup K_1)+q_0q_1$. This proves A) of Claim \ref{Q0Q1CmNbrK0K1Triangle}. Before proving B), we prove the following two intermediate results. 

\vspace*{-8mm}
\begin{addmargin}[2em]{0em}
\begin{subclaim}\label{IntermediateSubclaimKiList2Size4} For each $i\in\{0,1\}$ with $|L(p_i)|\geq 2$, $K_i$ is a broken wheel with principal path $Q_i$ and $|V(K_i)|\leq 4$, and furthermore, $p_i$ is predictable. \end{subclaim}

\begin{claimproof} Suppose toward a contradiction that there is an $i\in\{0,1\}$ and a $c\in L(p_i)$ violating Subclaim \ref{IntermediateSubclaimKiList2Size4}, say $i=1$ without loss of generality. By 1) of Claim \ref{KiSizeAtMost4UnivColorBlock}, there is a $K_1$-universal color $c$ of $L(p_1)$.  By Claim \ref{AnyLColMidEndExtPCL11}, since $|L(p_0)|\geq 1$, there is an $L$-coloring $\phi$ of $V(P)$ which does not extend to an $L$-coloring of $G$, where $\phi(p_1)=c$. By Theorem \ref{thomassen5ChooseThm}, there is an $r\in\Lambda_{K_0}(\phi(p_0), \phi(q_0), \bullet)$. Since $c$ is $K_1$-universal, $\Lambda_{K_1}(\bullet, \phi(q_1), c)=L(u)\setminus\{\phi(q_1)\}$, so $r\in \Lambda_{K_1}(\bullet, \phi(q_1), c)$. Since  $N(z)=\{q_0, q_1\}$ and $G-z=(K_0\cup K_1)+q_0q_1$, it follows that $\phi$ extends to an $L$-coloring of $G$, a contradiction. \end{claimproof}\end{addmargin}

\vspace*{-8mm}
\begin{addmargin}[2em]{0em}
\begin{subclaim}\label{EachAugColorP0Q0MissP1} Let $i\in\{0,1\}$ with $|V(K_{1-i})|=4$. Then, for each $\phi\in\textnormal{End}(Q_i, K_i)$, $\phi(u)\not\in L(p_{1-i})$. \end{subclaim}

\begin{claimproof} Suppose without loss of generality that $i=0$ and that $|V(K_1)|=4$. Thus, $u=u_{t-1}$. Since no chord of $C$ has $p_1$ as an endpoint, and $G$ contains no induced 4-cycles, $K_1$ is a broken wheel with principal path $uq_1p_1$. Let $\phi\in\textnormal{End}(p_0q_0u, K_0)$. Suppose toward a contradiction that $\phi(u)\in L(p_1)$. Since $|L_{\phi}(q_0)|\geq 3$, there is a $c\in L(z)$ with $|L_{\phi}(q_0)\setminus\{c\}|\geq 3$. Since $u=u_{t-1}$ and $N(z)=\{q_0, q_1\}$, there is an extension of $\phi$ to an $L$-coloring $\phi'$ of $\{p_0, u, z, p_1\}$, where $\phi'(z)=c$ and $\phi'$ uses the same color on $u, p_1$. By our choice of $c$, we have $|L_{\phi'}(q_0)|\geq 3$. Since $\phi'$ uses the same color on $u$ and $p_1$, we have $|L_{\phi'}(q_1)|\geq 3$ as well, and $\phi'$ also leaves two colors in $u_t$, i.e $|L_{\phi'}(u_t)|\geq 2$. Since $\phi\in\textnormal{End}(Q_0, K_0)$, it follows that $\phi'$ is $(P,G)$-sufficient, contradicting our assumption that $G$ is a counterexample. \end{claimproof}\end{addmargin}

Now we prove B). Suppose toward a contradiction that there is an $i\in\{0,1\}$ with $|L(p_i)|\geq 2$, where $i$ violates B), say $i=1$ without loss of generality. By Subclaim \ref{IntermediateSubclaimKiList2Size4}, we have $L(p_1)\subseteq L(u_t)$, and $K_1$ is a broken wheel with principal path $Q_1$, where $|V(K_1)|\leq 4$. Since B) is violated, we have $u\neq u_t$, so $|V(K_1)|=4$ and $u=u_{t-1}$. In particular, $K_1-q_1$ is 2-path $u_{t-1}u_tp_1$, and $up_1\not\in E(G)$. Furthermore, by Claim \ref{KiBWheelSizeG}, $L(p_1)\not\subseteq L(u_{t-1})$.  

\vspace*{-8mm}
\begin{addmargin}[2em]{0em}
\begin{subclaim}\label{EachP0P12ListSubCLQ0Q1Adj} $|L(p_0)|=|L(p_1)|=2$. \end{subclaim} 

\begin{claimproof} Since $|L(p_1)|\geq 2$ by assumption, we have $1\leq|L(p_0)|\leq 2$. Suppose that Subclaim \ref{EachP0P12ListSubCLQ0Q1Adj} does not hold. Thus, $|L(p_0)|=1$ and $|L(p_1)|=3$. Since $L(p_1)\subseteq L(u_t)$, we have $L(p_1)=L(u_t)$. By Theorem \ref{SumTo4For2PathColorEnds}, there is a $\phi\in\textnormal{End}(p_0q_0u, K_0)$. By Subclaim \ref{EachAugColorP0Q0MissP1}, $\phi(u_{t-1})\not\in L(p_1)$, since $u=u_{t-1}$. Thus, $\phi(u_{t-1})\not\in L(u_t)$.  Now, since $|L(z)|=5$ and $|L_{\phi}(q_1)|\geq 3$, we have $|\{c\in L(z): |L_{\phi}(z)\setminus\{c\}|\geq 3\}|\geq 2$. In particular, we have $|L(p_1)|+|\{c\in L(z): |L_{\phi}(q_0)\setminus\{c\}|\geq 3\}|\geq 5$. Since $|L_{\phi}(q_1)|\geq 4$, there is a $c\in L(z)$ and a $d\in L(p_1)$ with $|L_{\phi}(q_0)\setminus\{c\}|\geq 3$ and $|L_{\phi}(q_1)\setminus\{c, d\}|\geq 3$ as well. Since $N(z)=\{q_0, q_1\}$, $\phi$ extends to an $L$-coloring $\phi'$ of $\{p_0, z, u_{t-1}, p_1\}$, where $\phi'$ uses the colors $c,d$ on the respective vertices $z, p_1$. Thus, each of $q_0, q_1$ has an $L_{\phi'}$-list of size at least three. But since $\phi\in\textnormal{End}(p_0q_0u, K_0)$ and $\phi(u_{t-1})\not\in L(u_t)$, it follows that $\phi'$ is $(P,G)$-sufficient, contradicting our assumption that $G$ is a counterexample. \end{claimproof}\end{addmargin}

Since $|L(p_0)|=2$, it follows from Subclaim \ref{IntermediateSubclaimKiList2Size4} that $L(p_0)\subseteq L(u_1)$ and that $K_0$ is a broken wheel with principal path $p_0q_0u$, where $|V(K_0)|\leq 4$. In particular, $t\geq 2$. 

\vspace*{-8mm}
\begin{addmargin}[2em]{0em}
\begin{subclaim}\label{K0TriangleK1Size4GSize8} $|V(K_0)|=3$. \end{subclaim}

\begin{claimproof} Suppose not. Thus, $u=u_2$ and $t=3$. Since $|V(K_1)|=4$ as well and each of $K_0, K_1$ is a broken wheel, it follows from Claim \ref{KiBWheelSizeG} that $L(p_0)\not\subseteq L(u_2)$ and $L(p_1)\not\subseteq L(u_2)$ as well. As indicated above, we have $L(p_0)\subseteq L(u_1)$ and, since $t=3$, $L(p_1)\subseteq L(u_3)$. Since $|L(p_0)|=|L(p_1)|=2$, there exist colors $a,b$, where $a\in L(p_0)$ and $b\in L(p_1)$, and furthermore, $L(u_1)\setminus\{a\}\subseteq L(u_2)$ and $L(u_3)\setminus\{b\}\not\subseteq L(u_2)$. By Claim \ref{AnyLColMidEndExtPCL11}, there is an $L$-coloring $\phi$ of $V(P)$, where $\phi$ does not extend to an $L$-coloring of $G$ and $\phi$ uses $a,b$ on the respective vertices $p_0, p_1$. Since $\phi$ does not extend to an $L$-coloring of $G$, we have $\Lambda_{K_0}(a, \phi(q_0), \bullet)\cap\Lambda_{K_1}(\bullet, \phi(q_1), b)=\varnothing$. Since $|L(u)|=3$, at least one of the two sets in the intersection above has size less than two. Yet, by 4) of Theorem \ref{BWheelMainRevListThm2}, each of $\Lambda_{K_0}(a, \phi(q_0), \bullet)$ and $\Lambda_{K_1}(\bullet, \phi(q_1), b)$ has size at least two, a contradiction. \end{claimproof}\end{addmargin}

Applying Subclaim \ref{K0TriangleK1Size4GSize8}, we get $t=2$ and $u=u_1$, so $|V(G)|=7$ and we have the graph in Figure \ref{FigureReductionGQ0Q1Nbr}.

\vspace*{-8mm}
\begin{addmargin}[2em]{0em}
\begin{subclaim}\label{SubForEachQ0Q1FromUList} For each $c\in L(u_1)\setminus L(u_2)$ and $i\in\{0,1\}$, we have $|L(q_i)\setminus\{c\}|=4$. \end{subclaim}

\begin{claimproof} Suppose there is an $i\in\{0,1\}$ and a $c\in L(u_{t-1})$ with $|L(q_i)\setminus\{c\}|\geq 5$. Consider the following cases.

\textbf{Case 1:} $i=1$

Since $|L(p_0)\setminus\{c\}|\geq 1$ and $|L(z)|=5$, there is an $L$-coloring $\phi$ of $\{p_0, u_1, z, p_1\}$, where $|L_{\phi}(q_0)|\geq 3$ and $\phi(u_1)=c$. By our choice of $c$, we have $|L_{\phi}(q_0)|\geq 3$ as well. Since $c\not\in L(u_2)$, $\phi$ is $(P,G)$-sufficient, contradicting our assumption that $G$ is a counterexample to Theorem \ref{MainHolepunchPaperResulThm}.

\textbf{Case 2:} $i=0$

In this case, since $|L(z)|=5$, there is a $d\in L(z)$ with $|L(q_1)\setminus\{c, d\}|\geq 4$. Since $|L(p_0)\setminus\{c\}|\geq 1$, there is an $L$-coloring $\psi$ of of $\{p_0, u_1, z, p_1\}$, where $\psi$ uses $c, d$ on the respective vertices $u_1, z$. Thus, $|L_{\psi}(q_1)|\geq 3$, and, by our choice of $c$, we have $|L_{\psi}(q_0)|\geq 3$ as well. As above, since $c\not\in L(u_2)$, $\psi$ is $(P,G)$-sufficient, contradicting our assumption that $G$ is a counterexample.  \end{claimproof}\end{addmargin}

\begin{center}\begin{tikzpicture}
\node[shape=circle,draw=black] (p0) at (0,0) {$p_0$};
\node[shape=circle,draw=black] (u1) at (2, 0) {$u_1$};
\node[shape=circle,draw=black]  (u2) at (4, 0) {$u_2$};
\node[shape=circle,draw=black]  (p1) at (6, 0) {$p_1$};
\node[shape=circle,draw=black] (q0) at (0,2) {$q_0$};
\node[shape=circle,draw=black] (q1) at (6,2) {$q_1$};
\node[shape=circle,draw=black] (z) at (3,3) {$z$};
\node[shape=circle,draw=white] (K0) at (0.7,0.7) {$K_0$};

 \draw[-] (p0) to (q0) to (z) to (q1) to (p1) to (u2) to (u1) to (p0);
\draw[-] (q0) to (q1) to (u2);
\draw[-] (q0) to (u1);
\draw[-] (q1) to (u1);
\end{tikzpicture}\captionof{figure}{}\label{FigureReductionGQ0Q1Nbr}\end{center}

Since $K_0$ is a triangle, each $L$-coloring of the edge $p_0u_1$ lies in $\textnormal{End}(p_0q_0u, K_0)$, and in particular, since $|L(p_0)|=2$, it follows that, for each $c\in L(u_1)$, there is an element of $\textnormal{End}(p_0q_0u, K_0)$ using $c$ on $u$. Thus, by Subclaim \ref{EachAugColorP0Q0MissP1}, since $u=u_1$, we have $L(u)\cap L(p_1)=\varnothing$. We now let $S=L(u_1)\setminus L(u_2)$. Since $L(p_1)\subseteq L(u_2)$, we have $|L(u_1)\cap L(u_2)|\leq 1$, so $|S|\geq 2$. 

\vspace*{-8mm}
\begin{addmargin}[2em]{0em}
\begin{subclaim}\label{SAndL(z)DisjointSubCL} $S\cap L(z)=\varnothing$. \end{subclaim}

\begin{claimproof} Suppose toward a contradiction that there is an $f\in S\cap L(z)$ and let $\phi$ be an arbitrary $L$-coloring of $\{p_0, u_1, z, p_1\}$, where $\phi$ uses $f$ on both of respective vertices $u_1, z$. Such a $\phi$ exists, since $|L(p_0)\setminus\{f\}|\geq 1$, and each of $q_0, q_1$ has an $L_{\phi}$-list of size at least three. Since $f\not\in L(u_2)$, it follows that $\psi$ is $(P,G)$-sufficient, contradicting our assumption that $G$ is a counterexample. \end{claimproof}\end{addmargin}

By Subclaim \ref{SubForEachQ0Q1FromUList}, $S\subseteq L(q_0)\cap L(q_1)$ and $|L(q_0)|=|L(q_1)|=5$. By Subclaim \ref{SAndL(z)DisjointSubCL}, $S\cap L(z)=\varnothing$. Since $|S|\geq 2$, we get $|L(z)\setminus L(q_i)|\geq 2$ for each $i=0,1$. Since $|L(z)\setminus L(q_0)|\geq 2$, we have $|S|+|L(p_1)|+|L(z)\setminus L(q_0)|\geq 6$. Since $\{u_1, p_1, z\}$ is an independent subset of $V(G)$, there is an $L$-coloring $\phi$ of $\{u_1, p_1, z\}$, where $\phi(u_1)\in S$ and $\phi(z)\not\in L(q_0)$, and $|L_{\phi}(q_1)|\geq 3$. Let $\phi'$ be an extension of $\phi$ to an $L$-coloring of $\textnormal{dom}(\phi)\cup\{p_0\}$. As $\phi'(z)\not\in L(q_0)$, each of $q_0$ and $q_1$ has an $L_{\phi'}$-list of size at least three. Since $\phi'(u_1)\not\in L(u_2)$, it follows that $\phi'$ is $(P,G)$-sufficient, contradicting our assumption that $G$ is a counterexample. \end{claimproof}

With Claim \ref{Q0Q1CmNbrK0K1Triangle} in hand, we now complete Subsection \ref{RuleOutCommonNbrQ0Q1OnPathFromU1toUtSubHolepunch}. 

\begin{Claim}\label{NoCommNbrQ0Q1OnPathu1Tout} $q_0$ and $q_1$ have no common neighbor in $C-z$.   \end{Claim}

\begin{claimproof} Suppose toward a contradiction that $q_0, q_1$ have a common neighbor $u$ on $C-z$. Note that $u$ is unique. By A) of Claim \ref{Q0Q1CmNbrK0K1Triangle}, $u\in\{u_1, \ldots, u_t\}$ and $q_0q_1\in E(G)$, and furthermore, $G-z=(K_0\cup K_1)+q_0q_1$. In particular, $N(z)=\{q_0, q_1\}$, and $v_0=v_1=u$. 

\vspace*{-8mm}
\begin{addmargin}[2em]{0em}
\begin{subclaim}\label{ForEachI01AndAugBothEndColorinQi} For each $i\in\{0,1\}$ and $\phi\in\textnormal{End}(Q_i, K_i)$, if $|L(p_{1-i})|\geq 2$, then $L_{\phi}(q_i)|=3$, and, in particular, $|L(q_i)|=5$.  \end{subclaim}

\begin{claimproof} Suppose without loss of generality that $|L(p_1)|\geq 2$ and let $\phi\in\textnormal{End}(p_0q_0u, K_0)$. By B) of Claim \ref{Q0Q1CmNbrK0K1Triangle}, $K_1$ is a triangle, so $u=u_t$. Suppose toward a contradiction that $|L_{\phi}(q_i)|\neq 3$. Thus, $|L_{\phi}(q_0)|\geq 4$. Since $|L(p_1)|\geq 2$, there is a $c\in L_{\phi}(p_1)$. Note that $|L_{\phi}(q_1)|\geq 4$ as well, since $N(q_1)\cap\textnormal{dom}(\phi)=\{u\}$. Since $|L(z)|=5$, there is a $d\in L(z)$ with $|L_{\phi}(q_1)\setminus\{c,d\}|\geq 3$. Since $N(z)=\{q_0, q_1\}$, there is an extension of $\phi$ to an $L$-coloring $\phi'$ of $\{p_0, u, z, p_1\}$, where $\phi'$ uses $d, c$ on the respective vertices $z, p_1$. As $|L_{\phi}(q_0)|\geq 4$, each of $q_0, q_1$ has an $L_{\phi'}$-list of size at least three. Since $\phi\in\textnormal{End}(p_0q_0u, K_0)$,  it follows that $\phi'$ is $(P,G)$-sufficient, contradicting our assumption that $G$ is a counterexample.  \end{claimproof}\end{addmargin}

\vspace*{-8mm}
\begin{addmargin}[2em]{0em}
\begin{subclaim}\label{ExCl22CaseWhereEdgeQ0Q1SubSCl} For some $i\in\{0,1\}$, $|L(p_i)|\neq 2$. \end{subclaim}

\begin{claimproof} Suppose not. Thus, $|L(p_0)|=L(p_1)|=2$. It follows from B) of Claim \ref{Q0Q1CmNbrK0K1Triangle} that each of $K_0$ and $K_1$ is a triangle, and, in particular, $t=1$ and $V(G)=V(P)\cup\{u\}$. It also follows from B) of Claim \ref{Q0Q1CmNbrK0K1Triangle} that $L(p_0)\cup L(p_1)\subseteq L(u)$. For each $i=0,1$, since $K_i$ is a triangle, any $L$-coloring of $\{p_i, u\}$ is an element of $\textnormal{End}(Q_i, K_i)$, and furthermore, since $|L(p_i)|=2$ and $|L(u)|=3$, it follows that, for each $x\in\{p_i, u\}$ and $c\in L(x)$, there is an element of $\textnormal{End}(Q_i, K_i)$ using $c$ on $x$. Thus, by Subclaim \ref{ForEachI01AndAugBothEndColorinQi}, we have $L(u)\subseteq L(q_i)$ and $|L(q_i)|=5$ for each $i=0,1$. Now consider the following cases. 

\textbf{Case 1:} $L(u)\cap L(z)\neq\varnothing$

In this case, we let $a\in L(u)\cap L(z)$. Since $|L(p_0)|=|L(p_1)|=2$ and $z\not\in N(u)$, there is an $L$-coloring $\phi$ of $\{p_0, z, u, p_1\}$ using $a$ on each of $z$ and $u$. Each of $q_0$ and $q_1$ has an $L_{\phi}$-list of size at least three. Since $\phi$ is already an $L$-coloring of $V(G)\setminus\{q_0, q_1\}$, we contradict our assumption that $G$ is a counterexample. 

\textbf{Case 2:} $L(u)\cap L(z)=\varnothing$

Since $|L(q_0)|=|L(q_1)|=5$ and $L(u)\subseteq L(q_0)\cap L(q_1)$, we have $|L(z)\setminus L(q_i)|\geq 3$ for each $i=0,1$, as $|L(u)|=3$. Since $|L(z)|=5$, it follows that there is a $c\in L(z)$ with $c\not\in L(q_0)\cup L(q_1)$. Let $\phi$ be an arbitrary $L$-coloring of $\{p_0, z, u, p_1\}$ using $c$ on $z$. Each of $q_0$ and $q_1$ has an $L_{\phi}$-list of size at least three, and since $\phi$ is already an $L$-coloring of $V(G)\setminus\{q_0, q_1\}$, we contradict our assumption that $G$ is a counterexample. \end{claimproof}\end{addmargin}

Applying Subclaim \ref{ExCl22CaseWhereEdgeQ0Q1SubSCl}, we suppose without loss of generality that $|L(p_0)|=1$ and $|L(p_1)|=3$. By B) of Claim \ref{Q0Q1CmNbrK0K1Triangle}, $K_1$ is a triangle, and, since $|L(p_1)|=3$, we have $L(u)=L(p_1)$. Let $a$ be the lone color of $L(p_0)$. Let $L(p_1)=\{b_0, b_1, b_2\}$. Applying Claim \ref{AnyLColMidEndExtPCL11}, for each $k=0,1,2$, there is an $L$-coloring $\psi_k$ of $V(P)$ which does not extend to an $L$-coloring of $G$, where $\psi_k(p_1)=b_k$. Since $N(z)=\{q_0, q_1\}$, it follows that, for each $k=0,1,2$, $\Lambda_{K_0}(a, \psi_k(q_0), \bullet)=\{b_k\}$. Since $\{\psi(q_0): k=0,1,2\}|=3$, this contradicts 2ii) of Corollary \ref{CorMainEitherBWheelAtM1ColCor}. \end{claimproof}

\subsection{Ruling out the edge $q_0q_1$}\label{RuleOutEdgeQ0Q1}

With the intermediate result of Claim \ref{NoCommNbrQ0Q1OnPathu1Tout} in hand, we prove that $q_0$ and $q_1$ are not adjacent. 

\begin{Claim}\label{Q0AndQ1AreNotAdjacentCL}  $q_0q_1\not\in E(G)$.\end{Claim}

\begin{claimproof} Suppose toward a contradiction that $q_0q_1\in E(G)$. We first note the following

\vspace*{-8mm}
\begin{addmargin}[2em]{0em}
\begin{subclaim}\label{NeithK0K1EdgeSubV0V1} Neither $K_0$ nor $K_1$ is an edge. \end{subclaim}

\begin{claimproof} Suppose that at least one of $K_0, K_1$ is an edge, say $K_0$ without loss of generality. Let $G'$ be the subgraph of $G$ bounded by outer cycle $(p_0(C\setminus\mathring{P})v_1q_1q_0$. That is $G'=(G\setminus\{z\})\setminus (K_1\setminus\{q_1, p_1\})$. Note that the outer cycle of $G'$ contains the 3-path $P'=p_0q_0q_1v_1$. By Theorem \ref{ThmFirstLink3PathForUseInHolepunch}, there exist $|L(p_0)|$ different elements of $\textnormal{Crown}(P', G')$, each of which uses a different color on $v_1$. By Corollary \ref{GlueAugFromKHCor}, there exists a $\phi\in\textnormal{Crown}(P', G')$ and a $(q_1, K_1)$-sufficient $L$-coloring $\psi$ of $Q_1-q_1$ with $\phi(v_1)=\psi(v_1)$. Since $K_0$ is an edge, we have $N(q_0)\cap\textnormal{dom}(\phi\cup\psi)=\{p_0\}$, so $q_0$ has an $L_{\phi\cup\psi}$-list of size at least four. Since $q_0$ has an $L_{\phi\cup\psi}$-list of size at least three and $N(w)=\{q_0, q_1\}$, it follows that $\phi\cup\psi$ extends to an $L$-coloring $\sigma$ of $\textnormal{dom}(\phi\cup\psi)\cup\{q_0, q_1\}$, where each of $q_0, q_1$ has an $L_{\sigma}$-list of size at least three. It follows from our choice of $\phi$ and  $\psi$ that $\sigma$ is $(P, G)$-sufficient, contradicting our assumption that $G$ is a counterexample. \end{claimproof}\end{addmargin}

For each $i=0,1$, let $G^{\ast}_i$ be the subgraph of $G$ bounded by outer cycle $p_i(C\setminus\mathring{P})v_{1-i}q_{1-i}q_i$. Note that the outer cycle $G^{\ast}_i$ contains the 3-path $R^{\ast}_i=p_iq_iq_{1-i}v_{1-i}$, and no chord of the outer cycle of $G^{\ast}_i$ is incident to $q_{1-i}$. 

\vspace*{-8mm}
\begin{addmargin}[2em]{0em}
\begin{subclaim}\label{NoAugAugMatchRi} For each $i\in\{0,1\}$ and $\pi\in\textnormal{End}(R^{\ast}_i, G^{\ast}_i)$, there is no element of $\textnormal{End}(Q_{1-i}, K_{1-i})$ which uses $\pi(v_{1-i})$ on $v_{1-i}$. \end{subclaim}

\begin{claimproof} Suppose for the sake of definiteness that $i=0$ and let $\pi\in\textnormal{End}(R^{\ast}_0, G^{\ast}_0)$. By Subclaim \ref{NeithK0K1EdgeSubV0V1}, $Q_1$ is a 2-path. Suppose toward a contradiction that there is a $\phi\in\textnormal{End}(Q_1, K_1)$ with $\phi(v_1)=\pi(v_1)$. Thus, $\pi\cup\phi$ is a proper $L$-coloring of $\{p_0, v_1, p_1\}$. Since $|L_{\pi\cup\phi}(q_1)|\geq 3$, there is an $s\in L(z)$ with $|L_{\pi\cup\phi}(q_1)\setminus\{s\}|\geq 3$. Let $\sigma$ be an extension of $\pi\cup\phi$ to an $L$-coloring of $\{p_0, v_1, z, p_1\}$ using $s$ on $z$. By Claim \ref{NoCommNbrQ0Q1OnPathu1Tout}, $v_0\neq v_1$, so we have $v_1q_0\not\in E(G)$ and thus each of $q_0, q_1$ has an $L_{\sigma}$-list of size at least three.. Since $G$ is a counterexample, $\sigma$ extends to an $L$-coloring $\sigma'$ of $V(P)\cup\{v_1\}$ which does not extend to an $L$-coloring of $G$, yet it follows from our choice of $\phi, \pi$ that $\sigma'$ extends to an $L$-coloring of $G$, a contradiction.  \end{claimproof}\end{addmargin}

\vspace*{-8mm}
\begin{addmargin}[2em]{0em}
\begin{subclaim} $|L(p_0)|=|L(p_1)|=2$. \end{subclaim}

\begin{claimproof} Suppose not, and suppose without loss of generality that $|L(p_0)|=1$ and $|L(p_1)|=3$. Since $|L(v_1)|=3$ and no chord of the outer face of $G^{\ast}_0$ is incident to $q_1$, it follows from Theorem \ref{3ChordVersionMainThm1} that there is a $\pi\in\textnormal{End}(R^{\ast}_0, G^{\ast}_0)$. Since $|L(p_1)|=3$, it follows from Theorem \ref{SumTo4For2PathColorEnds} that there is a $\phi\in\textnormal{End}(Q_1, K_1)$ with $\phi(v_1)=\pi(v_1)$, contradicting Subclaim \ref{NoAugAugMatchRi}. \end{claimproof}\end{addmargin}

\vspace*{-8mm}
\begin{addmargin}[2em]{0em}
\begin{subclaim} For some $i\in\{0,1\}$, $K_i$ is a triangle.  \end{subclaim}

\begin{claimproof} Suppose not. By Subclaim \ref{NeithK0K1EdgeSubV0V1}, neither $K_0$ nor $K_1$ is an edge, and since $|L(p_0)|=|L(p_1)|=2$ it follows from 2) of Claim \ref{KiSizeAtMost4UnivColorBlock} that, for each $i\in\{0,1\}$, there is an $a_i\in L(p_i)$ which is almost $K_i$-universal. By Claim \ref{AnyLColMidEndExtPCL11}, there is an $L$-coloring $\sigma$ of $V(P)$ which does not extend to an $L$-coloring of $G$, where $\sigma$ uses $a_0, a_1$ on the respective vertices $p_0, p_1$. Since $\sigma(q_0)\not\in\Lambda_{K_0}(a_0, \sigma(q_0), \bullet)$ and $\sigma(q_1)\not\in\Lambda_{K_1}(\bullet, \sigma(q_1), a_1)$, it follows from Corollary \ref{2ListsNextToPrecEdgeCor} that the $L$-coloring $(\sigma(q_0), \sigma(q_1))$ of $q_0q_1$ extends to an $L$-coloring of  $G^{\ast}_0\cap G^{\ast}_1$ which uses a color of the respective sets $\Lambda_{K_0}(a_0, \sigma(q_0), \bullet)$ and $\Lambda_{K_1}(\bullet, \sigma(q_1), a_1)$ on the respective vertices $q_0, q_1$.  Thus, $\sigma$ extends to an $L$-coloring of $G$, a contradiction. \end{claimproof}\end{addmargin}

Since $|L(p_0)|=|L(p_1)|=2$, suppose without loss of generality that $K_1$ is a triangle, so $v_1=u_t$ and $R^{\ast}_0=p_0q_0q_1v_1$. By Theorem \ref{3ChordVersionMainThm1}, since $|L(u_t)|=3$, there is a $\pi\in\textnormal{End}(R^{\ast}_0, G^{\ast}_0)$. Since $|L(p_1)|=2$, there is a $c\in L(p_1)\setminus\{\pi(u_t)\}$. As $K_1$ is a triangle, there is an element of $\textnormal{End}(Q_1, K_1)$ using $\pi(u_t), c$ on the respective vertices $u_t, p_1$, contradicting Subclaim \ref{NoAugAugMatchRi}. \end{claimproof}

\subsection{Dealing with chords and 2-chords of $C$ incident to $z$: notation and setup}\label{PrelimNotatSetup}

As indicated in the introduction to the proof of Theorem \ref{MainHolepunchPaperResulThm}, Subsections \ref{CommNbrCChordInterm}-\ref{MTriangleInterSectAtLeast2X0X1} consist of the proofs of Claims \ref{2ChordIncidentZVI} and \ref{noChordzOrAtLeast2Rule1Out}. To prove these results, we first introduce the following notation. 

\begin{defn}\label{SubgraphsH0H1MForEdgezOfC}
\emph{We define subgraphs $H_0, H_1, M$ of $G$, vertices $x_0, x_1$ and paths $P^{H_0}$ and $P^{H_1}$ as follows. }
\begin{enumerate}[label=\emph{\arabic*)}] 
\itemsep-0.1em
\item\emph{If $z$ is incident to a chord of $C$ then, for each $i\in\{0,1\}$, $x_i$ is the unique vertex of $N(z)\cap V(C\setminus\{q_0, q_1\})$ which is farthest from $p_i$ on the path $p_0u_1\ldots u_tp_1$. Otherwise, then, for each $i\in\{0,1\}$, $x_i=p_i$.}
\item\emph{For each $i\in\{0,1\}$, $H_i$ is the subgraph of $G$ bounded by outer face $C^{H_i}=(x_i(C\setminus\mathring{P})p_i)q_iz$. In particular, if no chord of $C$ is incident to $z$, then $H_i$ is just the 2-path $p_iq_iz$, and otherwise the outer face of $H_i$ is a cycle, as $z$ is not adjacent to either of $p_0, p_1$.} 
\item\emph{$M$ is the subgraph of $G$ induced by the vertex set $\{z, x_0, x_1\}\cup V(G\setminus (H_0\cup H_1))$.}
\item\emph{For each $i=0,1$, we let $P^{H_i}$ be the graph $(P\cap H_i)+x_iz$. That is, if there is a chord of $C$ incident to $z$, then, for each $i=0,1$, $P^{H_i}$ is the 3-path $p_iq_izx_i$.} 
\end{enumerate}
\end{defn}

Before proceeding to Subsection \ref{SubSecZEndPointP0P1CommNbr}, we have the following two useful results.

\begin{Claim}\label{zPreciselyOneMPiQiNotSize3} If $z$ is incident to at least one chord $uz$ of $C$ and $i\in\{0,1\}$ with $|L(p_i)|=1$, then $v_i=x_i$. \end{Claim}

\begin{claimproof} Suppose without loss of generality that $|L(p_0)|=1$ and $|L(p_1)|=3$, and that $z$ is incident to at least one chord of $C$. Suppose toward a contradiction that $v_0\neq x_0$. Since $C$ is incident to at least one chord of $C$, $P^{H_0}$ is a 3-path. Furthermore, no chord of the outer cycle of $H_0$ is incident to $z$, so, applying Theorem \ref{3ChordVersionMainThm1}, we fix an $L$-coloring $\phi$ of $\{p_0, x_0\}$ which is $(P^{H_0}, H_0)$-sufficient.

\vspace*{-8mm}
\begin{addmargin}[2em]{0em}
\begin{subclaim}\label{zIncAtLeastTwoSubX0V0} $z$ is incident to at least two chords of $C$. \end{subclaim}

\begin{claimproof} Suppose not. Thus, $z$ is incident to precisely one chord of $C$. Since $p_0, p_1\not\in N(z)$, there is a $u\in V(C\setminus P)$ such that $x_0=x_1=u$ and $M=uz$. Furthermore, $H_0\cup H_1=G$. Since $|L(p_1)|=3$, it follows from 1) of Theorem \ref{ThmFirstLink3PathForUseInHolepunch} that there is a $\psi\in\textnormal{Crown}(P^{H_1}, H_1)$ with $\psi(u)=\phi(u)$. Thus, $\phi\cup\psi$ is a proper $L$-coloring of its domain, which is $\{p_0\}\cup\textnormal{dom}(\psi)$. As $uz$ is the only chord of $C$ incident to $z$, we have $|L_{\phi\cup\psi}(z)|\geq 4$. Since $p_0\not\in N(q_1)$ and $\psi\in\textnormal{Crown}(P^{H_1}, H_1)$, we have $|L_{\phi\cup\psi}(q_1)|\geq 3$, so there is an extension of $\phi\cup\psi$ to an $L$-coloring $\sigma$ of $\textnormal{dom}(\psi)\cup\{p_0, z\}$ with $|L_{\sigma}(q_1)|\geq 3$. Since $v_0\neq x_0$, we have $uq_0\not\in E(G)$, so $N(q_0)\cap\textnormal{dom}(\sigma)=\{p_0, z\}$. Thus, $|L_{\sigma}(q_0)|\geq 3$ as well. By our choice of $\phi$ and $\psi$, it follows that $\sigma$ is $(P,G)$-sufficient, contradicting our assumption that $G$ is a counterexample. \end{claimproof}\end{addmargin}

By Subclaim \ref{zIncAtLeastTwoSubX0V0}, $x_0\neq x_1$, Now we apply Theorem \ref{CornerColoringMainRes} to the graph $M\cup H_1$, which contains the 3-path $x_0zq_1p_1$. By 1) of Theorem \ref{CornerColoringMainRes}, since $|L(p_1)|=3$, there is $\pi\in\textnormal{End}(z, x_0zq_1p_1, M\cup H_1)$ with $\pi(x_0)=\phi(x_0)$. Since $z$ is not adjacent to $p_0$, $\phi\cup\pi$ is a proper $L$-coloring of $\{p_0, z, x_0, p_1\}$. Possibly $x_1=v_1$, but since $z$ is incident to at least two chords of $C$, $x_0$ is not adjacent to $q_1$, so $|L_{\phi\cup\pi}(q_1)|\geq 3$. Likewise, since $q_0\not\in N(x_0)$, $|L_{\phi\cup\pi}(q_0)|\geq 3$ as well. Since $\phi\cup\pi$ is $(P,G)$-sufficient, we contradict our assumption that $G$ is a counterexample. \end{claimproof}

\begin{Claim}\label{EachAEachEllColoringH1PathFails} If there is at least one chord of $C$ incident to $z$ and $i\in\{0,1\}$ is an index with $v_i\neq x_i$, then $P^{H_i}$ is an induced path in $G$ and, for each $\phi\in\textnormal{End}(p_{1-i}q_{1-i}z, H_{1-i}\cup M)$ and each $b\in L(p_i)$, there is an $L$-coloring of $V(P^{H_i})$ which uses $\phi(z), b$ on the respective vertices $z, p_i$ and does not extend to an $L$-coloring of $H_i$.  \end{Claim}

\begin{claimproof} For the sake of definiteness, we let $i=1$ and $v_1\neq x_1$, so $q_1x_1\not\in E(G)$. Since $zp_1\not\in E(G)$ and $G$ contains no induced cycles of length four, we have $v_1\neq u_t$. If $P^{H_1}$ is not an induced path in $G$, then $zp_1\not\in E(G)$, we have $v_1=u_t$ and $u_tp_1$ is the lone chord of $P^{H_1}$. But then $u_tp_1q_1z$ is an induced 4-cycle, which is false since $G$ contains no induced 4-cycles. Thus, $P^{H_1}$ is an induced path in $G$. Now let $\phi\in\textnormal{Emd}(p_0q_0z, H_0\cup M)$ and$b\in L(p_1)$. By Claim \ref{AnyLColMidEndExtPCL11}, there is an $L$-coloring $\phi'$ of $V(P)$ which does not extend to an $L$-coloring of $G$, where $\phi'$ uses $\phi(p_0), \phi(z), b$ on the respective vertices $p_0, z, p_1$. Since $\phi\in\textnormal{End}(p_0q_0z, H_0\cup M)$, the $L$-coloring $(\phi'(p_0), \phi'(q_0), \phi'(z))$ of $p_0q_0z$ extends to an $L$-coloring $\psi$ of $H_0\cup M$, so $(\psi(x_1), a, \phi'(q_1), b)$ is an $L$-coloring of $x_1zq_1p_1$ which does not extend to an $L$-coloring of $H_1$. Since $P^{H_1}$ is an induced path in $G$, this is a proper $L$-coloring of $V(P^{H_1})$.  \end{claimproof}

\subsection{Ruling out a common neighbor in $G\setminus C$ of $z$ with either endpoint of $P$}\label{SubSecZEndPointP0P1CommNbr}

\begin{Claim}\label{piAndzHavenoNbrOutsideOuterCycleHlpCLMa} For each $i\in\{0,1\}$, $p_i$ and $z$ have no common neighbor in $G\setminus C$. \end{Claim}

\begin{claimproof} Let $i\in\{0,1\}$, say $i=1$ without loss of generality, and suppose toward a contradiction that $p_i, z$ have a common neighbor $w\in V(G\setminus C)$. Since $p_1z\not\in E(G)$ and $G$ contains no induced 4-cycles, we have $N(q_1)=\{z, w, p_1\}$, and $G-q_1$ is bounded by outer cycle $C^{G-q_1}=(C-q_1)+zwp_1$. 

\vspace*{-8mm}
\begin{addmargin}[2em]{0em}
\begin{subclaim}\label{UniqueUCommNbrwq1Z} There is a unique $u\in\{u_1, \ldots, u_t\}$ with $q_0, z, w\in N(u)$. \end{subclaim}

\begin{claimproof} Suppose not. Note that $C^{G-q_1}$ contains the 4-path $R=p_0q_0zwp_1$, and. by assumption, the three internal vertices of $R$ have no common neighbor in $C^{G-q_1}\setminus R$. Since $|V(G-q_1)|<|V(G)|$ and $|L(w)|\geq 5$, it follows from the minimality of $G$ that there is a partial $L$-coloring $\psi$ of $V(C^{G-q_1})\setminus\{q_0, w\}$ with $p_0, p_1\in\textnormal{dom}(\psi)$, where each of $q_0$ and $w$ has an $L_{\psi}$-list of size at least three any extension of $\psi$ to an $L$-coloring of $\textnormal{dom}(\psi)\cup\{q_0, w\}$ extends to an $L$-coloring of $G-q_1$. Now, since $N(q_1)\cap\textnormal{dom}(\psi)=\{z, p_1\}$, $q_1$ also has an $L_{\psi}$-list of size at least three. Since $V(C^{G-q_1}\setminus R)=V(C\setminus P)$ and $G$ is a counterexample, $\psi$ extends to an $L$-coloring $\psi^*$ of $\textnormal{dom}(\psi)\cup\{q_0, q_1\}$, where $\psi^*$ does not extend to an $L$-coloring of $G$.  Since $L_{\psi}(w)|\geq 3$ and $L_{\psi^*}(w)=L_{\psi}(w)\setminus\{\psi^*(q_0)\}$, we have $L_{\psi^*}(w)\neq\varnothing$, so $\psi^*$ extends to an $L$-coloring of $G$, a contradiction. \end{claimproof}\end{addmargin}

Let $u$ be as in Subclaim \ref{UniqueUCommNbrwq1Z}. In the notation of Definition \ref{Defnu1.utKiTi}, we have $u=v_0$. In the notation of \ref{SubgraphsH0H1MForEdgezOfC}, $M$ is an edge, i.e $M=uz$, and furthermore, $K_0=H_0-z$. Let $J=H_1\setminus\{z, q_1\}$. Note that $J$ is bounded by outer cycle $(u\ldots u_t)p_1wu$. Let $A_0=\textnormal{Col}(\textnormal{End}(p_0q_0u, K_0)\mid u)$ and let $A_1=\textnormal{Col}(\textnormal{End}(uq_1p_1, J)\mid u)$. For each $i\in\{0,1\}$, since $|L(u)|=3$ and $1\leq |L(p_i)|\leq 3$, it follows from Theorem \ref{SumTo4For2PathColorEnds} that $|A_i|\geq |L(p_i)|$. Since $|L(p_0)|+|L(p_1)|=4$, we have $A_0\cap A_1\neq\varnothing$, so there is an $L$-coloring $\phi$ of $\{p_0, u, p_1\}$, where $\phi$ restricts to an element of $\textnormal{End}(p_0q_0u, K_0)$ and also restricts to an element of $\textnormal{End}(uwp_1, J)$. Now, $|L_{\phi}(q_0)|\geq 3$ and $|L_{\phi}(z)|\geq 4$, so there is an extension of $\phi$ to an $L$-coloring $\phi'$ of $\{p_0, u, z, p_1\}$ such that $L_{\phi'}(q_1)|\geq 3$. We have $|L_{\phi'}(q_0)|\geq 3$ as well, since $\textnormal{dom}(\phi')\cap N(q_0)=\{p_0, z\}$. Since $G$ is a counterexample, there is an extension of $\phi'$ to an $L$-coloring $\psi$ of $V(P)\cup\{u\}$, where $\psi$ does not extend to an $L$-coloring of $G$. But since $|L_{\psi}(w)|\geq 1$ and $\phi$ both restricts an element of $\textnormal{End}(p_0q_0u, K_0)$ and also restricts to an element of $\textnormal{End}(uwp_1, J)$, we get that $\psi$ extends to an $L$-coloring of $G$, a contradiction.  \end{claimproof}

\subsection{Dealing with 2-chords of $C$ incident to $z$: Part I}\label{CommNbrCChordInterm}

As indicated in the introduction to Section \ref{4ChordAnalogueFor2ChSec}, over Subsections \ref{CommNbrCChordInterm}-\ref{Part22ChordSubsecR},  we show that, for each $i\in\{0,1\}$ and for any 2-chord of $C$ with $z$ as an endpoint and the other endpoint in $C\setminus P$, if the midpoint of this 2-chord is adjacent to $q_i$, then each endpoint of the 2-chord is also adjacent to $q_i$. We break the proof of this into three claims, which are Claims \ref{CommNbRWSharesNoNBrOnCStrengthen}, \ref{ThreeNeighborForViAndZWChord}, and \ref{2ChordIncidentZVI}, where we complete the proof of the result indicated above in Claim \ref{2ChordIncidentZVI}. 

\begin{Claim}\label{CommNbRWSharesNoNBrOnCStrengthen} Let $i\in\{0,1\}$ and suppose that $z$ and $q_i$ have a common neighbor $w$ in $G\setminus C$. Then the three vertices of $\{w, z, q_{1-i}\}$ do not have a common neighbor in $C\setminus P$. \end{Claim}

\begin{claimproof} Suppose without loss of generality that $i=1$ and that $z$ and $q_1$ have a common neighbor $w\in V(G\setminus C)$. Suppose toward a contradiction that that there is a common neighbor of $w, z$ and $q_0$ in $C\setminus P$ and let $m\in\{1, \ldots, t\}$, where $u_m$ is a common neighbor of $w,z$ and $q_0$. Since $wq_1\in E(G)$, $u$ is the necessarily the unique chord of $C$ incident to $z$. In the notation of Definition \ref{SubgraphsH0H1MForEdgezOfC}, we have $M=u_mz$ and $G=H_0\cup H_1$, and, since $G$ is short-separation-free, $K_0=H_0-z$, and $N(z)=\{q_0, q_1, w, u_m\}$. Furthermore, $u_m=x_0=x_1$. Since $v_1\neq u_m$, it follows from Claim \ref{EachAEachEllColoringH1PathFails} that $P^{H_1}=u_mzq_1p_1$ is an induced path in $G$.

We let $u_n$ be the neighbor of $w$ of maximal index among $\{u_m, \ldots, u_t\}$. Possibly $m=n$. We now let $R$ be the subgraph of $G$ bounded by outer face $(u_m\ldots u_n)w$ . If $m=n$, then $R$ is just an edge. Note that $v_1\neq u_m$, or else $G$ contains a triangle separating $w$ from $p_0$, contradicting short-separation-freeness. Furthermore, by Claim \ref{piAndzHavenoNbrOutsideOuterCycleHlpCLMa}, $p_1\not\in N(w)$, so it follows from our choice of $n$ that the outer cycle of $J\cup K_1$ has no chord incident to $w$. We now let $J$ be the subgraph of $H_1-z$ bounded by outer cycle $(u_n(C\setminus \mathring{P})v_1)q_1w$. Possibly $u_n=v_1$, in which case, $J$ is a triangle. In any case, since every chord of the outer face of $C$ is incident to a vertex of $\mathring{P}$, it follows from the definition of $u_n, v_1$ that the outer cycle of $J$ is an induced subgraph of $G$. This is illustrated in Figure \ref{FigureRJHSubgraphsForCommNb}. By Claim \ref{piAndzHavenoNbrOutsideOuterCycleHlpCLMa}, $p_1\not\in N(w)$. Note that the outer cycle of $J\cup K_1$ contains the 3-path $u_nwq_1p_1$, and no chord of the outer face of $J\cup K_1$ is incident to $w$. Let $X$ be the set of $L$-colorings $\psi$ of $\{u_n, p_1\}$ such that $\psi$ extends to at least $|L(w)\setminus\{\psi(w)\}|-1$ different elements of $\textnormal{End}(w, u_nwq_1p_1, J\cup K_1)$. Note that, for any $L$-coloring $\psi$ of $\{u_n, p_1\}$, we have $L_{\psi}(w)=L(w)\setminus\{\psi(w)\}$, as $p_1w\not\in E(G)$. As indicated above, no chord of the outer face of $J\cup K_1$ is incident to $w$. Since $L(u_n)|=3$ and $2\leq |L(p_1)|\leq 3$, we immediately have the following by Theorem \ref{CornerColoringMainRes}.

\vspace*{-8mm}
\begin{addmargin}[2em]{0em}
\begin{subclaim}\label{ColSizeXunLp1} $|\textnormal{Col}(X\mid u_n)|\geq |L(p_1)|$. \end{subclaim}\end{addmargin}

For each $\psi\in X$, we let $\Delta[\psi]$ be the set of $f\in L(w)\setminus\{\psi(w)\}$ such that $\psi$ extends to an element of $\textnormal{End}(w, u_nwq_1p_1, J\cup K_1)$ using $f$ on $w$. By definition of $X$, $\Delta[\psi]$ is a subset of $L(w)\setminus\{\psi(w)\}$ consisting of all but at most one color of $L(w)\setminus\{\psi(w)\}$, and, in particular, $|\Delta[\psi]|\geq 3$.  

\begin{center}\begin{tikzpicture}
\node[shape=circle,draw=black] (p0) at (-4,0) {$p_0$};
\node[shape=circle,draw=black] (u1) at (-3, 0) {$u_1$};
\node[shape=circle,draw=white] (u1+) at (-2, 0) {$\ldots$};
\node[shape=circle,draw=black] (um) at (-1, 0) {$u_m$};
\node[shape=circle,draw=white] (umid) at (0, 0) {$\ldots$};
\node[shape=circle,draw=black] (un) at (1, 0) {$u_n$};
\node[shape=circle,draw=white] (un+) at (2., 0) {$\ldots$};
\node[shape=circle,draw=black] (v1) at (3, 0) {$v_1$};
\node[shape=circle,draw=white] (v1+) at (4, 0) {$\ldots$};
\node[shape=circle,draw=black] (ut) at (5, 0) {$u_t$};
\node[shape=circle,draw=black] (p1) at (6, 0) {$p_1$};
\node[shape=circle,draw=black] (q0) at (-2,2) {$q_0$};
\node[shape=circle,draw=black] (q1) at (6,2) {$q_1$};
\node[shape=circle,draw=black] (z) at (-1,3) {$z$};
\node[shape=circle,draw=black] (w) at (0,2) {$w$};
\node[shape=circle,draw=white] (K0) at (-2.1,0.9) {$K_0$};
\node[shape=circle,draw=white] (J) at (2.8,0.9) {$J$};
\node[shape=circle,draw=white] (R) at (0,0.9) {$R$};
\node[shape=circle,draw=white] (K1) at (5.5, 0.9) {$K_1$};
 \draw[-] (p1) to (ut) to (v1+) to (v1) to (un+) to (un) to (umid) to (um) to (u1+) to (u1) to (p0) to (q0) to (z) to (q1) to (p1);
 \draw[-] (q0) to (um);
 \draw[-] (z) to (um);
 \draw[-] (w) to (z);
 \draw[-] (w) to (un);
 \draw[-] (w) to (q1);
 \draw[-] (w) to (um);
\draw[-] (q1) to (v1);

\end{tikzpicture}\captionof{figure}{}\label{FigureRJHSubgraphsForCommNb}\end{center}

\vspace*{-8mm}
\begin{addmargin}[2em]{0em}
\begin{subclaim} $m\neq n$, i.e $|E(R)|>1$. \end{subclaim}

\begin{claimproof} Suppose toward a contradiction that $R$ is an edge. It follows from Subclaim \ref{ColSizeXunLp1} and Corollary \ref{GlueAugFromKHCor} that there is a $\phi\in\textnormal{End}(p_0q_0u_m, K_0)$ and a $\psi\in X$ with $\phi(u_m)=\psi(u_m)$, so $\phi\cup\psi$ is a proper $L$-coloring of its domain, which is $\{p_0, u_m, p_1\}$. Since $u_mz$ is the only chord of $C$ incident to $z$, we have $|L_{\phi\cup\psi}(z)|\geq 4$. Since $|L_{\phi\cup\psi}(q_0)|\geq 3$, there is an extension of $\psi\cup\phi$ to an $L$-coloring $\sigma$ of $\{p_0, u_m, z, p_1\}$ with $|L_{\sigma}(q_0)|\geq 3$. Since $q_1\not\in N(u_m)$, we have $|L_{\sigma}(q_1)|\geq 3$ as well. Since $G$ is a counterexample, $\sigma$ extends to an $L$-coloring $\sigma'$ of $V(P)\cup\{u_m\}$ which does not extend to $L$-color $G$. Yet we have $L_{\sigma'}(w)=L(w)\setminus\{\sigma'(z), \sigma'(q_1), \psi(w)\}$, which is a subset of $L(w)\setminus\{\psi(w)\}$ of size at least two, so there is a $c\in\Delta[\psi]$ left over in $L_{\sigma'}(w)$. Since $\phi\in\textnormal{End}(p_0q_0u_m, K_0)$, it follows that $\sigma'$ extends to an $L$-coloring of $G$, a contradiction. \end{claimproof}\end{addmargin}

Since $|E(R)|>1$, $u_mwu_n$ is a 2-path. Now, by Claim \ref{zPreciselyOneMPiQiNotSize3}, we have $|L(p_0)|\neq 3$. Thus, $|L(p_1)|\geq 2$. We now show over the course of Subclaims \ref{RTriangleOneColorUmWorks}-\ref{RTriangleRuleOutCase} that $R$ is not a triangle. 

\vspace*{-8mm}
\begin{addmargin}[2em]{0em}
\begin{subclaim}\label{RTriangleOneColorUmWorks} If $R$ is a triangle, then $u_n\not\in N(q_1)$ and furthermore, $\textnormal{Col}(\textnormal{End}(p_0q_0u_m, K_0)\mid u_m)=\textnormal{Col}(\textnormal{End}(u_{m+1}wq_1p_1, J\cup K_1)\mid u_{m+1})$ and $|\textnormal{Col}(\textnormal{End}(p_0q_0u_m, K_0)|=|\textnormal{Col}(\textnormal{End}(P^{H_1}, H_1)\mid u_m)|=1$. \end{subclaim}

\begin{claimproof} Suppose that $R$ is a triangle. Thus, $n=m+1$. Furthermore, $u_{m+1}\not\in N(q_1)$, or else there is a 4-cycle separating $w$ from $p_0$. Since $|L(u_{m+1})|=3$ and no chord of the outer face of $J\cup K_1$ is incident to $w$, it follows from Theorem \ref{3ChordVersionMainThm1} that $\textnormal{End}(u_{m+1}wq_1p_1, J\cup K_1)\neq\varnothing$. Likewise, by Theorem \ref{SumTo4For2PathColorEnds}, $\textnormal{End}(p_0q_0u_m, K_0)\neq\varnothing$. Suppose toward a contradiction that Subclaim \ref{RTriangleOneColorUmWorks} does not hold. Thus, there exist a $\phi\in\textnormal{End}(p_0q_0u_m, K_0)$ and a $\psi\in\textnormal{End}(u_{m+1}wq_1p_1, J\cup K_1)$ with $\phi(u_m)\neq\psi(u_{m+1})$, so $\phi\cup\psi$ is a proper $L$-coloring of its domain, which is $\{p_0, u_{m}, u_{m+1}, p_1\}$. Since $u_mz$ is the only chord of $C$ incident to $z$, we have $|L_{\phi\cup\psi}(z)|\geq 4$, so $\phi\cup\psi$ extends to an $L$-coloring $\sigma$ of $\{p_0, u_{m}, z, u_{m+1}, p_1\}$ with $|L_{\sigma}(q_0)|\geq 3$. We also have $|L_{\sigma}(q_1)|\geq 3$, since $u_{m+1}\not\in N(q_1)$. Since $G$ is a counterexample, $\sigma$ extends to an $L$-coloring $\sigma'$ of $V(P)\cup\{u_m, u_{m+1}\}$ which does not extend to an $L$-coloring of $G$. We have $|N(w)\cap\textnormal{dom}(\sigma')|=4$, so there is a color left over for $w$. It follows from our choice of $\phi$ and $\psi$ that $\sigma'$ extends to an $L$-coloring of $G$, a contradiction. \end{claimproof}\end{addmargin}

\vspace*{-8mm}
\begin{addmargin}[2em]{0em}
\begin{subclaim}\label{ColorAInUm+1WorksForallP1} If $R$ is a triangle, then both of the following hold.
\begin{enumerate}[label=\arabic*)] 
\itemsep-0.1em
\item $|L(p_0)|=1$ and $|L(p_1)|=3$; AND
\item There is an $a\in L(u_m)\cap L(u_{m+1})$ such that any $L$-coloring of $u_{m+1}wq_1p_1$ using $a$ on $u_{m+1}$ extends to $L$-color all of $J\cup K_1$.
\end{enumerate}
\end{subclaim}

\begin{claimproof} Suppose that $R$ is a triangle. By Subclaim \ref{RTriangleOneColorUmWorks}, $|\textnormal{Col}(\textnormal{End}(p_0q_0u_m, K_0)\mid u_m)|=1$, so it follows from Theorem \ref{SumTo4For2PathColorEnds} that $|L(p_0)|=1$, and thus $|L(p_1)|=3$. Let $a$ be the lone color of $\textnormal{Col}(\textnormal{End}(p_0q_0u_m, K_0)\mid u_m)$.  By Theorem \ref{3ChordVersionMainThm1}, since $|L(u_{m+1})|=3$, it follows that, for each $b\in L(p_1)$, there is a $\phi\in\textnormal{End}(u_{m+1}wq_1p_1, J\cup K_1)$ with $\phi(p_1)=b$. By Subclaim \ref{RTriangleOneColorUmWorks}, we have $\phi(p_1)=a$ for each such $\phi$. That is, any $L$-coloring of $u_{m+1}wq_1p_1$ using $a$ on $u_{m+1}$ extends to $L$-color all of $J\cup K_1$. \end{claimproof}\end{addmargin}

Now we let $A=\textnormal{Col}(\textnormal{End}(p_0q_0z, H_0)\mid z)$. Since $1\leq |L(p_0)|\leq 3$ and $|L(z)|=5$, it follows from Theorem \ref{SumTo4For2PathColorEnds} that $|A|\geq |L(p_0)|+2$. In particular, $|A|\geq 3$.

\vspace*{-8mm}
\begin{addmargin}[2em]{0em}
\begin{subclaim}\label{RTriangleRuleOutCase} $R$ is not a triangle. \end{subclaim}

\begin{claimproof} Suppose that $R$ is a triangle. Thus, $n=m+1$, and, by Subclaim \ref{RTriangleOneColorUmWorks}, $u_{m+1}\not\in N(q_1)$. Let $a\in L(u_m)\cap L(u_{m+1})$ as in 2) of  Subclaim \ref{ColorAInUm+1WorksForallP1}. By 1) of Subclaim \ref{ColorAInUm+1WorksForallP1}, $|L(p_0)|=|L(p_1)|=2$. By Subclaim \ref{ColSizeXunLp1}, there is a $\psi\in X$ with $\psi(u_{m+1})\neq a$. Since $|A|\geq |L(p_0)|+2$, we have $|A|\geq 4$, and since $|\Delta[\psi]|\geq 3$, there is a $c\in A$ with $|\Delta[\psi]\setminus\{c\}|\geq 3$. By Claim \ref{AnyLColMidEndExtPCL11}, there is an $L$-coloring $\tau$ of $V(P)$ which uses $c, \psi(p_1)$ on the respective vertices $z, p_1$ and does not extend to $L$-color $G$. Note that we have not colored $u_m$. Since $c\in A$, the $L$-coloring $(\tau(p_0), \tau(q_0), \tau(z))$ of $p_0q_0z$ extends to an $L$-coloring $\tau_0$ of $H_0$. Consider the following cases.

\textbf{Case 1:} $\tau_0(u_m)=a$

In this case, the color $\psi(u_{m+1})$ is left over for $u_{m+1}$. By our choice of color for $z$, $|\Delta[\psi]\setminus\{\tau_0(u_m), \tau(z), \tau(q_1)\}|\geq 1$, and since $p_1w\not\in E(G)$, there is a color of $\Delta[\psi]$ left over for $w$, and thus $\tau\cup\tau_0$ extends to an $L$-coloring of $G$, a contradiction. 

\textbf{Case 2:} $\tau_0(u_m)\neq a$

In this case, the color $a$ is left over for $u_n$. Possibly $\tau(q_1)=a$, but $q_1u_{m+1}\not\in E(G)$. Since, $p_1\not\in N(w)$, and $|L(w)\setminus\{a, \tau_0(u_m), \tau(q_1), \tau_0(q_1)\}|\geq 1$, it follows from our choice of $a$ that the union $\tau\cup\tau_0$ extends to an $L$-coloring of $G$, a contradiction. \end{claimproof}\end{addmargin}

We now let $L(p_1)=\{b_0, \ldots, b_{|L(p_1)|-1}\}$. Note that, since there is precisely one chord of $G$ incident to $z$, we have $H_0\cup M=H_0$, and since $v_1\neq u_m$, it follows from Claim \ref{EachAEachEllColoringH1PathFails} that, for each $a\in A$ and $\ell\in\{0, \ldots, |L(p_1)|-1\}$, there is an $L$-coloring $\pi_{a}^{\ell}$ of $V(P^{H_1})$, where $\pi_a^{\ell}$ uses $a, b_{\ell}$ on the respective vertices $z, p_1$ and $\pi_a^{\ell}$ does not extend to $L$-color $H_1$. We now show that $J\cup K_1$ is a broken wheel with principal path $wq_1p_1$, and, in particular $u_n=u_t=v_1$. 

\vspace*{-8mm}
\begin{addmargin}[2em]{0em}
\begin{subclaim}\label{JK1BWheelPPathJTriangle} $J\cup K_1$ is a broken wheel with principal path $wq_1p_1$. In particular,  $u_n=v_1$. \end{subclaim}

\begin{claimproof} Suppose $J\cup K_1$ is not a broken wheel with principal path $wq_1p_1$. Since $|L(p_1)|>1$ and every chord of the outer face of $J\cup K_1$ is incident to $q_1$, it follows from 1) of Theorem \ref{EitherBWheelOrAtMostOneColThm} that there is an $\ell\in\{0, \ldots, |L(p_1)|-1\}$ such that any $L$-coloring of $wq_1p_1$ using $b_{\ell}$ on $p_1$ extends to an $L$-coloring of $J\cup K_1$. Say for the sake of definiteness that $\ell=0$. For each $a\in A$ and $s\in L_{\pi^0_a}(w)$, the $L$-coloring $(s, \pi^0_a(q_1), a)$ of $wq_1p_1$ extends to an $L$-coloring of $J\cup K_1$, and since $\pi^0_a$ does not extend to an $L$-coloring of $H_1$, it follows that $\Lambda_R(\pi^0_a(u_m), s, \bullet)$ is a proper subset of $L(u_n)\setminus\{s\}$. By Subclaim \ref{RTriangleRuleOutCase}, $R$ is not a triangle. Since $|L_{\pi^0_a}(w)|\geq 2$, it follows from 1) of Theorem \ref{EitherBWheelOrAtMostOneColThm} that $R$ is a broken wheel with principal path $u_mwu_n$, and $L(u_{m+1})=\{\pi^0_a(u_m)\}\cup L_{\pi^0_a}(w)$ as a disjoint union, and, in particular, $|L_{\pi^0_a}(w)|=2$. Since $|L_{\pi^0_a}(w)|=2$, the set $\{a, \pi^0_a(u_m), \pi^0_a(q_1)\}$ is a subset of $L(w)$. Thus, $L(u_{m+1})\subseteq L(w)$ and $A\subseteq L(w)$. Since $|A|\geq |L(p_0)|+2|\geq 3$, there is a $c\in A\cap L(u_{m+1})$, so we consider $\pi^0_c$. We have $L(u_{m+1})=\{\pi^0_c(u_m)\}\cup L_{\pi^0_c}(w)$ as a disjoint union. Since $\pi^0_c(z)=c$, we have $\pi^0_c(u_m)\neq c$ and $c\not\in L_{\pi^0_c}(w)$, contradicting the fact that $c\in L(u_m)$. \end{claimproof}\end{addmargin}

\vspace*{-8mm}
\begin{addmargin}[2em]{0em}
\begin{subclaim}\label{AtLeastOneColorFromAWith2inLambda} There exists an $a\in A$ such that, for each $\ell\in\{0, \ldots, |L(p_1)|-1\}$, there exists an $s\in L_{\pi^{\ell}_a}(w)$ with $|\Lambda_R(\pi^{\ell}_a(u_m), s, \bullet)|\geq 2$. \end{subclaim}

\begin{claimproof} Since $|A|\geq |L(p_0)|+2$, we have $|A|\geq 3$. Since $|L(w)\setminus L(u_{m+1})|\geq 2$, there is an $a\in A$ such that $|L(w)\setminus (L(u_{m+1})\cup\{a\})|\geq 2$. Let $\ell\in\{0, \ldots, |L(p_1)|-1\}$ and consider the $L$-coloring $\pi:=\pi^{\ell}_a$ of $P^{H_1}$. It suffices to show that there is an $s\in L_{\pi}(w)$ with $|\Lambda_R(\pi(u), s, \bullet)|\geq 2$. Suppose not. Since $L_{\pi}(w)\neq\varnothing$, it follows from 2i) of Corollary \ref{CorMainEitherBWheelAtM1ColCor} that $R$ is a broken wheel with principal path $uwu_n$. Since $|\Lambda_R(\pi(u_m), s, \bullet)|=1$ for each $s\in L_{\pi}(w)$, we have $L_{\pi}(w)\subseteq L(u_{m+1})$. As $p_1w\not\in E(G)$, we have $|L_{\pi}(w)|\geq 2$, so $L(u_{m+1})=\{\pi(u_m)\}\cup L_{\pi}(w)$ as a disjoint union, and $(L(w)\setminus L(u_{m+1}))\cap L_{\pi}(w)=\varnothing$. Since $L(w)\setminus L(u_{m+1})$ is a subset of $L(w)$ of size at least two and $p_1w\not\in E(G)$, it follows from our choice of $a$ that $L(w)\setminus L(u_{m+1})=\{\pi(u), \pi(q_0)\}$, so $\pi(u)\not\in L(u_{m+1})$, contradicting the fact that $L(u_{m+1})=\{\pi(u)\}\cup L_{\pi}(w)$. \end{claimproof}\end{addmargin}

\vspace*{-8mm}
\begin{addmargin}[2em]{0em}
\begin{subclaim}\label{AtLeastOnecolorBReakK1PrinPath} $K_1$ is a triangle with $L(p_1)\subseteq L(u_t)$. \end{subclaim}

\begin{claimproof} Suppose not. Thus, by 2) of Claim \ref{KiSizeAtMost4UnivColorBlock}, there is a $b\in L(p_1)$ which is almost $K_1$-universal, say $b=b_0$. Let $a\in A$ be as in Subclaim \ref{AtLeastOneColorFromAWith2inLambda} and consider the $L$-coloring $\pi:=\pi^0_a$ of $P^{H_1}$. By our choice of $\pi$, there is an $s\in L_{\pi}(w)$ with $\Lambda_R(\pi(u_m), s, \bullet)|\geq 2$. Since $b_0$ is almost $K_1$-universal we have $\Lambda_{K_1}(\bullet, \pi(q_1), b)|\geq 2$ as well. By Subclaim \ref{JK1BWheelPPathJTriangle}, $u_n=v_1$, and since $\pi$ does not extend to an $L$-coloring of $G$, the sets $\Lambda_R(\pi(u_m), s, \bullet)$ and $\Lambda_{K_1}(\bullet, \pi(q_1), b)$ are disjoint subsets of $L(v_1)$ of size at least two, which is false, as $|L(v_1)|=3$. \end{claimproof}\end{addmargin}

Since $v_1\neq u$, we have $m\in\{1, \ldots, t-1\}$. We note that $uu_t\not\in E(G)$, or else there is a 4-cycle separating $p_0$ from $w$. 

\vspace*{-8mm}
\begin{addmargin}[2em]{0em}
\begin{subclaim}\label{ColorUtZP13ListFails} For any $L$-coloring $\psi$ of $\{u_t, p_1, z\}$ with $\psi(z)\in A$ and $|L_{\psi}(q_1)|\geq 3$, $\psi$ extends to an $L$-coloring of $\{u_m, z, u_t, q_1, p_1\}$ which does not extend to an $L$-coloring of $H_1$. \end{subclaim}

\begin{claimproof} Since $\psi(z)\in A$, $\psi$ extends to an $L$-coloring $\psi'$ of of $\{p_0, z, u_t, p_1\}$ whose restriction to $\{p_0, z\}$ is an element of $\textnormal{End}(p_0q_0z, H_0)$. Note that each of $q_0$ and $q_1$ has an $L_{\psi'}$-list of size at least three. Since $G$ is a counterexample, $\psi'$ extends to an $L$-coloring $\psi''$ of $V(P)\cup\{u_t\}$, where $\psi''$ does not extend to an $L$-coloring of $G$. Since $\psi'$ restricts to an element of $\textnormal{End}(p_0q_0z, H_0)$, the $L$-coloring $(\psi''(p_0), \psi''(q_0), \psi''(z))$ of $p_0q_0z$ extends to an $L$-coloring $\psi_0$ of $H_0$. Since $P^{H_1}$ is an induced path in $G$ and $uv_1\not\in E(G)$, the union $\psi''\cup\psi_0$ does not extend to $L$-color $H_1$, so we are done. \end{claimproof}\end{addmargin}

\vspace*{-8mm}
\begin{addmargin}[2em]{0em}
\begin{subclaim}\label{ColoringQ1FailsListUtSubCL} For any $L$-coloring $\phi$ of $V(P^{H_1})$ which does not extend to an $L$-coloring of $P^{H_1}$, and any $r\in L_{\phi}(u_t)$, we have $|L_{\phi}(w)\setminus\{r\}|=1$. \end{subclaim}

\begin{claimproof} Suppose there is an $L$-coloring $\phi$ of $V(P^{H_1})$ and an $r\in L_{\phi}(u_t)$ such that $\phi$ does not extend to an $L$-coloring of $H_1$ and $|L_{\phi}(w)\setminus\{r\}|\neq 1$. Since $p_1w\not\in E(G)$, we have $|L_{\phi}(w)\setminus\{r\}|\geq 2$. Let $S=L_{\phi}(w)\setminus\{r\}$. Since $\phi$ does not extend to an $L$-coloring of $H_1$, we have $S\cap\Lambda_R(\phi(u), \bullet, r)=\varnothing$, and as $|S|\geq 2$ by assumption, it follows from 3) of Corollary \ref{CorMainEitherBWheelAtM1ColCor} that $|S|=2$ and $R$ is a broken wheel with principal path $u_mwu_t$, where $R-w$ has odd length. By Subclaim \ref{RTriangleRuleOutCase}, $R$ is a not a triangle, so $L(u_{m+1})=S\cup\{\phi(u)\}$ as a disjoint union and $L(u_{t-1})=S\cup\{r\}$ as a disjoint union. Recall that $|A|\geq |L(p_0)|+2$, so $|A|\geq 3$. We have $|L(w)\setminus L(u_{m+1})|\geq 2$, so there is an $a\in A$ with $|L(w)\setminus (L(u_{m+1})\cup\{a\})|\geq 2$. Thus, let $c, c'$ be two distinct colors of $L(w)\setminus L(u_{m+1})$. Since $|L(p_1)|\geq 2$, we consider the $L$-coloring $\pi_a^{\ell}$ of $V(P^{H_1})$ for each $\ell=0,1$. Consider the following cases. 

\textbf{Case 1:} For some $\ell\in\{0,1\}$, $L_{\pi^{\ell}_a}(u_t)\neq\{r\}$

Since $uu_t\not\in E(G)$, we have $|L_{\pi_a^{\ell}}(u_t)|\geq 1$ for each $\ell=0,1$, so, in this case, we suppose without loss of generality that there is an $f\in L_{\pi^0_a}(u_t)$ with $f\neq r$. Let $\pi=\pi_a^0$. Since $\pi$ does not extend to an $L$-coloring of $H_1$, we have $L_{\pi}(w)\cap\Lambda_{R}(\pi(u_m), \bullet, f)=\varnothing$. Since $|L_{\pi}(w)|\geq 2$, there is an $s\in L_{\pi}(w)\setminus\{f\}$. Since the $L$-coloring $(\pi(u_m), s, f)$ of $uwu_t$ does not extend to an $L$-coloring of $J$, it follows that $f\in L(u_{t-1})$ and $\pi(u_m)\in L(u_{m+1})$. Since $f\neq r$ and $L(u_{t-1})=S\cup\{r\}$, we have $f\in S$, and it follows from 3) of Corollary \ref{CorMainEitherBWheelAtM1ColCor} that $\pi(u_m)=f$. That is, $L_{\pi}(w)\setminus\{f\}=L_{\pi}(w)$, and $L_{\pi}(w)\cap\Lambda_R(f, \bullet, f)=\varnothing$. But since $f\not\in L_{\pi}(w)$ and at least one of $c, c'$ is distinct from $\pi(q_1)$, it follows from our choice of $a$ that at least one of $c, c'$ lies in $L_{\pi}(w)$. Since this color also lies outside of $L(u_{m+1})$, we have $L_{\pi}(w)\cap\Lambda_{R}(\pi(u_m), \bullet, f)\neq\varnothing$, a contradiction.

\textbf{Case 2:} For each $\ell\in\{0,1\}$, $L_{\pi^{\ell}_a}(u_t)=\{r\}$

In this case, for each $\ell=0,1$, since $\pi^{\ell}_a$ does not extend to an $L$-coloring of $H_1$, we have $L_{\pi^a_{\ell}}(w)\cap\Lambda_R(\pi^a_{\ell}(u_m), \bullet, r)=\varnothing$. By 3) of Corollary \ref{CorMainEitherBWheelAtM1ColCor}, we have $\pi^0_a(u_m)=\pi^1_a(u_m)=\phi(u_m)$.Thus, $L_{\pi^{\ell}_a}(w)\cap\Lambda_R(\phi(u_m), \bullet, r)=\varnothing$ for each $\ell=0,1$. Now, at least one of $c, c'$ is distinct from $r$, say $c\neq r$ for the sake of definiteness. Since $\phi(u)\in L(u_{m+1})$, we have $c\neq\phi(u)$ as well. Now, since $K_1$ is a triangle and $L_{\pi^{\ell}_a}(u_t)=\{r\}$ for each $\ell=0,1$, we have $\pi^0_a(q_1)\neq\pi^1_a(q_1)$, so suppose for the sake of definiteness that $\pi^0_a(q_1)\neq c$. But then, since $\phi(u)\neq c$, it follows from our choice of $a$ that $c\in L_{\pi^0_a}(w)$, and since $c\not\in L(u_{m+1})$ and $c\not\in\{\phi(u_m), r\}$, we have $c\in L_{\pi^0_a}(w)\cap\Lambda_R(\phi(u_m), \bullet, r)$, a contradiction. This proves Subclaim \ref{ColoringQ1FailsListUtSubCL}. \end{claimproof}\end{addmargin}

Combining Subclaims \ref{ColorUtZP13ListFails} and \ref{ColoringQ1FailsListUtSubCL}, we have the following. 

\vspace*{-8mm}
\begin{addmargin}[2em]{0em}
\begin{subclaim}\label{ForBreakEither2Listor3List} For any $L$-coloring $\sigma$ of $\{p_1, u_t, z\}$ with $\sigma(z)\in A$, either $|L_{\sigma}(w)|=3$ or $|L_{\sigma}(q_1)|=2$.\end{subclaim}

\begin{claimproof}  Suppose toward a contradiction that there is an $L$-coloring $\sigma$ of $\{p_1, u_t, z\}$ with $\sigma(z)\in A$, where $\sigma$ does not satisfy Subclaim \ref{ForBreakEither2Listor3List}. Since $p_1w\not\in E(G)$, it follows that $|L_{\sigma}(w)|\geq 4$ and $|L_{\sigma}(q_1)|\geq 3$. By Subclaim \ref{ColorUtZP13ListFails}, since $\sigma(z)\in A$, $\sigma$ extends to an $L$-coloring $\sigma'$ of $V(P^{H_1})\cup\{u_t\}$ which does not extend to an $L$-coloring of $H_1$. Since $|L_{\sigma}(w)|\geq 4$, we have $|L_{\sigma'}(w)|\geq 2$, contradicting Subclaim \ref{ColoringQ1FailsListUtSubCL}. \end{claimproof}\end{addmargin}

\vspace*{-8mm}
\begin{addmargin}[2em]{0em}
\begin{subclaim}\label{ADisjointLUTSubCLMW} $A$ is disjoint to $L(u_t)$, and furthermore, $L(u_t)\subseteq L(w)$. \end{subclaim}

\begin{claimproof} Suppose there is an $a\in A\cap L(u_t)$. Since $|L(p_1)|\geq 2$, we have $|L(p_1)\setminus\{a\}|\geq 1$. Since $zu_t\not\in E(G)$, there is an $L$-coloring of $\{p_1, u_t, z\}$ violating Subclaim \ref{ForBreakEither2Listor3List}. Thus, $A$ is disjoint to $L(u_t)$. Now suppose there is an $a\in L(u_t)\setminus L(w)$. Recall that $|A|\geq |L(p_0)|+2$. Since $|L(p_0)|+|L(p_1)|=4$, we have $|A|+|L(p_1)\setminus\{a\}|\geq 5$. As $|L(q_1)\setminus\{a\}|\geq 4$ and $p_1z\not\in E(G)$, there is an $L$-coloring of $\{p_1, u_t, z\}$ violating Subclaim \ref{ForBreakEither2Listor3List}. \end{claimproof}\end{addmargin}

Let $a\in A$ and consider the $L$-coloring $\pi=\pi^0_a$ of $V(P^{H_1})$. Since $|L_{\pi}(u_t)|\geq 1$, it follows from Subclaim \ref{ColoringQ1FailsListUtSubCL} that $a\in L(w)$. Thus, $A\subseteq L(w)$. By Subclaim \ref{ADisjointLUTSubCLMW}, $L(u_t)\subseteq L(w)$ and $A\cap L(u_t)=\varnothing$. Since $|A|\geq |L(p_0)|+2\geq 3$ and $|L(u_t)|=3$, we contradict the fact that $|L(w)|=5$. This completes the proof of Claim \ref{CommNbRWSharesNoNBrOnCStrengthen}. \end{claimproof}

\subsection{Dealing with 2-chords of $C$ with $z$ as an endpoint: part II}\label{SeqResultsfor2Chords4Path}

With Claim \ref{CommNbRWSharesNoNBrOnCStrengthen}. in hand, we prove the lone result which makes up Subsection \ref{SeqResultsfor2Chords4Path}. 

\begin{Claim}\label{ThreeNeighborForViAndZWChord} Let $i\in\{0,1\}$, where $v_i, z, q_i$ have a common neighbor $w\in V(G\setminus C)$. Then $N(w)\cap V(C)\subseteq V(P)\cup\{v_i\}$. \end{Claim}

\begin{claimproof} Suppose without loss of generality that $i=1$ and that $z, q_1$ and $v_1$ have a common neighbor $w\in V(G\setminus C)$. By Claim \ref{piAndzHavenoNbrOutsideOuterCycleHlpCLMa}, $v_1\neq p_1$, so $K_1$ is not just an edge, and $v_1$ lies on the path $C\setminus P$. We let $v_1=u_n$ for some $n\in\{1, \ldots, t\}$. Suppose toward a contradiction that $w$ has a neighbor on $C\setminus P$ other than $u_n$, and let $u_m$ be the neighbor of $w$ of maximal index among $\{u_1, \ldots, u_{n-1}\}$. Let $J$ be the subgraph of $G$ bounded by outer cycle $C^J=(p_0u_1\ldots u_m)wzq_0$. Note that $C^J$ contains the 4-path $P^J=p_0q_0zwu_m$. Let $R$ be the subgraph of $G$ with outer cycle $(u_m\ldots u_n)w$. As $m<n$, $R$ is not just an edge, and the outer face of $R$ contains the 2-path $u_mwu_n$. This is illustrated in Figure \ref{FigureRJHKSubgraphsCommForViQ1Z}. 

\begin{center}\begin{tikzpicture}
\node[shape=circle,draw=black] (p0) at (-4,0) {$p_0$};
\node[shape=circle,draw=black] (u1) at (-3, 0) {$u_1$};
\node[shape=circle,draw=white] (u1+) at (-2, 0) {$\ldots$};
\node[shape=circle,draw=black] (um) at (-1, 0) {$u_m$};
\node[shape=circle,draw=white] (umid) at (0, 0) {$\ldots$};
\node[shape=circle,draw=black] (un) at (1, 0) {$u_n$};
\node[shape=circle,draw=white] (v1+) at (2, 0) {$\ldots$};
\node[shape=circle,draw=black] (ut) at (3, 0) {$u_t$};
\node[shape=circle,draw=black] (p1) at (4, 0) {$p_1$};
\node[shape=circle,draw=black] (q0) at (-2,2) {$q_0$};
\node[shape=circle,draw=black] (q1) at (2,2) {$q_1$};
\node[shape=circle,draw=black] (z) at (-1,3) {$z$};
\node[shape=circle,draw=black] (w) at (0,2) {$w$};
\node[shape=circle,draw=white] (R) at (0,0.9) {$R$};
\node[shape=circle,draw=white] (J) at (-2,0.9) {$J$};
\node[shape=circle,draw=white] (K1) at (2.2, 0.9) {$K_1$};
 \draw[-] (p1) to (ut) to (v1+) to (un) to (umid) to (um) to (u1+) to (u1) to (p0) to (q0) to (z) to (q1) to (p1);

 \draw[-] (w) to (z);
 \draw[-] (w) to (un);
 \draw[-] (w) to (q1);
 \draw[-] (w) to (um);
\draw[-] (q1) to (un);

\end{tikzpicture}\captionof{figure}{}\label{FigureRJHKSubgraphsCommForViQ1Z}\end{center}

\vspace*{-8mm}
\begin{addmargin}[2em]{0em}
\begin{subclaim}\label{CrownForEachSideSubCL1} $\textnormal{Crown}_L(P^J, J)$ contains $|L(p_0)|$ distinct partial $L$-colorings of $V(C^J)$, each of which uses a different color on $u_m$. In particular, $\textnormal{Crown}_L(P^J, J)\neq\varnothing$. \end{subclaim}

\begin{claimproof}  Firstly, we have $|V(J)|<|V(G)|$, and $V(C^J\setminus P^J)\subseteq V(C\setminus P)$. It follows from Claim \ref{CommNbRWSharesNoNBrOnCStrengthen} that the three internal vertices of $P^J$ do not have a common neighbor in $C^J\setminus P^J$. Since $|L(u_m)|=3$, it immediately follows from the minimality of $G$ that $\textnormal{Crown}_L(P^J, J)$ contains $|L(p_0)|$ distinct partial $L$-colorings of $V(C^J)$, each of which uses a different color on $u_m$. \end{claimproof}\end{addmargin}

\vspace*{-8mm}
\begin{addmargin}[2em]{0em}
\begin{subclaim}\label{EachPiEachBExtensionSubCl5} For each $\pi\in\textnormal{Crown}_L(P^J, J)$ and $b\in L(p_1)$, there is an extension of $\pi$ to an $L$-coloring $\pi^b$ of $\textnormal{dom}(\pi)\cup\{p_1, q_0, q_1\}$ which does not extend to an $L$-coloring of $G$. \end{subclaim}

\begin{claimproof} Let $\pi'$ be an extension of $\pi$ to an $L$-coloring of $\textnormal{dom}(\pi)\cup\{p_1\}$ obtained by coloring $p1$ with $b$.  Possibly $b\in\{\phi(u), \phi(z)\}$, but $G$ contains neither of the edges $up_1, zp_1$. Note that $q_0$ has an $L_{\pi'}$-list of size at least three, since $\pi\in\textnormal{Crown}_L(P^J, J)$ and $p_1q_0\not\in E(G)$, and likewise, $q_1$ also has an $L_{\pi'}$-list of size at least three, since $u_m\not\in N(q_1)$ and $\textnormal{dom}(\pi')\cap N(q_1)=\{z, p_1\}$. Since $G$ is a counterexample, it follows tha  there is an extension of $\pi'$ to an $L$-coloring of $\pi^b$ of $\textnormal{dom}(\pi)\cup\{p_1, q_0, q_1\}$ which does not extend to an $L$-coloring of $G$. \end{claimproof}\end{addmargin}

For each $\pi\in\textnormal{Crown}_L(P^J, J)$ and $b\in L(p_1)$, we let $\pi^b$ be as in Subclaim \ref{EachPiEachBExtensionSubCl5}. 

\vspace*{-8mm}
\begin{addmargin}[2em]{0em}
\begin{subclaim}\label{ObstructionColoringBetweenJStarAndK1} For each $\pi\in\textnormal{Crown}_L(P^J, J)$ and $b\in L(p_1)$, we have $|L_{\pi^b}(w)|\geq 2$, and furthermore, for each $s\in L_{\pi^b}(w)$, we have $\Lambda_R(\pi(u_m), s, \bullet)\cap\Lambda_{K_1}(\bullet, \pi^b(q_1), b)=\varnothing$. \end{subclaim}

\begin{claimproof} Since $\pi\in\textnormal{Crown}_L(P^J, J)$ in, we have $|L_{\pi}(w)|\geq 3$, and thus $|L_{\pi^b}(w)|\geq 2$, as $p_1w\not\in E(G)$. Now suppose toward a contradiction that there exist an $s\in L_{\pi^b}(w)$ which violates the disjointness condition in Subclaim \ref{ObstructionColoringBetweenJStarAndK1}. Since $\pi\in\textnormal{Crown}_L(P^J, J)$, the $L$-coloring $(\pi(p_0), \pi^b(q_0), \pi(z), s, \pi(u_m))$ of the 4-path $p_0q_0zwu_m$ extends to an $L$-coloring of $J$, and since Subclaim \ref{ObstructionColoringBetweenJStarAndK1} is violated, it follows that $\pi^b$ extends to an $L$-coloring of $G$, which is false. \end{claimproof}\end{addmargin}

\vspace*{-8mm}
\begin{addmargin}[2em]{0em}
\begin{subclaim}\label{EitherLp01OrRTrianForSubMBReak55} Either $|L(p_0)|=1$ or $R$ is a triangle. \end{subclaim}

\begin{claimproof} Suppose that $R$ is not a triangle and suppose toward a contradiction that $|L(p_0)|>1$. Since $R$ is not a triangle, it follows from our choice of index $m$ that $R$ is not a broken wheel with principal path $u_mwu_n$, and since every chord of $C$ is incident to a vertex of $\mathring{P}$, it follows that $u_mu_n\not\in E(G)$. Furthermore, since $|L(p_0)|\geq 2$, it follows from Subclaim \ref{CrownForEachSideSubCL1} that there are two elements of $\textnormal{Crown}_L(P^J, J)$ using different colors on $u_m$. By 1) of Theorem \ref{EitherBWheelOrAtMostOneColThm}, there is a $\pi\in\textnormal{Crown}_L(P^J, J)$ such that any $L$-coloring of $\{u_m, w, u_n\}$ using $\pi(u_m)$ on $u_m$ extends to $L$-color $R$. Now we simply choose an $r\in\Lambda_{K_1}(\bullet, \pi^b(q_1), b) $. Possibly $r=\pi(u_m)$, but $u_mu_n\not\in E(G)$, and since $|L_{\pi^b}(w)\setminus\{r\}|\geq 1$, $\pi^b$ extends to an $L$-coloring of $G$, which is false.  \end{claimproof}\end{addmargin}

We now rule out the possibility that $|L(p_1)|=1$.

\vspace*{-8mm}
\begin{addmargin}[2em]{0em}
\begin{subclaim}\label{Lp1SizeAtLeastTwoSubCL55} $|L(p_1)|\geq 2$. \end{subclaim}

\begin{claimproof} Suppose toward a contradiction that $|L(p_1)|=1$ and let $b$ be the lone color of $L(p_1)$. Thus, $|L(p_0)|=3$. Applying Subclaim \ref{CrownForEachSideSubCL1}, let $\pi_0, \pi_1, \pi_2$ be three distinct elements of $\textnormal{Crown}_L(P^J, J)$, each of which uses a different color on $u_m$. Note that, for each $\ell=0,1,2$, $|L_{\pi^b_{\ell}}(w)|\geq 2$. By Subclaim \ref{EitherLp01OrRTrianForSubMBReak55}, $R$ is a triangle, so $n=m+1$. It follows that, for each $\ell=0,1,2$, we have $\Lambda_{K_1}(\bullet, \pi^b_{\ell}(q_1), b)=\{\pi_{\ell}(u_m)\}$, or else, if there is an $\ell$ for which this does not hold, then we choose an $r\in\Lambda_{K_1}(\bullet, \pi^b_{ell}(q_1), b)$ with $r\neq\pi_{\ell}(u_m)$, and there is a color left in $L_{\pi_{\ell}^b}(w)\setminus\{r\}$, so $\pi^{\ell}$ extends to an $L$-coloring of $G$ which is false. Since $\Lambda_{K_1}(\bullet, \pi^b_{\ell}(q_1), b)=\{\pi_{\ell}(u_m)\}$ for each $\ell=0,1,2$ and each $\pi_{\ell}$ uses a different color on $u_m$, we contradict 2ii) of Corollary \ref{CorMainEitherBWheelAtM1ColCor}.  \end{claimproof}\end{addmargin}

\vspace*{-8mm}
\begin{addmargin}[2em]{0em}
\begin{subclaim}\label{SubClMJStarTriangleNexttoK1} $R$ is a triangle. \end{subclaim}

\begin{claimproof} Suppose not. Thus, it follows from the maximality of $m$ that $R$ is not a broken wheel with principal path $u_mwu_n$, and the the outer face of $R$ is an induced cycle of $G$. Since $R$ is not a triangle, $u_mu_n\not\in E(G)$. By Subclaim \ref{EitherLp01OrRTrianForSubMBReak55}, $|L(p_0)|=1$, so $|L(p_1)|=3$. Applying Subclaim \ref{CrownForEachSideSubCL1}, we fix a $\pi\in\textnormal{Crown}_L(P^J, J)$. By 2i) of Corollary \ref{CorMainEitherBWheelAtM1ColCor}, for each $b\in L(p_1)$ and each $s\in L_{\pi^b}(w)$, we have $|\Lambda_R(\pi(u_m), s, \bullet)|\geq 2$. Applying Subclaim \ref{ObstructionColoringBetweenJStarAndK1}, it follows that, for each $b\in L(p_1)$, we have $|\Lambda_{K_1}(\bullet, \pi^b(q_1), b)|=1$.  Thus, no color of $L(p_1)$ is almost $K_1$-universal, so it follows from 2) of Claim \ref{KiSizeAtMost4UnivColorBlock} that $K_1$ is a triangle. Since $|L_{\pi}(w)|\geq 3$ and $R$ is not a broken wheel with principal path $u_mwu_n$, it follows from 1) of Theorem \ref{EitherBWheelOrAtMostOneColThm} that there exists a set $U\subseteq L_{\pi}(w)$ with $|U|=2$ such that any $L$-coloring of $u_mwu_n$ using a color of $U$ on $w$ extends to an $L$-coloring of $R$. In particular, since $u_mu_n\not\in E(G)$, we have $\Lambda_R(\pi(u_m), c, \bullet)=L(u_n)\setminus\{c\}$ for each $c\in U$. Since $|L(p_1)|=3$, there is a  $b'\in L(p_1)\setminus U$. Since at least one color of $U$ is distinct from $\pi^{b'}(q_1)$, we contradict Subclaim \ref{ObstructionColoringBetweenJStarAndK1}. \end{claimproof}\end{addmargin}

Applying Subclaim \ref{ObstructionColoringBetweenJStarAndK1} again, we immediately have the following.

\vspace*{-8mm}
\begin{addmargin}[2em]{0em}
\begin{subclaim}\label{EachMapstoPi+Ell}  For each $\pi\in\textnormal{Crown}_L(P^J, J)$ and $b\in L(p_1)$, $\Lambda_{K_1}(\bullet, \pi^b(q_1), b)=\{\pi(u_m)\}$. \end{subclaim}\end{addmargin}

By Subclaim \ref{EachMapstoPi+Ell}, there is no almost $K_1$-universal color of $L(p_1)$. By Subclaim \ref{Lp1SizeAtLeastTwoSubCL55}, $|L(p_1)|\geq 2$, so it follows from 2) of Claim \ref{KiSizeAtMost4UnivColorBlock} that $K_1$ is a triangle, i,e $n=t$ and $u_t=v_1$. Since $R$ is a triangle, and since $|L(u_t)|=3$ and $|L(p_0)|+|L(p_1)|=4$, it follows from Subclaim \ref{CrownForEachSideSubCL1} that there is a $\pi\in\textnormal{Crown}_L(P^J, J)$ and an $b\in L(p_1)$ with $L(u_t)\setminus\{\pi(u_{t-1}), b\}|\geq 2$, contradicting Subclaim \ref{EachMapstoPi+Ell}. This completes the proof of Claim \ref{ThreeNeighborForViAndZWChord}. \end{claimproof}

\subsection{Dealing with 2-chords of $C$ with $z$ as an endpoint: part III}\label{Part22ChordSubsecR}

With the facts of Subsections \ref{CommNbrCChordInterm}-\ref{SeqResultsfor2Chords4Path} in hand, we prove the following result, which makes up Subsection \ref{Part22ChordSubsecR}. 

\begin{Claim}\label{2ChordIncidentZVI} For each $i\in\{0,1\}$, if $w$ is a vertex of $ V(G\setminus C)$ which is a common neighbor of $z$ and $q_i$, then $w$ has no neighbors in $C\setminus\mathring{P}$, except possibly $v_i$. \end{Claim}

\begin{claimproof} Suppose toward a contradiction that there is an $\in\{0,1\}$ violating Claim \ref{2ChordIncidentZVI}, say $i=1$ without loss of generality. Thus, there is a $w\in V(G\setminus C)$, where $w$ is adjacent to each of $q_1$ and $z$, and $w$ has a neighbor in $V(C\setminus P)\setminus\{v_1\}$. Let $m_0, m_1\in\{1, \ldots, t\}$, where $u_{m_0}$ is the neigbor of $w$ of minimal index on $C\setminus P$ and $u_{m_1}$ is the neighbor of maximal index on $C\setminus P$. Let $J_0$ be the subgraph of $G$ bounded by outer cycle $p_0u_1\ldots u_{m_0}wzq_0$ and let $J_1$ be the subgraph of $G$ bounded by outer cycle $p_1u_t\ldots u_{m_1}wzq_1$. Let $R$ be the subgraph of $G$ bounded by outer walk $(u_{m_0}\ldots u_{m_1})w$ (possibly $m_0=m_1$ and $R$ is an edge). This is illustrated in Figure \ref{FigureFor2ChordCase}. where $wq_1$ is a chord of the outer cycle of $J_1$.

\begin{center}\begin{tikzpicture}
\node[shape=circle,draw=black] (p0) at (-5,0) {$p_0$};
\node[shape=circle,draw=black] (u1) at (-4, 0) {$u_1$};
\node[shape=circle,draw=white] (u1+) at (-3, 0) {$\ldots$};
\node[shape=circle,draw=black] (um) at (-1.5, 0) {$u_{m_0}$};
\node[shape=circle,draw=white] (umid) at (0, 0) {$\ldots$};
\node[shape=circle,draw=black] (un) at (1.5, 0) {$u_{m_1}$};
\node[shape=circle,draw=white] (un+) at (3, 0) {$\ldots$};
\node[shape=circle,draw=black] (ut) at (4, 0) {$u_t$};
\node[shape=circle,draw=black] (p1) at (5, 0) {$p_1$};
\node[shape=circle,draw=black] (q0) at (-5,2) {$q_0$};
\node[shape=circle,draw=black] (q1) at (5,2) {$q_1$};
\node[shape=circle,draw=black] (z) at (0,3) {$z$};
\node[shape=circle,draw=black] (w) at (0,2) {$w$};
\node[shape=circle,draw=white] (J0) at (-2.8,1.1) {$J_0$};
\node[shape=circle,draw=white] (J1) at (2.8,1.1) {$J_1$};
\node[shape=circle,draw=white] (R) at (0,1.1) {$R$};

 \draw[-] (p1) to (ut) to (un+) to (un) to (umid) to (um) to (u1+) to (u1) to (p0) to (q0) to (z) to (q1) to (p1);
 \draw[-] (w) to (z);
 \draw[-] (w) to (un);
 \draw[-] (w) to (q1);
 \draw[-] (w) to (um);

\end{tikzpicture}\captionof{figure}{}\label{FigureFor2ChordCase}\end{center}

\vspace*{-8mm}
\begin{addmargin}[2em]{0em}
\begin{subclaim}\label{NwNeighborsAmongUmToUn} $N(w)\cap V(C)\subseteq\{z\}\cup\{u_{m_0}, \ldots, u_{m_1}\}$, and $u_{m_1}\not\in N(q_1)$. 
\end{subclaim}

\begin{claimproof} By Claim \ref{piAndzHavenoNbrOutsideOuterCycleHlpCLMa}, $w$ is adjacent to neither $p_0$ nor $p_1$, so it is immediate by our choice of $m_0, m_1$ that $N(w)\cap V(C)\subseteq\{z\}\cup\{u_{m_0}, \ldots, u_{m_1}\}$. By assumption, $w$ has a neighbor in $V(C\setminus P)\setminus\{v_1\}$, so, in particular $u_{m_0}\neq v_1$. Suppose toward a contradiction that $u_{m_1}$ is adjacent to $q_1$, that is, $u_{m_1}=v_1$. In that case, since $u_{m_0}\neq v_1$, we contradict Claim \ref{ThreeNeighborForViAndZWChord}. \end{claimproof}\end{addmargin}

Now we apply Theorem \ref{CornerColoringMainRes}. Since $G$ is short-separation-free, $J_1-z$ is bounded by outer cycle $wu_{m_1}\ldots u_tp_1q_1$. Let $X$ be the set of $L$-colorings $\sigma$ of $\{u_{m_1}, p_1\}$ such that $\sigma$ extends to at least $|L(w)\setminus\{\sigma(u_n)\}|-1$ different elements of $\textnormal{End}(w, u_{m_1}wq_1p_1, J_1-z)$. For any $L$-coloring $\phi$ of $\{p_1, u_{m_1}\}$, we have $L_{\phi}(w)=L(w)\setminus\{\phi(u_{m_1})\}$, since $p_1\not\in N(w)$. Note that no chord of the outer face of $J_1-z$ is incident to $w$. Since $L(u_{m_1})|=3$ and $2\leq |L(p_1)|\leq 3$, we immediately have the following by Theorem \ref{CornerColoringMainRes}.

\vspace*{-8mm}
\begin{addmargin}[2em]{0em}
\begin{subclaim}\label{X1AtLeastLP1Color} $|\textnormal{Col}(X\mid u_{m_1})|\geq |L(p_1)|$. \end{subclaim}\end{addmargin}

For each $\phi\in X$, we let $\Delta[\phi]$ be the set of $c\in L(u_n)\setminus\{\phi(u_{m_1})\}$ such that there is an extension of $\phi$ to an element of $\textnormal{End}(w, u_{m_1}wq_1p_1, J_1-z,)$ which uses $c$ on $u_{m_1}$. We now note that the outer cycle of $J_0\cup R$ contains the 4-path $P'=p_0q_0zwu_{m_1}$. 

\vspace*{-8mm}
\begin{addmargin}[2em]{0em}
\begin{subclaim}\label{CrownForEachSideSubCL1} $\textnormal{Crown}_{L}(P', J_0\cup R)$ contains $|L(p_0)|$ distinct partial $L$-colorings of $V(C^{J_0\cup R})$, each of which uses a different color on $u_{m_1}$. \end{subclaim}

\begin{claimproof}  Firstly, we have $|V(J_0\cup R)|<|V(G)|$, since $p_1\not\in V(J_0\cup R)$, By Claim \ref{CommNbRWSharesNoNBrOnCStrengthen}, the three internal vertices of $P'$ do not have a common neighbor in $C^{J_0\cup R}\setminus P'$. Since $|L(u_n)|=3$, it immediately follows from the minimality of $G$ that $\textnormal{Crown}_{L}(P', J_0\cup R)$ contains $|L(p_0)|$ distinct partial $L$-colorings of $V(C^{J_0\cup R})$, each of which uses a different color on $u_{m_1}$. \end{claimproof}\end{addmargin}

\vspace*{-8mm}
\begin{addmargin}[2em]{0em}
\begin{subclaim}\label{wNotAdjQ0SubCLFor2Chord} $w\in N(q_0)$. \end{subclaim}

\begin{claimproof} Suppose $w\not\in N(q_0)$. By Subclaim \ref{X1AtLeastLP1Color}, $|\textnormal{Col}(X\mid u_{m_1})|\geq |L(p_1)|$. Since $|L(p_0)+|L(p_1)|=4=|L(u_{m_1})|+1$, it follows from Subclaim \ref{CrownForEachSideSubCL1} that there is a $\psi\in\textnormal{Crown}_{L}(P', J_0\cup R)$, where $\psi(u_{m_1})\in\textnormal{Col}(X\mid u_{m_1})$. Let $c=\psi(u_{m_1})$. Since $c\in\textnormal{Col}(X\mid u_{m_1})$, there is an $L$-coloring $\phi$ of $\{u_{m_1}, p_1\}$, where $\phi(u_{m_1})=c$ and $\phi$ extends to $|L(w)\setminus\{c\}|-1$ elements of $\textnormal{End}(w, u_{m_1}wq_1p_1, J_1-z)$. Since $\psi$ is an $L$-coloring of $\{u_{m_1}, p_1\}$ and $p_1\not\in N(z)$, the union $\psi\cup\phi$ is a proper $L$-coloring of its domain, which is $\textnormal{dom}(\psi)\cup\{p_1\}$. Note that $\textnormal{dom}(\psi)\subseteq V(C^J)\setminus\{q_1, w\}\subseteq V(C)\subseteq\{q_0, q_1\}$, and $p_0, u_{m_1}, z\in\textnormal{dom})(\psi)$. Since $\psi\in\textnormal{Crown}_{L}(P', J_0\cup R)$ and $q_0\not\in N(p_1)$, we have $|L_{\psi\cup\phi}(q_0)|\geq 3$. Since $u_n\not\in N(q_1)$, we have $N(q_1)\cap\textnormal{dom}(\psi\cup\phi)=\{z, p_1\}$, so $|L_{\psi\cup\phi}(q_1)|\geq 3$ as well. Since $G$ is a counterexample, $\psi\cup\phi$ extends to an $L$-coloring $\sigma$ of $\textnormal{dom}(\psi\cup\phi)\cup\{q_0, q_1\}$, where $\sigma$ does not extend to $L$-color $G$. By our choice of $\phi$, $|\Delta[\phi]|\geq |L(w)\setminus\{c\}|-1$. That is, since $\Delta[\phi])$ is a subset of $L(w)\setminus\{c\}$, this set contains all but at most one color of $L(w)\setminus\{c\}$. Since $\psi\in\textnormal{Crown}_{L}(P', J_0\cup R)$, we have $|L_{\psi}(w)|\geq 3$. By assumption, $q_0\not\in N(w)$, and since $p_1\not\in N(w)$ as well, we have $L_{\sigma}(w)=L_{\psi}(w)\setminus\{\sigma(q_0)\}$. Thus, $L_{\sigma}(w)$ is a subset of $L(w)\setminus\{c\}$ of size at least two, so there is a $d\in L_{\psi}(w)\cap \Delta[\phi]$. Since $\psi\in\textnormal{Crown}_{L}(P', J_0\cup R)$ and there is an extension of $\phi$ to an element of $\textnormal{End}(w, u_{m_1}wq_1p_1, J_1-z)$ using $d$ on $w$, it follows that $\sigma$ extends to an $L$-coloring of $G$, a contradiction.\end{claimproof}\end{addmargin}

Since each of $q_0$ and $q_1$ is adjacent to $w$, we have $N(z)=\{q_0, w, q_1\}$.  We now introduce the following definition. For each $i\in\{0,1\}$, we say that $q_i$ is \emph{covered} if $N(q_i)\cap\{u_{m_0}, u_{m_1}\}=\{u_{m_i}\}$. Otherwise we say that $q_i$ is \emph{uncovered}. Note that if $q_i$ is uncovered, then $q_i$ has no neighbors in $N(w)\cap V(C)$. By Subclaim \ref{NwNeighborsAmongUmToUn}, $q_1\not\in N(u_{m_1})$, so $q_1$ is uncovered. In the subclaims below, we prove some facts about an arbitrary $i\in\{0,1\}$ such that $q_i$ is uncovered (in particular, we eventually show that $q_0$ is covered). We introduce the terms covered and uncovered to make explicit that, since $w$ is adjacent to each of $q_0, q_1$, we don't lose any generality when we prove facts of this form by choosing an $i\in\{0,1\}$ for which $q_i$ is uncovered, even though we started by choosing the index $i=1$ to violate Claim \ref{2ChordIncidentZVI}. 

\vspace*{-8mm}
\begin{addmargin}[2em]{0em}
\begin{subclaim}\label{EitherREdgeOrNoNbrUmUnQ0Q1} Either $R$ is an edge or both of $q_0, q_1$ are uncovered. \end{subclaim}

\begin{claimproof} Suppose that $R$ is not an edge. Since $q_1$ is uncovered, we just need to show that $q_0$ is uncovered. Since $R$ is not an edge, we have $N(w)\cap V(C\setminus P)\not\subseteq V(P)\cup\{v_0\}$. If $q_0u_{m_0}\in E(G)$, then $w$ is adjacent to all three of $v_0, z, q_0$, contradicting Claim \ref{ThreeNeighborForViAndZWChord}. Thus, $q_0u_{m_0}\not\in E(G)$, so $q_0$ is uncovered, as required.  \end{claimproof}\end{addmargin}

For each $i=0,1$, it follows from Subclaim \ref{NwNeighborsAmongUmToUn} that $w$ is incident to no chords of the outer face of $J_i-z$. Since $|L(u_{m_0})|=|L(u_{m_1})|=3$, we have the following by Theorem \ref{3ChordVersionMainThm1}.

\vspace*{-8mm}
\begin{addmargin}[2em]{0em}
\begin{subclaim}\label{EachUmAug0EachUnAug1} For each $i=0,1$ and each $a\in L(p_i)$, there is a $\pi\in\textnormal{End}(p_iq_iwu_{m_i}, J_i-z)$ using $a$ on $p_i$.\end{subclaim}\end{addmargin}

For each $i=0,1$, we now let $J_i^{\dagger}$ be the subgraph of $G$ bounded by outer cycle $wq_i(v_i(C\setminus\mathring{P})u_{m_{1-i}})$. Note that $J_0^{\dagger}\cap J_1^{\dagger}=R$.

\vspace*{-8mm}
\begin{addmargin}[2em]{0em}
\begin{subclaim}\label{EKEdgeThenColorPhiSubFor556} Let $i\in\{0,1\}$, where $q_i$ is uncovered and $|E(K_i)|>1$. Then for any $L$-coloring $\phi$ of $\{u_{m_{1-i}}, q_i, p_i\}$, where $\phi(p_i)$ is almost $K_i$-universal, and any $S\subseteq L_{\phi}(w)$ of size at least two, $\phi$ extends to an $L$-coloring of $(R\cup J_i)-z$ which uses a color of $S$ on $w$. \end{subclaim}

\begin{claimproof} Suppose without loss of generality that $i=1$, and suppose  that $|E(K_1)|>1$. Note that the outer cycle of $J_1^{\dagger}$ contains the 2-path $u_{m_0}wq_1$, and since every chord of $C$ has a neighbor in $\mathring{P}$, it follows from the definition of $v_1$ that every chord of the outer face of $J_1^{\dagger}$ is incident to $w$. Furthermore, since $w\not\in N(v_1)$ by Subclaim \ref{NwNeighborsAmongUmToUn}, $J_1^{\dagger}$ is not a broken wheel with principal path $u_mwq_1$. By assumption, $\phi(p_1)$ is almost $K_1$-universal, so $|\Lambda_{K_1}(\bullet, \phi(q_1), \phi(p_1))|\geq 2$. Let $L^{\dagger}$ be a list-assignment for $V(J^{\dagger})$ where $L^{\dagger}(v_1)=\{\phi(q_1)\}\cup\Lambda_{K_1}(\bullet, \phi(q_1), \phi(p_1))$, and otherwise $L^{\dagger}=L$. Since $\phi(q_1)\not\in\Lambda_{K_1}(\bullet, \phi(q_1), \phi(p_1))$, we have $|L^{\dagger}(v_1)|\geq 3$. By 1) of Theorem \ref{EitherBWheelOrAtMostOneColThm}, at most one $L^{\dagger}$-coloring of $u_mwq_1$ does not extend to an $L^{\dagger}$-coloring of $J_1^{\dagger}$, so there is an $s\in S$ such that the $L$-coloring $(\phi(u_m), s, \phi(q_1)$ of $u_{m_0}wq_1$ extends to an $L^{\dagger}$-coloring of $J_1^{\dagger}$. It follows that $\phi$ extends to an $L$-coloring of $(R\cup J_1)-z$ using $s$ on $w$. \end{claimproof}\end{addmargin}

We also have the following observation

\vspace*{-8mm}
\begin{addmargin}[2em]{0em}
\begin{subclaim}\label{QiUncoveredNotAdj} For each $i=0,1$, if $q_i$ is uncovered, then neither vertex of $\{u_{m_0}, u_{m_1}\}$ is adjacent to $p_i$ \end{subclaim}

\begin{claimproof} It suffices to show that $u_{m_i}\not\in N(p_i)$, as $u_{m_{1-i}}\in N(p_i)$ only if $R$ is an edge. Suppose  that $u_{m_i}\in N(p_i)$. Thus, $G$ contains the 4-cycle $u_{m_i}p_iq_iw$. By Subclaim \ref{NwNeighborsAmongUmToUn}, $w\not\in N(p_i)$. Since $G$ contains no induced 4-cycles, we have $u_{m_i}\in N(q_i)$, contradicting our assumption that $q_i$ is uncovered. \end{claimproof}\end{addmargin}

Applying Subclaim \ref{EKEdgeThenColorPhiSubFor556}, we have the following. 

\vspace*{-8mm}
\begin{addmargin}[2em]{0em} 
\begin{subclaim}\label{IfBothUncoveredK0K1Edge} For each $i\in\{0,1\}$, if $|L(p_i)|\geq 2$ and both of $q_0, q_1$ are uncovered, then $K_i$ is either an edge or a triangle, and, in particular, if $K_i$ is atriangle, then $p_i$ is predictable. \end{subclaim}

\begin{claimproof} Suppose toward a contradiction that there is an $i\in\{0,1\}$ violating Subclaim \ref{IfBothUncoveredK0K1Edge}, say $i=1$ without loss of generality (no generality is lost in this assumption, since each of $q_0, q_1$ is uncovered). Thus, $|L(p_1)|\geq 2$, and $K_1$ is not an edge, and, since Subclaim \ref{IfBothUncoveredK0K1Edge} is violated, it follows from 2) of Claim \ref{KiSizeAtMost4UnivColorBlock} that there is an $a\in L(p_1)$ which is almost $K_1$-universal. Applying Subclaim \ref{EachUmAug0EachUnAug1}, we fix a $\pi\in\textnormal{End}(p_0q_0wu_{m_0}, J_0-z)$. Since $q_0$ is uncovered, we have $q_0u_{m_0}\not\in E(G)$, so $|L_{\pi}(q_1)|\geq 4$. Since $|L(z)|\geq 5$ and $q_0p_1\not\in E(G)$, it follows that there is an extension of $\pi$ to an $L$-coloring $\pi'$ of $\{p_0, u_{m_0}, z\}$ such that $|L_{\pi'}(q_0)|\geq 3$. By Subclaim \ref{QiUncoveredNotAdj}, $p_1$ is not adjacent to $u_{m_0}$, so the color $a$ is left over for $p_1$ and $\pi'$ extends to an $L$-coloring $\pi''$ of $\{p_0, u_{m_0}, z, p_1\}$ with $\pi''(p_1)=a$. Possibly $R$ is an edge, but since $q_1$ is uncovered, $u_{m_0}q_1\not\in E(G)$, so $|L_{\pi''}(q_1)|\geq 3$ as well. Since $G$ is a counterexample, $\pi''$ extends to an $L$-coloring $\pi^*$ of $V(P)\cup\{u_{m_0}\}$, where $\pi^*$ does not extend to an $L$-coloring of $G$. We have $|N(w)\cap\textnormal{dom}(\pi^*)|=4$, so $|L_{\pi^*}(w)|\geq 2$ by our choice of color for $z$. Since $a$ is an almost $K_1$-universal color of $L(p_1)$, it follows from Subclaim \ref{EKEdgeThenColorPhiSubFor556}, there is an $s\in L_{\pi^*}(w)$ such that the $L$-coloring $(\pi(u_{m_0}), s, \pi^*(q_1), a)$ of $u_{m_0}wq_1p_1$ extends to $L$-color all of $(R\cup J_1)-z$, and since $\pi\in\textnormal{End}(p_0q_0wu_{m_0}, J_0-z)$, it follows that $\pi^*$ extends to an $L$-coloring of $G$, which is false.  \end{claimproof}\end{addmargin}

With Subclaim \ref{IfBothUncoveredK0K1Edge} in hand, we can now prove the following slightly stronger statement. 

\vspace*{-8mm}
\begin{addmargin}[2em]{0em}
\begin{subclaim}\label{KTriangleForRNotEdgeSub6} For each $i\in\{0,1\}$, if $|L(p_i)|\geq 2$ and $q_i$ is uncovered, then $K_i$ is either an edge or a triangle, and, in particular, if $K_i$ is a triangle, then $p_i$ is predictable. \end{subclaim}

\begin{claimproof} Suppose toward a contradiction that there is an $i\in\{0,1\}$ violating Subclaim \ref{KTriangleForRNotEdgeSub6}. Thus, $|L(p_i)|\geq 2$ and $q_i$ is uncovered. Furthermore, $K_i$ is not an edge. As Subclaim \ref{KTriangleForRNotEdgeSub6} is violated, it follows from Subclaim \ref{IfBothUncoveredK0K1Edge} that $q_{1-i}$ is covered, so we have $i=1$, and $q_0u_{m_0}\in E(G)$. As $K_1$ is not just an edge, it follows from 2) of Claim \ref{KiSizeAtMost4UnivColorBlock} that there is an $a\in L(p_1)$ which is almost $K_1$-universal. Consider the following cases.

\textbf{Case 1:} $L(z)\neq L(w)$

In this case, since $|L(w)|=5$, we fix a $c\in L(z)\setminus L(w)$. The trick is to leave $u_{m_0}$ uncolored. By Claim \ref{AnyLColMidEndExtPCL11}, there is an $L$-coloring $\sigma$ of $V(P)$ which does not extend to an $L$-coloring of $G$, where $\sigma$ uses $c,a$ on the respective vertices $w, p_1$. Since $q_0u_{m_0}\in E(G)$, $K_0$ is not an edge, and $v_0=u_{m_0}$. In particular, by Theorem \ref{thomassen5ChooseThm}, the $L$-coloring $(\sigma(p_0), \sigma(q_0))$ of $p_0q_0$ extends to an $L$-coloring $\sigma_0$ of $K_0$. Furthermore, since $G$ is short-separation-free, $N(w)\cap V(J_0)=\{q_0, z, u_{m_0}\}$. Let $S=L_{\sigma_0\cup\sigma}(w)$. By our choice of color for $z$, we have $|S|\geq 2$, and Since $a$ is an almost $K_1$-universal color of $L(p_1)$, it follows from Subclaim \ref{EKEdgeThenColorPhiSubFor556} that there is an $s\in S$ such that the $L$-coloring $(\sigma_0(u_{m_0}), s, \sigma(q_1), a)$ of $u_{m_0}wq_1p_1$ extends to $L$-color all of $(R\cup J_1)-z$, so $\sigma_0\cup\sigma$ extends to an $L$-coloring of $G$, which is false. This proves Subclaim \ref{KTriangleForRNotEdgeSub6}.

\textbf{Case 2:} $L(z)=L(w)$ 

In this case, applying Subclaim \ref{EachUmAug0EachUnAug1}, we fix a $\pi\in\textnormal{End}(p_0q_0wu_m, J_0-z)$. We first note that there is an extension of $\pi$ to an $L$-coloring $\pi'$ of $\{p_0, u_{m_0}, z\}$, where $|L_{\pi'}(q_0)|\geq 3$ and $|L_{\pi'}(w)|\geq 4$. This is immediate if $\pi(u_{m_0})\in L(z)$, since $zu_m\not\in E(G)$ and we can color $z$ with the same color as $u_m$, so now suppose that $\pi(u_m)\not\in L(z)$. Thus, $\pi(u_m)\not\in L(w)$ by the assumption of Case 2, and, since $|L_{\pi}(q_0)|\geq 3$ we just choose an arbitrary $s\in L(z)$ such that $|L_{\pi}(q_0)\setminus\{s, \pi(u_{m_0})\}|\geq 4$, and use that color on $z$ instead. In any case, we let $\pi'$ be as above. Since $q_1$ is uncovered, it follows from Subclaim \ref{QiUncoveredNotAdj} that $p_1$ has no neighbors in $\textnormal{dom}(\pi')$, so $\pi'$ extends to an $L$-coloring $\pi''$ of $\{p_0, u_{m_0},z, p_1\}$ with $\pi''(p_1)=a$. We have $|L_{\pi''}(q_1)|\geq 3$, as $q_1$ is uncovered. Since $G$ is a counterexample, $\pi''$ extends to an $L$-coloring $\pi^*$ of $V(P)\cup\{u_{m_0}\}$, where $\pi^*$ does not extend to an $L$-coloring of $G$. Since $G$ is a counterexample, $\pi'$ extends to an $L$-coloring $\pi^*$ of $V(P)\cup\{u_{m_0}\}$, where $\pi^*$ does not extend to an $L$-coloring of $G$. By our construction of $\pi^*$, we have $|L_{\pi^*}(w)|\geq 2$. Since $a$ is an almost $K_1$-universal color of $L(p_1)$, it follows from Subclaim \ref{EKEdgeThenColorPhiSubFor556}, there is an $s\in L_{\pi^*}(w)$ such that the $L$-coloring $(\pi(u_{m_0}), s, \pi^*(q_1), a)$ of $u_{m_0}wq_1p_1$ extends to $L$-color all of $(R\cup J_1)-z$, and since $\pi\in\textnormal{End}(p_0q_0wu_{m_0}, J_0-z)$, it follows that $\pi^*$ extends to an $L$-coloring of $G$, which is false. \end{claimproof}\end{addmargin}

We now have the following.

\vspace*{-8mm}
\begin{addmargin}[2em]{0em}
\begin{subclaim}\label{SigmaPiSigma'3ListsLeft} Let $i\in\{0,1\}$, where $|L(p_i)|\geq 2$ and $q_i$ is uncovered. Let $\sigma\in\textnormal{End}(u_{m_{1-i}}wq_iv_i, J_i^{\dagger})$ and $\pi\in\textnormal{Crown}(p_{1-i}q_{1-i}wu_{m_{1-i}}, J_{1-i}-z)$, where $\pi(u_{m_{1-i}})=\sigma(u_{m_{1-i}})$. Let $\sigma'$ be an extension of $\pi\cup\sigma$ to an $L$-coloring of $\textnormal{dom}(\pi\cup\sigma)\cup\{p_i\}$. Then each of $q_0, q_1$ has an $L_{\sigma'}$-list of size precisely three. \end{subclaim}

\begin{claimproof} Note that possibly $v_i=p_i$ and $\sigma'=\pi\cup\sigma$. In any case, since $w$ is not adjacent to $p_1$, and since $J_{1-i}$ has no chords of its outer cycle which are incident to $w$, we have $N(w)\cap\textnormal{dom}(\sigma')=\{u_{m_{1-i}}\}$. Since $q_i$ is uncovered, we have $N(q_i)\cap\textnormal{dom}(\sigma')=\{v_i, p_i\}$. Thus, $q_i$ has an $L_{\sigma'}$-list of size at least three. Possibly $q_{1-i}$ is covered, but, in any case, we have $N(q_{1-i})\cap\textnormal{dom}(\sigma')=N(q_{1-i})\cap\textnormal{dom}(\pi)$, so $|L_{\sigma'}(q_{1-i})|\geq 3$ by definition of $\textnormal{Crown}(p_{1-i}q_{1-i}u_{m_{1-i}}, J_{1-i}-z)$, and each of $q_0, q_1$ has an $L_{\sigma'}$-list of size at least three. Suppose toward a contradiction that one of them has an $L_{\sigma'}$-list of size greater than three. Since $z$ has no neighbors in $C-\{q_0, q_1\}$, it follows that $\sigma'$ extends to an $L$-coloring $\sigma^*$ of $\textnormal{dom}(\sigma')\cup\{z\}$, where each of $q_0, q_1$ has an $L_{\sigma^*}$-list of size at least three. Since $G$ is a counterexample, $\sigma^*$ extends to an $L$-coloring $\tau$ of $\textnormal{dom}(\sigma^*)\cup\{q_0, q_1\}$ which does not extend to $L$-color $G$. Since $N(w)\cap\textnormal{dom}(\sigma')=\{u_{m_{1-i}}\}$, we have $N(w)\cap\textnormal{dom}(\tau)=\{q_0, z, q_1, u_{m_{1-i}}\}$, so there is a color left over for $w$ in $L_{\tau}(w)$. By Subclaim \ref{KTriangleForRNotEdgeSub6}, $K_i$ is either an edge or a triangle, so we get $V(J_i\cup R)=V(J_i^{\dagger})\cup\{p_i\}$. Thus, it follows from our choice of $\pi$ and $\sigma$ that $\tau$ extends to an $L$-coloring of $G$, a contradiction.  \end{claimproof}\end{addmargin}

Applying Subclaim \ref{KTriangleForRNotEdgeSub6}, we have the following.

\vspace*{-8mm}
\begin{addmargin}[2em]{0em}
\begin{subclaim} $q_0$ is covered. \end{subclaim} 

\begin{claimproof} Suppose not. Thus, neither $q_0$ nor $q_1$ is covered, so there is an $i\in\{0,1\}$ such that $q_i$ is uncovered and $|L(p_i)|\geq 2$, and, in particular, since each of $q_0, q_1$ is uncovered, no generality is lost in supposing that $i=1$. Consider the following cases.

\textbf{Case 1:} $K_1$ is not an edge

In this case, by Subclaim  \ref{KTriangleForRNotEdgeSub6}, $K_1$ is a a triangle, since $|L(p_1)|\geq 2$. Applying Subclaim \ref{EachUmAug0EachUnAug1}, we fix a $\pi\in\textnormal{End}(p_0q_0wu_{m_0}, J_0-z)$. Since $K_1$ is a triangle, we have $v_1=u_t$. By definition of $v_1$,  no chord of the outer cycle of $J_1^{\dagger}$ is incident to $q_1$, and since $|L(u_t)|=3$, it follows from Theorem \ref{3ChordVersionMainThm1} that there is a $\sigma\in\textnormal{End}(u_{m_0}wq_1v_1, J_1^{\dagger})$ with $\sigma(u_{m_0})=\pi(u_{m_0})$. Since $|L(p_1)|\geq 2$, $\pi\cup\sigma$ extends to an $L$-coloring $\sigma'$ of $\textnormal{dom}(\pi\cup\sigma)\cup$. Since $q_0$ is uncovered, we have $|L_{\sigma'}(q_1)|\geq 4$. Finally, since $\textnormal{End}(p_0q_0wu_{m_0}, J_0-z)\subseteq \textnormal{Crown}(p_0q_0wu_{m_0}, J_0-z)$, we contradict Subclaim \ref{SigmaPiSigma'3ListsLeft}. 

\textbf{Case 2:} $K_1$ is an edge

In this case, since $K_1$ is an edge, we have $v_1=p_1$ and $J_1^{\dagger}=(R\cup J_1)-z$. In particular, the outer cycle of $J_1^{\dagger}$ contains the 3-path $u_{m_0}wq_1p_1$. Consider the following subcases.

\textbf{Subcase 2.1} $|L(p_0)|=1$

In this case, we have $|L(p_1)|=3$, and, applying Subclaim \ref{EachUmAug0EachUnAug1}, we again fix a $\pi\in\textnormal{End}(p_0q_0wu_{m_0}, J_0-z)$. Since $|L(p_1)|=3$ and no chord of the outer cycle of $J_1^{\dagger}$ is incident to $q_1$, it follows from Theorem \ref{3ChordVersionMainThm1} that there is a $\sigma\in\textnormal{End}(u_{m_0}wq_1p_1, J_1^{\dagger})$ with $\pi(u_{m_0})=\sigma(u_{m_0})$, and since $q_0$ is uncovered, $|L_{\pi\cup\sigma}(q_0)|\geq 4$. Since $\pi\cup\sigma$ already colors $p_1$ and $\textnormal{End}(p_0q_0wu_{m_0}, J_0-z)\subseteq \textnormal{Crown}(p_0q_0wu_{m_0}, J_0-z)$, we contradict Subclaim \ref{SigmaPiSigma'3ListsLeft}. 

\textbf{Subcase 2.2} $|L(p_1)|=2$

In this case, we have $|L(p_0)|=2$. Since $|L(q_0)|=2$ and $q_0$ is uncovered by assumption, we suppose that $K_0$ is an edge as well, or else we are back to Case 1 with the roles of $q_0, q_1$ interchanged. By 1) of Theorem \ref{ThmFirstLink3PathForUseInHolepunch}, there exists a set of two elements of $\textnormal{Crown}(p_0q_0wu_{m_0}, J_0-z)$, each using a different color on $u_{m_0}$. Likewise, there exists a set of two elements of $\textnormal{Crown}(u_{m_0}wq_1p_1, J_1^{\dagger})$, each using a different color on $u_{m_0}$, so there exists a $\pi\in\textnormal{Crown}(p_0q_0wu_{m_0}, J_0-z)$ and a $\sigma\in\textnormal{Crown}(u_{m_0}wq_1p_1, J_1^{\dagger})$ with $\pi(u_{m_0})=\sigma(u_{m_0})$. Thus, $\pi\cup\sigma$ is a proper $L$-coloring of its domain. Since no chord of the outer face of $J_0-z$ is incident to $w$, we have $N(w)\cap\textnormal{dom}(\pi\cup\sigma)\subseteq\textnormal{dom}(\sigma)$, so $|L_{\pi\cup\sigma}(w)|\geq 3$ by definition of $\textnormal{Crown}(u_{m_0}wq_1p_1, J_1^{\dagger})$. Since $|L(z)|\geq 5$ and $z$ has no neighbors in $C-\{q_0, q_1\}$, it follows that $\pi\cup\sigma$ extends to an $L$-coloring $\sigma^*$ of $\textnormal{dom}(\pi\cup\sigma)\cup\{z\}$ with $|L_{\sigma^*}(w)|\geq 3$. Each of $K_0, K_1$ is an edge, so each of $q_0, q_1$ has an $L_{\pi\cup\sigma}$-list of size at least four and thus an $L_{\sigma^*}$-list of size at least three. Since $G$ is a counterexample, $\sigma^*$ extends to an $L$-coloring $\tau$ of $\textnormal{dom}(\sigma^*)\cup\{q_0, q_1\}$, where $\tau$ does not extend to $L$-color $G$. By our choice of color for $z$, we have $|L_{\tau}(w)|\geq 1$. As there is a color left over for $w$, it follows from our choice of $\pi$ and $\sigma$ that $\tau$ extends to an $L$-coloring of $G$, a contradiction. \end{claimproof}\end{addmargin}

Since $u_{m_0}q_0\in E(G)$, it follows from Subclaim \ref{EitherREdgeOrNoNbrUmUnQ0Q1} that $R$ is an edge, i.e $m_0=m_1=m$ for some index $m\in\{1,\ldots, t\}$. Furthermore, $K_0$ contains the 2-path $Q_0=p_0q_0u_m$, i.e $K_0$ is not just an edge. In particular, $K_0=J_0\setminus\{w, z\}$. This is illustrated in Figure \ref{FigureSecond22ChordCaseDelta}. Recalling the notation introduced at the start of the proof of Claim \ref{2ChordIncidentZVI}, we now have the following.

\vspace*{-8mm}
\begin{addmargin}[2em]{0em}
\begin{subclaim}\label{LzLwNotSameList} $L(z)\setminus L(w)\neq\varnothing$.\end{subclaim}

\begin{claimproof} It follows from Subclaim \ref{X1AtLeastLP1Color}, together with Corollary \ref{GlueAugFromKHCor}, that there exist a $\phi\in\textnormal{End}(Q_0, K_0)$ and a $\psi\in X$ such that $\phi(u_m)=\psi(u_m)=c$ for some color $c$. We just need to show that $c\in L(w)\setminus L(z)$. If that holds, then $L(z)\setminus L(w)\neq\varnothing$, since $|L(z)|=|L(w)=5$. Suppose that $c\not\in L(w)\setminus L(z)$. Thus, either $c\in L(z)$ or $c\not\in L(w)$. Recall that $|\Delta[\psi]|\geq |L(w)\setminus\{c\}|-1\geq 3$. Note that $c\not\in\Delta[\psi]$. We note now that $\phi\cup\psi$ extends to an $L$-coloring $\sigma$ of $\{p_0, u_m, z, p_1\}$ such that $|\Delta[\psi]\setminus\{\sigma(z)\}|\geq 3$ and $|L_{\sigma}(q_0)|\geq 3$. This is immediate if $c\in L(z)$, since, in that case, we color $z$ with $c$, and then we are done since $c\not\in\Delta[\psi]$ and two neighbors of $q_0$ are using the same color. On the other hand, if $c\not\in L(z)$, then $c\not\in L(w)$ by assumption, so $|\Delta[\psi]|\geq 4$. In that case, since $|L(z)|\geq 5$ and $|L_{\phi}(q_0)|\geq 3$, we just choose an $s\in L(z)$ such that $|L_{\phi}(w)\setminus\{s\}|\geq 3$, and color $z$ with that color instead. In any case, such a $\sigma$ exists. Furthermore, since $u_m$ is not adjacent to $q_1$, we have $|L_{\sigma}(q_1)|\geq 3$ as well, so each of $q_0, q_1$ has an $L$-list of size at least three. Since $G$ is a counterexample, $\sigma$ extends to an $L$-coloring $\sigma'$ of $V(P)\cup\{u_m\}$, where $\sigma'$ does not extend to $L$-color $G$. But it follows from our choice of $\sigma(z)$ that there is at least one color left in $\Delta[\psi]\cap L_{\sigma'}(w)$, so $\sigma'$ extends to an $L$-coloring of $G$, a contradiction. \end{claimproof}\end{addmargin}

\begin{center}\begin{tikzpicture}
\node[shape=circle,draw=black] (p0) at (-3,0) {$p_0$};
\node[shape=circle,draw=black] (u1) at (-2, 0) {$u_1$};
\node[shape=circle,draw=white] (u1+) at (-1, 0) {$\ldots$};
\node[shape=circle,draw=black] (um) at (0, 0) {$u_m$};
\node[shape=circle,draw=white] (ut+) at (1, 0) {$\ldots$};
\node[shape=circle,draw=black] (ut) at (2, 0) {$u_t$};
\node[shape=circle,draw=black] (p1) at (3, 0) {$p_1$};
\node[shape=circle,draw=black] (q0) at (-3,2) {$q_0$};
\node[shape=circle,draw=black] (q1) at (3,2) {$q_1$};
\node[shape=circle,draw=black] (z) at (0,3) {$z$};
\node[shape=circle,draw=black] (w) at (0,2) {$w$};
\node[shape=circle,draw=white] (K0) at (-2.3,0.8) {$K_0$};
\node[shape=circle,draw=white] (J1) at (1.8,0.8) {$J_1$};

 \draw[-] (p1) to (ut) to (ut+) to (um) to (u1+) to (u1) to (p0) to (q0) to (z) to (q1) to (p1);
 \draw[-] (w) to (z);
 \draw[-] (w) to (q1);
\draw[-] (w) to (um);
\draw[-] (q0) to (um);
\draw[-] (w) to (q0);
\end{tikzpicture}\captionof{figure}{}\label{FigureSecond22ChordCaseDelta}\end{center}

Applying Subclaim \ref{LzLwNotSameList}, we have the following. 

\vspace*{-8mm}
\begin{addmargin}[2em]{0em}
\begin{subclaim}\label{K0ColorsNotAlmostUniversal} There is no almost $K_0$-universal color of $L(p_0)$. \end{subclaim}

\begin{claimproof} Suppose toward a contradiction that there is such an $a\in L(p_0)$. Applying Subclaim \ref{LzLwNotSameList}, we fix $d\in L(z)\setminus L(w)$. By Claim \ref{AnyLColMidEndExtPCL11}, there is an $L$-coloring $\tau$ of $V(P)$ which does not extend to an $L$-coloring of $G$, where $\tau$ uses $a,d$ on the respective vertices $p_0, z$. We now fix a color $r\in L(v_1)$, where $r=\tau(p_1)$ if $K_1$ is an edge, and otherwise $r$ is a color of $\Lambda_{K_1}(\bullet, \tau(q_1), \tau(p_1))$. Let $H$ be the subgraph of $G$ with outer cycle $C^H=u_mq_0zq_1(v_1(C\setminus\mathring{P})u_m)$ and let $L'$ be a list-assignment for $V(H)$ in which $L'(u_m)=\{\tau(q_0)\}\cup\Lambda_{K_0}(\tau(p_0), \tau(q_0), \bullet)$ and otherwise $L'=L$. By our assumption on $a$, $|\Lambda_{K_0}(\tau(p_0), \tau(q_0), \bullet)|\geq 2$, so $|L'(u_m)|\geq 3$. Now, the $L'$-coloring $(\tau(q_0), \tau(z), \tau(q_1), r)$ of the 3-path $q_0zq_1v_1$ does not extend to $L'$-color $H$, or else $\tau$ extends to $L$-color $G$. Note that no chord of the outer face of $H$ is incident to any vertex of $\{q_0, z, q_1\}$. Furthermore, $w$ is the unique vertex of $H\setminus C^H$ with more than two neighbors on $q_0zq_1v_1$. Since $q_1$ is uncovered, and since $|L'(w)\setminus\{\tau(q_0), \tau(z), \tau(v_1)\}|\geq 3$ by our choice of color for $z$, we contradict Lemma \ref{PartialPathColoringExtCL0}. \end{claimproof}\end{addmargin}

\vspace*{-8mm}
\begin{addmargin}[2em]{0em}
\begin{subclaim}\label{TwoSetsUmColorsetsDisjoint}  $\textnormal{Col}(\textnormal{End}(Q_0, K_0)\mid u_m)\cap\textnormal{Col}(\textnormal{End}(u_mwq_1p_1, J_1-z)\mid u_m)=\varnothing$. \end{subclaim}

\begin{claimproof}  Suppose not. Thus, there is a $\phi\in\textnormal{End}(Q_0, K_0)$ and a $\psi\in\textnormal{End}(u_mwq_1p_1, J_1-z)$ with $\phi(u_m)=\psi(u_m)$. Since $|L_{\phi\cup\psi}(q_0)|\geq 3$, $\phi\cup\psi$ extends to an $L$-coloring $\sigma$ of $\{p_0, u_m, z, p_1\}$ with $|L_{\sigma}(q_0)|\geq 3$. Since $q_1$ is uncovered, $|L_{\sigma}(q_1)|\geq 3$ as well, and since $G$ is a counterexample, $\sigma$ extends to an $L$-coloring $\sigma'$ of $V(P)\cup\{u_m\}$, where $\sigma'$ does not extend to an $L$-coloring of $G$. But since $|L_{\sigma'}(w)|\geq 1$, there is a color left for $w$ and it follows from our choice of $\phi, \psi$ that $\sigma'$ extends to an $L$-coloring of $G$, a contradiction. \end{claimproof}\end{addmargin}

\vspace*{-8mm}
\begin{addmargin}[2em]{0em}
\begin{subclaim} $|L(p_0)|=1$ \end{subclaim}

\begin{claimproof} Suppose toward a contradiction that $|L(p_0)|\geq 2$. By Subclaim \ref{K0ColorsNotAlmostUniversal}, there is no almost $K_0$-universal color of $L(p_0)$, and, by 2) of Claim \ref{KiSizeAtMost4UnivColorBlock}, $K_0$ is a triangle, so $m=1$. For any $c\in L(p_0)$, there is an $L$-coloring of $p_0u_1$ using $c$ on $u_1$, since $|L(p_0)\setminus\{c\}|\geq 1$. Since $K_0$ is a triangle, it follows that $\textnormal{Col}(\textnormal{End}(Q_0, K_0)\mid u_m)=L(u_m)$. Since $\textnormal{End}(u_mwq_1p_1, J_1-z)\neq\varnothing$ by Subclaim \ref{EachUmAug0EachUnAug1}, this contradicts Subclaim \ref{TwoSetsUmColorsetsDisjoint}. \end{claimproof}\end{addmargin}

Since $|L(p_0)|=1$, we have $|L(p_1)|=3$. Applying Theorem \ref{SumTo4For2PathColorEnds}, we now fix a $\phi\in\textnormal{End}(Q_0, K_0)$. 

\vspace*{-8mm}
\begin{addmargin}[2em]{0em}
\begin{subclaim}\label{K1TriUmUTnotEdge} $K_1$ is a triangle, and furthermore, $u_mu_t\not\in E(G)$  \end{subclaim}

\begin{claimproof} Suppose that $K_1$ is not a triangle. Since $|L(p_1)|=3$ and $q_1$ is uncovered, it follows from Subclaim \ref{KTriangleForRNotEdgeSub6} that $K_1$ is an edge, so no chord of the outer face of $J_1-z$ is incident to $q_1$. By Theorem \ref{3ChordVersionMainThm1}, since $|L(p_1)|=3$ there is a $\pi\in\textnormal{End}(u_mwq_1p_1, J_1-z)$ with $\pi(u_m)=\phi(u_m)$, contradicting Subclaim \ref{TwoSetsUmColorsetsDisjoint}. Now suppose toward a contradiction that $u_mu_t\in E(G)$. Since no chord of $C$ is incident to either of $u_m, u_t$, we have $m=t-1$, and, since $K_1$ is a triangle, $G$ contains the 4-cycle $wu_{t-1}u_tq_1$. Since $q_1\not\in N(u_{t-1})$ and $G$ contains no induced 4-cycles, we have $w\in N(u_t)$, contradicting Subclaim \ref{NwNeighborsAmongUmToUn} \end{claimproof}\end{addmargin}

Since $K_1$ is a triangle, $J_1^{\dagger}=J_1\setminus\{p_1,z\}$, and $v_1=u_t$. Since $|L(p_1)|=3$, it follows from Subclaim \ref{KTriangleForRNotEdgeSub6} that $L(p_1)=L(u_t)$, and it follows from  Subclaim \ref{X1AtLeastLP1Color} that, there is a $\psi\in X$ with $\psi(u_m)=\phi(u_m)$. We now let $T:=\{s\in L(z): |L_{\phi}(q_0)\setminus\{s\}|\geq 3\}$. Note that $|T|\geq 2$. 

\vspace*{-8mm}
\begin{addmargin}[2em]{0em}
\begin{subclaim}\label{EachTColorZUsesSonZ} For each $s\in T$, there is an extension of $\phi\cup\psi$ to an $L$-coloring $\tau^s$ of $V(P)\cup\{u_m\}$ which colors $z$ with $s$ and does not extend to an $L$-coloring of $G$. \end{subclaim}

\begin{claimproof} Let $s\in T$. There is an extension of $\phi\cup\psi$ to an $L$-coloring $\tau$ of $\{p_0, u_m, p_1, z\}$ onbtained by coloring $z$ with $s$. Since $s\in T$ and $u_m\not\in N(q_1)$, each of $q_0, q_1$ has an $L_{\tau}$-list of size at least three, and since $G$ is a couterexample, $\tau$ extends to an $L$-coloring $\tau^s$ of $V(P)\cup\{u_m\}$ which does not extend to an $L$-coloring of $G$. \end{claimproof}\end{addmargin}

For each $s\in T$, let $\tau^s$ be as in Subclaim \ref{EachTColorZUsesSonZ}. For each $s\in T$, since $\tau^s$ does not extend to an $L$-coloring of $G$, we have $\Delta[\psi]\cap L_{\tau^s}(w)=\varnothing$, and since $\Delta[\psi]$ is a subset of $L(w)\setminus\{\psi(u_m)\}$ of size at least three, we have $\Delta[\psi]=\{s, \tau^s(q_0), \tau^s(q_1\}$. Since this hold for each $s\in T$, we have $2\leq |T|\leq 3$. In particular, $T\neq L(z)$, so it follows that $|L_{\phi}(q_1)|=3$ and $T=L(z)\setminus L_{\phi}(q_1)$. For each $s\in T$, we have $\tau^s(q_1)\in L_{\phi}(q_1)$, so $\tau^s(q_1)\not\in T$. Thus, there exists a set of two colors $s_0, s_1$ and a color $f\in L_{\phi}(q_0)$ such that $T=\{s_0, s_1\}$, and, for each $k=0,1$, we have $\tau^{s_k}(q_1)=s_{1-k}$ and $\tau^{s_k}(q_0)=f$. In particular, $T\subseteq L(q_1)\setminus\{\psi(p_1)\}$. 

\vspace*{-8mm}
\begin{addmargin}[2em]{0em}
\begin{subclaim}\label{TAndUTShareNoColorsSubCL1} $T\cap L(u_t)=\varnothing$. \end{subclaim}

\begin{claimproof} Suppose not, and suppose without loss of generality that $s_0\in T\cap L(u_t)$. By Theorem \ref{3ChordVersionMainThm1}, there is a $\pi\in\textnormal{End}(u_mwq_1u_t, J_1^{\dagger})$ using $\phi(u_m)$ on $u_m$. Since $L(u_t)=L(p_1)$, we have either $\pi(u_t)=s_0$ or $s_0\in L(p_1)\setminus\{\pi(u_t)\}$. In any case, since $z$ has no neighbors in $C-\{q_0, q_1\}$, there is an extension of $\phi\cup\pi$ to an $L$-coloring $\sigma$ of $\{p_0, u_m z, u_t, p_1\}$ which uses $s_0$ on $z$ and uses $s_0$ on precisely one endpoint of the edge $u_tp_1$. In particular, $|L_{\sigma}(q_1)|\geq 3$, and $|L_{\sigma}(q_0)|\geq 3$ as well, since $s_0\in T$. Since $G$ is a counterexample, $\sigma$ extends to an $L$-coloring $\sigma'$ of $V(P)\cup\{u_m, u_t\}$ which does not extend to $L$-color $G$. But since $w$ only has 4 neighbors in $\textnormal{dom}(\sigma')$, there is a color left over for $w$, so it follows from our choice of $\phi$ and $\pi$ that $\sigma'$ extends to an $L$-coloring of $G$, a contradiction. \end{claimproof}\end{addmargin}

\vspace*{-8mm}
\begin{addmargin}[2em]{0em}
\begin{subclaim}\label{SubCLwStarAdjUmWQ15Cycle} There is a $w^*\in V(G\setminus C)$ adjacent to each of $u_m, w, q_1$, and $N(w)=V(\mathring{P})\cup\{u_m, w^*\}$. \end{subclaim}

\begin{claimproof} Choose an arbitrary $k\in\{0,1\}$ and let $L'$ be a list-assignment for $V(J_1^{\dagger})$ such that $L'(u_t)=(L(u_t)\setminus\{\psi(p_1)\})\cup\{s_{1-k}\}$. By Subclaim \ref{TAndUTShareNoColorsSubCL1}, $T\cap L(u_t)=\varnothing$, and since $L(u_t)=L(p_1)$, we have $|L'(u_t)|=3$. Since $\textnormal{dom}(\tau^{s_k})$ contains precisely four neighbors of $w$, we choose an $r\in L_{\tau^{s_k}}(w)$. Recalling that $\tau^{s_k}(q_1)=s_{1-k}$, it follows that the $L$-coloring $(\phi(u_m), r, s_{1-k}, \psi(p_1)$ of $u_mwq_1p_1$ does not extend to $L$-color $J_1$, or else $\tau^{s_k}$ extends to $L$-color $G$. Thus, the $L'$-coloring  $(\phi(p_1), r, s_{1-k}, \psi(p_1))$ of $u_mwq_1$ does not extend to $L'$-color $J_1^{\dagger}$, since any such extension would use a color other than $s_{1-k}$ on $u_t$ (and thus be an $L$-coloring of $J_1^{\dagger}$) and also use a color other than $\psi(p_1)$ on $u_t$. Since the outer cycle of $J_1^{\dagger}$ has no chords, it follows from Lemma \ref{PartialPathColoringExtCL0} that there is a vertex of $J_1^{\dagger}$ which does not lie on the outer face of $J_1^{\dagger}$ and is adjacent to all three of $u_m, w, q_1$.   \end{claimproof}\end{addmargin}

Letting $w^*$ be as in \ref{SubCLwStarAdjUmWQ15Cycle}, we now have the following.  

\vspace*{-8mm}
\begin{addmargin}[2em]{0em}
\begin{subclaim}\label{AnLColumutp1Ext} Any $L$-coloring of $\{u_m, q_1, p_1\}$ extends to $L$-color all of $J_1\setminus\{w, z\}$ \end{subclaim}

\begin{claimproof} Suppose toward a contradiction that there is an $L$-coloring $\pi$ of $\{u_m, q_1, p_1\}$ which does not extend to $L$-color all of $J_1\setminus\{w, z\}$. Note that, since the outer cycle of $J_1^{\dagger}$ is induced and $G$ has no cycles of length four which are induced, we have $u_mu_t\not\in E(G)$, so there is an $r\in L_{\pi}(u_t)$. In particular, there is no $L$-coloring of $J_1^{\dagger}$ using $\pi(u_m), \pi(q_1), r$ on the respective vertices $u_m, q_1, u_t$, and since $|L(w)|=5$ and the outer cycle of $J_1^{\dagger}$ is induced, it follows from Lemma \ref{PartialPathColoringExtCL0} that there is a vertex of $J_1^{\dagger}$ which does not lie on the outer cycle of $J_1^{\dagger}$ and is adjacent to all three of $u_m, u_t, q_1$. Since $G$ is short-separation-free, this vertex is precisely $w^*$, and, letting $H=J_1\setminus\{z, w, q_1, p_1\}$, $H$ is bounded by outer cycle $(u_m\ldots u_t)w^*$.

Now we recall that, for each $k=0,1$, $\tau^{s_k}$ uses $s_k, s_{1-k}$ on the respective vertices $z, q_1$. We fix a color $r\in L(w)\setminus\{\phi(u_m), f, s_0, s_1\}$. For each $k=0,1$, since $\tau^{s_k}$ does not extend to an $L$-coloring of $G$, it follows that, for each $a\in L(w^*)\setminus\{r, \phi(u_m), s_{1-k}\}$, we have $\Lambda_H^{u_mw^*u_t}\cap L_{\tau^{s_k}}(u_t)=\varnothing$. Since $T\cap L(u_t)=\varnothing$, we have $L_{\tau^{s_k}}(u_t)=L(u_t)\setminus\{\psi(p_1)\}$, so we conclude that, for each $a\in L(w^*)\setminus\{r, \phi(u_m), s_{1-k}\}$, we have $\Lambda_H(\phi(u_m), a, \bullet)=\{\psi(p_1)\}$. Since this holds for each $k=0,1$ and  since $s_0\neq s_1$, there exist at least three colors $a\in L(w^*)$ such that $\Lambda_H(\phi(u_m), a, \bullet)=\{\psi(p_1)\}$, contradicting 2) of Theorem \ref{BWheelMainRevListThm2}. \end{claimproof}\end{addmargin}

We now have enough to finish the proof of Claim \ref{2ChordIncidentZVI}. Applying Subclaim \ref{LzLwNotSameList}, we have $(z)\setminus L(w)\neq\varnothing$. By Claim \ref{AnyLColMidEndExtPCL11}, there is an $L$-coloring $\pi$ of of $V(P)$ which does not extend to an $L$-coloring of $G$, where $\pi(z)\not\in L(w)$. Let $r\in\Lambda_{K_0}(\pi(p_0), \pi(q_0), \bullet)$. Since $u_m$ is not adjacent to either of $p_1, q_1$, it follows from Subclaim \ref{AnLColumutp1Ext} that there is an $L$-coloring of $J_1\setminus\{w,z\}$ using $r, \pi(q_1), \pi(p_1)$ on the respective vertices $u_m, q_1, p_1$, so $\pi$ extends to an $L$-coloring of $G-w$. Since $w$ only has five neighbors, it follows from our choice of color for $z$ that $\pi$ extends to an $L$-coloring of $G$, a contradiction. This completes the proof of Claim \ref{2ChordIncidentZVI}. \end{claimproof}

We now show in Subsections \ref{FirstPrelimForChordsZ}-\ref{MTriangleInterSectAtLeast2X0X1} that there is no chord of $C$ incident to $z$. The structure of the proof that there is no chord of $C$ incident to $z$ is as follows. 

\begin{enumerate}[label=\roman*)] 
\itemsep-0.1em
\item In Subsection \ref{FirstPrelimForChordsZ}, we show that, if $z$ is incident to at least one chord of $C$, then, for each $i\in\{0,1\}$, $v_i=x_i$.
\item In Subsection \ref{SubMNotTrianForCHORD}, we show that if $z$ is incident to a chord of $C$, then $M$ is a triangle. 
\item In Subsections \ref{TriangleAndLx0x1InserAtMost1} and \ref{MTriangleInterSectAtLeast2X0X1}, we deal with the case where $M$ is a triangle. More precisely, we show in Subsection \ref{TriangleAndLx0x1InserAtMost1} that, if $M$ is a triangle, then $|L(x_0)\cap L(x_1)|\geq 2$. We then show that in Subsection \ref{MTriangleInterSectAtLeast2X0X1} that $M$ is not a triangle and conclude that there are no chords of $C$ incident to $z$. 
\end{enumerate}

\subsection{Dealing with the case where $M$ is an edge}\label{FirstPrelimForChordsZ}

Subsection \ref{FirstPrelimForChordsZ} consists of two claims, the first of which is as follows.

\begin{Claim}\label{CompressClaimsFromSec912} Suppose that $z$ is incident to at least one chord of $C$. Then, for each $i\in\{0,1\}$, $x_i=v_i$, (and, in particular, $K_i$ is not an edge), and furthermore, $z$ is incident to at least two chords of $C$.
 \end{Claim}

\begin{claimproof} Suppose that $z$ is incident to at least one chord of $C$ . Suppose toward a contradiction that there is an $i\in\{0,1\}$ with $x_i\neq v_i$, say $i=1$. By Claim \ref{zPreciselyOneMPiQiNotSize3}, $|L(p_1)|\geq 2$. Let $A=\textnormal{Col}(\textnormal{End}(p_0q_0z, H_0\cup M)\mid z)$. Since $|L(z)|=5$ and $1\leq |L(p_0)|\leq 3$, we have $|A|\geq |L(p_0)|+2$ by Theorem \ref{SumTo4For2PathColorEnds}. By Claim \ref{EachAEachEllColoringH1PathFails}, $P^{H_1}$ is an induced path in $G$, and, for each $a\in A$ and $b\in L(p_1)$, there is an $L$-coloring $\pi^a_b$ of $V(P^{H_1})$ which does not extend to an $L$-coloring of $H_1$. 

\vspace*{-8mm}
\begin{addmargin}[2em]{0em}
\begin{subclaim} $K_1$ is not an edge. \end{subclaim}

\begin{claimproof} Suppose that $K_1$ is an edge. Choose an arbitrary $a\in A$ and $b\in L(p_1)$ and let $\pi=\pi^a_b$. Since every chord of the outer face of $H_1$ is incident to $q_1$, there are no chords of $C^{H_1}$, so it follows from Lemma \ref{PartialPathColoringExtCL0} that there is a vertex $w\in V(H_1\setminus C^{H_1})$ with at least three neighbors in $\{x_1, z, q_1, p_1\}$. If $w$ is adjacent to both of $z, p_1$, then we contradict Claim \ref{piAndzHavenoNbrOutsideOuterCycleHlpCLMa}, so $w$ is adjacent to both of $x_1, q_1$ and precisely one of $z, p_1$. If $w$ is adjacent to $p_1$, then $zq_1wx_1$ is an induced 4-cycle of $G$, which is false, as there are no induced 4-cycles of $G$. Thus, $w$ is adjacent to all three of $x_1, q_1, z$. Since $x_1\neq v_1$, this contradicts Claim \ref{2ChordIncidentZVI}.  \end{claimproof}\end{addmargin}

\vspace*{-8mm}
\begin{addmargin}[2em]{0em}
\begin{subclaim} $K_1$ is a triangle and $L(p_1)\subseteq L(u_t)$. \end{subclaim}

\begin{claimproof} Suppose not. Since $|L(p_1)|\geq 2$ and $K_1$ is not an edge, it follows from 2) of Claim \ref{KiSizeAtMost4UnivColorBlock} that there is a $b\in L(p_1)$ which is almost $K_1$-universal. We no choose an arbitrary $a\in A$ and let $\pi:=\pi^a_b$. Let $L'$ be a list-assignment for $H_1'$ where $L'(v_1)=\Lambda_{K_1}(\bullet, \pi(q_1), b)\cup\{\pi(q_1)\}$ and otherwise $L'=L$. Since $pi(q_1)\not\in\Lambda_{K_1}(\bullet, \pi(q_1), b)$, we have $|L'(v_1)|\geq 3$. Since $\pi$ does not extend to an $L$-coloring of $H_1$, the $L'$-coloring $(\pi(x_1), a, \pi(q_1))$ of $x_1zq_1$ which does not extend to $L'$-color $H_1$, as any such extension uses a color other than $\pi(q_1)$ on $v_1$. Since the outer cycle of $H_1'$ has no chords, it follows from Lemma \ref{PartialPathColoringExtCL0} that $x_1, z, q_1$ have a common neighbor in $H_1'$ which is not on the outer face of $H_1'$. This contradicts Claim \ref{2ChordIncidentZVI}. \end{claimproof}\end{addmargin}

\vspace*{-8mm}
\begin{addmargin}[2em]{0em}
\begin{subclaim}\label{WAdjAllThreeV1Q1ZAnd2ColLef} There is a unique $w\in V(H_1)\setminus V(C^{H_1})$ which is adjacent to all three of $v_1, q_1, z$. Furthermore, for each $a\in A$ and $b\in L(p_1)$ and each $c\in L(u_t)\setminus\{b, \pi_b^a(q_1)\}$, we have $|L(w)\setminus\{a, c, \pi_b^a(q_1)\}|=2$ \end{subclaim}

\begin{claimproof} Choose an arbitrary $a\in A$ and $b\in B$ and let $\pi:=\pi_b^a$ and $c\in L(u_t)\setminus\{b, \pi(q_1)\}$. Since $G$ contains no induced 4-cycles, we have $x_1\not\in N(u_t)$, and in particular, $x_1\neq u_{t-1}$. Note that, since $v_1=u_t$, there are no chords of the outer cycle of $H_1-p_1$, and since the $L$-coloring $\pi'=(\pi(x_1), a, \pi(q_1), c)$ of $x_1zq_1u_t$ does not extend to $L$-color $H_1-p_1$, it follows from Lemma \ref{PartialPathColoringExtCL0} that there is a vertex $w\in V(H_1-p_1)$ which does not lie on the outer face of $H_1-p_1$, where $|L_{\pi'}(w)|\leq 2$.  If $w$ is adjacent to each of $x_1, q_1$, then, since there are no induced 4-cycles in $G$, we have $w\in N(z)$, contradicting Claim \ref{2ChordIncidentZVI}, so $w$ is adjacent to precisely one of $x_1, q_1$ and both of $z, u_t$. Since there are no induced 4-cycles in $G$, we have $q_1\in N(w)$, and thus $x_1\not\in N(w)$. Thus, $|L_{\pi'}(w)|=2$ and $L_{\pi'}(w)=L(w)\setminus\{a, c, \pi(q_1)\}$. \end{claimproof}\end{addmargin}

We now let $w$ be as in Subclaim \ref{WAdjAllThreeV1Q1ZAnd2ColLef}. 

\vspace*{-8mm}
\begin{addmargin}[2em]{0em}
\begin{subclaim}\label{AUtDisjointCovering} $A\cap L(u_t)=\varnothing$. \end{subclaim}

\begin{claimproof} Suppose there is an $A\in A\cap L(u_t)$. Since $|L(p_1)|\geq 2$, there is a $b\in L(p_1)\setminus\{a\}$. Since $\pi^a_b$ uses $a$ on $z$, we have $\pi^a_b(q_1)\neq a$ and $a\in L(u_t)\setminus\{a, \pi^a_b(q_1)\}$, contradicting Subclaim \ref{WAdjAllThreeV1Q1ZAnd2ColLef}. \end{claimproof}\end{addmargin}

We now have enough to finish the proof of  Claim \ref{CompressClaimsFromSec912}. It immediately follows from Subclaim \ref{WAdjAllThreeV1Q1ZAnd2ColLef} that $A\subseteq L(w)$. Recall that $|A|\geq |L(p_0)|+2\geq 3$, so it follows from Subclaim \ref{AUtDisjointCovering} that there is a $c\in L(u_t)\setminus L(w)$. Since $|L(p_1)|\geq 2$, there is a $b\in L(p_1)\setminus\{c\}$. Now choose an arbitrary $a\in A$. We have either $\pi^a_b(q_1)=c$ or $c\in L(u_t)\setminus\{b, \pi^a_b(q_1)\}$. In either case, we contradict Subclaim \ref{WAdjAllThreeV1Q1ZAnd2ColLef}. We conclude that $x_i=v_i$ for each $i=0,1$. In particular, since $z$ is adajcent to neither $p_0$ nor $p_1$, neither $K_0$ nor $K_1$ is an edge. If $z$ is incident to precisely one chord of $C$, then there is a $u\in V(C\setminus P)$ such that $u=x_0=x_1$, and $u$ is adjacent to all three vertices of $\mathring{P}$, contradicting the assumption in the statement of Theorem \ref{MainHolepunchPaperResulThm}. Thus, $z$ is incident to at least two chords of $C$. \end{claimproof}

We now prove the second of the two results which make up Subsection \ref{FirstPrelimForChordsZ}. 

\begin{Claim}\label{UnivAndAlmostUnivColorClaims} Suppose that $z$ is incident to at least one chord of $C$ and let $i\in\{0,1\}$. Then both of the following hold.
\begin{enumerate}[label=\arabic*)]
\itemsep-0.1em
\item Either $M$ is a triangle or there is no $a\in L(p_i)$ which is almost $K_i$-universal; AND
\item No color of $L(p_i)$ is $K_i$ universal. In particular, $p_i$ is predictable, and, if $|L(p_i)|\geq 2$, then $K_i$ is a broken wheel with principal path $Q_i$ and $|V(K_i)|\leq 4$. 
\end{enumerate} \end{Claim}

\begin{claimproof} By Claim \ref{CompressClaimsFromSec912}, $v_i=x_i$ for each $i=0,1$, and, in particular, neither $K_0$ nor $K_1$ is an edge. We first prove that each $i\in\{0,1\}$ satisfies 1). Suppose that $M$ is not a triangle and suppose toward a contradiction that there is an $i\in\{0,1\}$ and an $a\in L(p_i)$ which is almost $K_i$-universal, say $i=1$ for the sake of definiteness. 

\vspace*{-8mm}
\begin{addmargin}[2em]{0em}
\begin{subclaim}\label{AnyLColorObs33} Any $L$-coloring of $zq_1p_1$ using $a$ on $p_1$ extends to $L$-color all of $H_1$. \end{subclaim}

\begin{claimproof} Let $\phi$ be an $L$-coloring of $zq_1p_1$ with $\phi(p_1)=a$. Since $a$ is almost $K_1$-universal, we have $|\Lambda_{K_1}(\bullet, \phi(q_1), \phi(p_1)|\geq 2$, so there is at least one color left in $\Lambda_{K_1}(\bullet, \phi(q_1), \phi(p_1))$ which is different to $\phi(z)$, and thus $\phi$ extends to $L$-color $H_1$. \end{claimproof}\end{addmargin}

\vspace*{-8mm}
\begin{addmargin}[2em]{0em}
\begin{subclaim}\label{EitherY0Y1EmptyOrNotBWheel}
$M$ is a broken wheel with principal path $x_0zx_1$.
\end{subclaim}

\begin{claimproof} Suppose not. Since $C\setminus\mathring{P}$ is an induced subgraph of $G$ and $M$ is not a triangle, we have $x_0x_1\not\in E(G)$, as $G$ is short-separation-free. Since $|L(z)\setminus L(x_0)|\geq 2$ and $M$ is not a broken wheel with principal path $x_0zx_1$, it follows from 1) of Theorem \ref{EitherBWheelOrAtMostOneColThm} that there is a $d\in L(z)\setminus L(x_0)$ such that any $L$-coloring of $x_0zx_1$ using $d$ on $z$ extends to an $L$-coloring of $M$. By Claim \ref{AnyLColMidEndExtPCL11}, there is an $L$-coloring $\phi$ of $V(P)$ using $d, a$ on the respective vertices $z, p_1$, where $\phi$ does not extend to an $L$-coloring of $G$. Since $\phi(z)\not\in L(x_0)$, it follows from 2) of Theorem \ref{EitherBWheelOrAtMostOneColThm} that the $L$-coloring $(\phi(p_0), \phi(q_0), \phi(z))$ extends to an $L$-coloring of $H_0$. By Subclaim \ref{AnyLColorObs33}, the $L$-coloring $(\phi(z), \phi(q_1), \phi(p_1))$ of $zq_1p_1$ extends to an $L$-coloring of $H_1$. Since $x_0x_1\not\in E(G)$, $H_0\cup H_1$ is an induced subgraph of $G$, so it follows that $\phi$ extends to a proper $L$-coloring $\psi$ of $V(H_0\cup H_1)$. By our choice of color for $z$, $\psi$ extends to $L$-color $M$ as well, so $\phi$ extends to an $L$-coloring of $G$, a contradiction. \end{claimproof}\end{addmargin}

Applying Subclaim \ref{EitherY0Y1EmptyOrNotBWheel}, $M$ is a broken wheel with principal path $x_0zx_1$, and, by assumption, $M$ is not a triangle, so $x_0x_1\not\in E(G)$. Since $L(p_0)$ is nonempty, we fix a $c\in L(p_0)$. Let $\ell_0, \ell_1\in\{1, \ldots, t\}$, where $x_0=u_{\ell_0}$ and $x_1=u_{\ell_1}$. Since $M$ is not a triangle, we have $\ell_0<\ell_1-1$. By Claim \ref{AnyLColMidEndExtPCL11}, For each $s\in L(z)\setminus L(u_{\ell_0+1})$, there is an $L$-coloring $\sigma^s$ of $V(P)$ which does not extend to an $L$-coloring of $G$, where $\sigma^s$ uses $c, s, a$ on the respective vertices $p_0, z, p_1$. 

\vspace*{-8mm}
\begin{addmargin}[2em]{0em}
\begin{subclaim} $H_0$ is a broken wheel with principal path $p_0q_0z$. \end{subclaim}

\begin{claimproof} Suppose not. Since every chord of the outer face of $G_0$ is incident to $q_0$, and since $|L(z)\setminus L(u_{\ell_0+1})|\geq 2$, it follows from 1) of Theorem \ref{EitherBWheelOrAtMostOneColThm} that there is an $s\in L(z)\setminus L(u_{\ell_0+1})$ such that any $L$-coloring of $p_0q_0z$ using $s$ on $z$ extends to an $L$-coloring of $H_0$. Since $u_{\ell_0}u_{\ell_1}\not\in E(G)$, it follows from Subclaim \ref{AnyLColorObs33} that $\sigma^s$ extends to an $L$-coloring $\sigma^s_*$ of $V(H_0\cup H_1)$, and since $s\not\in L(u_{\ell_0+1})$, it follows that $\sigma^s_*$ extends to $L$-color $M$ as well, so $\sigma^s$ extends to an $L$-coloring of $G$, a contradiction.  \end{claimproof}\end{addmargin}

Since $H_0$ is a broken wheel with principal path $p_0q_0z$, $K_0$ is a broken wheel with principal path $p_0q_0u_{\ell_0}$. It follows from Subclaim \ref{AnyLColorObs33} that, for each $s\in L(z)\setminus L(u_{\ell_0+1})$, $\sigma^s$ extends to an $L$-coloring $\tau^s$ of $V(P\cup H_1)$. Let $s, s_*$ be distinct colors of $L(z)\setminus L(u_{\ell_0+1})$. Since $M$ is not a triangle, and since $s, s_*\not\in L(u_{\ell_0+1})$, it follows that neither the $L$-coloring $(c, \tau^s(q_0), s)$ nor the $L$-coloring $(c, \tau^{s_*}(q_0), s')$ of $p_0q_0z$ extends to an $L$-coloring of $H_0$. Thus, we have $\Lambda_{K_0}(c, \tau^s(q_0), \bullet)=\{s\}$ and $\Lambda_{K_0}(c, \tau^{s_*}(q_0), \bullet)=\{s_*\}$. In particular, since $s\neq s_*$, we have $\tau^s(q_0)\neq\tau^{s_*}(q_0)$. By 1b) of Theorem \ref{BWheelMainRevListThm2} applied to the broken wheel $K_0$, we have $\tau^s(q_0)=s$ and $\tau^{s_*}(q_0)=s_*$, which is false, since $q_0z\in E(G)$ and each of $\tau^s$ and $\tau^{s_*}$ is a proper $L$-coloring of its domain. This proves 1) of Claim \ref{UnivAndAlmostUnivColorClaims}.

Now we prove 2). Suppose toward a contradiction that there is an $i\in\{0,1\}$ and an $a\in L(p_i)$ which is $K_i$-universal, say $i=1$ without loss of generality. Since every $K_1$-universal color of $L(p_1)$ is also almost $K_1$-universal color, it follows from 1) that $M$ is a triangle. Let $A=\textnormal{Col}(\textnormal{End}(p_0q_0z, H_0)\mid z)$. Since $|L(z)|=5$ and $1\leq |L(p_0)|\leq 3$, we have $|A|\geq |L(p_0)|+2\geq 3$ by Theorem \ref{SumTo4For2PathColorEnds}. Let $B=L(z)\setminus L(x_1)$.

\vspace*{-8mm}
\begin{addmargin}[2em]{0em}
\begin{subclaim}\label{ABDisjointSubM} $A\cap B=\varnothing$. In particular, $L(z)=A\cup B$ as a disjoint union. \end{subclaim}

\begin{claimproof} Suppose toward a contradiction that there is a $b\in A\cap B$. By Claim \ref{AnyLColMidEndExtPCL11}, there is an $L$-coloring $\tau$ of $V(P)$ which does not extend to an $L$-coloring of $G$, where $\tau$ does not extend to an $L$-coloring of $G$. Since $b\in A$, the $L$-coloring $(\tau(p_0), \tau(q_0), \tau(z))$ of $p_0q_0z$ extends to an $L$-coloring $\tau_0$ of $H_0$. Since $b\in B$, we have $L(x_1)\setminus\{\tau(z), \tau_0(x_0), \tau(q_1)\}\neq\varnothing$. Since $a$ is $K_1$-universal and $M$ is a triangle, the union $\tau\cup\tau_0$ extends to an $L$-coloring of $G$, which is false. \end{claimproof}\end{addmargin}

Since $|B|\geq 2$, it follows from Subclaim \ref{ABDisjointSubM} that $|B|=2$ and $|A|=3$, so $|L(p_0)|=1$. Let $c$ be the lone color of $p_0$ and let $B=\{b_0, b_1\}$. 

\vspace*{-8mm}
\begin{addmargin}[2em]{0em}
\begin{subclaim}\label{BColorFlipOneToOtherK0} $B\subseteq (L(x_0)\cap L(q_0))\setminus\{c\}$ and, for each $k=0,1$, $\Lambda_{K_0}(c, b_k, \bullet)=\{b_{1-k}\}$. In particular, $L(u_1)=\{c, b_0, b_1\}$.
\end{subclaim}

\begin{claimproof} Since $A\cap B=\varnothing$, it follows from the definition of $A$ that, for each $k\in\{0,1\}$, there is an $L$-coloring $\sigma_k$ of $p_0q_0z$ which does not extend to an $L$-coloring of $H_0$, where $\sigma_k(z)=b_k$. Thus, for each $k=0,1$, we have $\Lambda_{K_0}(c, \sigma_k(q_0), \bullet)=\{b_k\}$. In particular, $B\subseteq L(x_0)$. Since $|\Lambda_{K_0}(c, \sigma_k(q_0), \bullet)|=1$ for each $k\in\{0,1\}$, it follows from 2i) of Corollary \ref{CorMainEitherBWheelAtM1ColCor} that $K_0$ is a broken wheel with principal path $p_0q_0x_0$, so $H_0$ is a broken wheel with principal path $p_0q_0z$.  Since $b_0\neq b_1$, it follows that $\sigma_0(q_0)\neq\sigma_1(q_0)$, and, by 1b) of Theorem \ref{BWheelMainRevListThm2}, we have $\sigma_k(q_0)=b_{1-k}$ for each $k=0,1$. Thus, $B\subseteq (L(q_0)\cap L(x_0))\setminus\{c\}$, and furthermore, since $|\Lambda_{K_0}(c, b_k, \bullet)|=1$ for each $k=0,1$, we have $L(u_1)=\{c, b_0, b_1\}$. \end{claimproof}\end{addmargin}

By definition, $b_0\not\in L(x_1)$. By Subclaim \ref{BColorFlipOneToOtherK0}, $b_0\in (L(x_0)\cap L(q_0))\setminus\{c\}$. Since $|L(z)\setminus\{b_0\}|=4$ and $b_0\in L(q_0)$, there is a $d\in L(z)\setminus\{b_0\}$ such that $|L(q_0)\setminus\{c,d\}|\geq 4$. Now, there is an $L$-coloring $\psi$ of $\{p_0, x_0, z, p_1\}$, where $\psi$ uses the colors $b_0, d, a$ on the respective vertices $x_0, z, p_1$. Since $c\not\in B$ and $z$ is adjacent to neither of $p_0, p_1$, $\psi$ is a proper $L$-coloring of its domain, even if $K_0$ is a triangle. By our choice of color for $z$, we have $|L_{\psi}(q_0)|\geq 3$. Since $x_1$ is uncolored, we have $|L_{\psi}(q_1)|\geq 3$. Since $G$ is a counterexample to Theorem \ref{MainHolepunchPaperResulThm}, $\psi$ extends to an $L$-coloring $\psi'$ of $V(P)\cup\{x_0\}$ such that $\psi'$ does not extend to an $L$-coloring of $G$. Possibly $K_1$ is a triangle, but since $b_0\not\in L(x_1)$, we have $|L_{\psi'}(x_1)|\geq 1$, so it follows from our choice of color for $p_1$ that, the $L$-coloring $(\psi'(z), \psi'(q_1), \psi'(p_1))$ of $zq_1p_1$ extends to an $L$-coloring $\sigma$ of $H_1$, and the union $\sigma\cup\psi'$ is a proper $L$-coloring of its domain. Since $\sigma\cup\psi'$ does not extend to an $L$-coloring of $G$, we have $b_0\not\in\Lambda_{K_0}(c, \psi'(q_0), \bullet)$. Since $\psi'(q_0)\neq\psi(x_0)$, it follows that $K_0$ is not a triangle. Furthermore, $\psi'(q_0)\not\in B$, or else, since $\psi(x_0)=b_0$, we have $\psi'(q_0)=b_1$ and we contradict Subclaim \ref{BColorFlipOneToOtherK0}. Since $\psi'(q_0)\neq c$, we have $\psi'(q_0)\not\in L(u_1)$ by Subclaim \ref{BColorFlipOneToOtherK0}. Since $K_0$ is not a triangle and $\psi'(q_0)\not\in L(u_1)$, we have $\Lambda_{K_0}(c, \psi'(q_0), \bullet)=L(x_0)\setminus\{\psi'(q_0)\}$, as $K_0$ is a broken wheel with principal path $p_0q_0x_0$. Thus, $b_0\in\Lambda_{K_0}(c, \psi'(q_0), \bullet)$, which is false, as indicated above.  

We conclude that, for each $i\in\{0,1\}$, no color of $L(p_i)$ is $K_i$-universal. We have $L(p_0)\subseteq L(u_1)$ and $L(p_1)\subseteq L(u_t)$ by 2) of Theorem \ref{EitherBWheelOrAtMostOneColThm}. By Claim \ref{KiSizeAtMost4UnivColorBlock}, for each $i=0,1$, with $|L(p_i)|\geq 2$, $K_i$ is a broken wheel with principal path $Q_i$, where $|V(K_i)|\leq 4$. \end{claimproof}

With Claims \ref{CompressClaimsFromSec912}-\ref{UnivAndAlmostUnivColorClaims} in hand, we now deal with the case where $M$ is not a triangle. 

\subsection{Dealing with the case where $|V(M)|>3$}\label{SubMNotTrianForCHORD}

Subsection \ref{SubMNotTrianForCHORD} consists of the following result. 

\begin{Claim}\label{MTriangleClaimRuleOutZChord} If $z$ is incident to at least one chord of $C$, then $M$ is a triangle. \end{Claim}

\begin{claimproof} Suppose that $z$ is incident to at least one chord of $C$. By Claim \ref{CompressClaimsFromSec912}, $x_i=v_i$ for each $i=0,1$, and, in particular, neither $K_0$ nor $K_1$ is an edge. Furthermore, $z$ is incident to at least two chords of $C$, so $x_0\neq x_1$. Suppose toward a contradiction that $M$ is not a triangle. By 2) of Claim \ref{UnivAndAlmostUnivColorClaims}, we have $L(p_0)\subseteq L(u_1)$ and $L(p_1)\subseteq L(u_t)$, and furthermore, by 1) of Claim \ref{UnivAndAlmostUnivColorClaims}, for each $i\in\{0,1\}$, there is no almost $K_i$-universal color of $L(p_i)$. It follows from 2) of Claim \ref{KiSizeAtMost4UnivColorBlock} that, for each $i\in\{0,1\}$ with $|L(p_i)|\geq 2$, $K_i$ is a triangle. 

\vspace*{-8mm}
\begin{addmargin}[2em]{0em}
\begin{subclaim}\label{MBWheelButNotTriangle} $M$ is a broken wheel with principal path $x_0zx_1$. \end{subclaim}

\begin{claimproof} Suppose not. By 1) of Theorem \ref{EitherBWheelOrAtMostOneColThm} that there is at most one $L$-coloring of $\{x_0, z, x_1\}$ which does not extend to $L$-color $M$. Consider the following cases:

\textbf{Case 1:} For some $i\in\{0,1\}$, $|L(p_i)|\neq 2$

In this case, we suppose without loss of generality that $|L(p_0)|=1$ and $|L(p_1)|=3$. Thus, $K_1$ is a triangle, so $x_1=u_t$. Since $|L(z)\setminus L(x_0)|\geq 2$, there is an $s\in L(z)\setminus L(x_0)$ such that any $L$-coloring of $\{x_0, z, x_1\}$ using $s$ on $z$ extends to an $L$-coloring of $M$. Since $|L(p_1)|=3$ and $|L(u_t)\setminus\{s\}|\geq 2$, there is an $s_1\in L(p_1)$ such that $|L(u_t)\setminus\{s_1, s\}|\geq 2$. Possibly $s=s_1$. In any case, by Claim \ref{AnyLColMidEndExtPCL11} there is an $L$-coloring $\tau$ of $V(P)$ which does not extend to an $L$-coloring of $G$, where $\tau$ uses $s, s_1$ on the respective vertices $z, p_1$. Since $s\not\in L(x_0)$, it follows from 2) of Theorem \ref{EitherBWheelOrAtMostOneColThm} that the $L$-coloring $(\tau(p_0), \tau(q_0), \tau(z))$ of $p_0q_0z$ extends to an $L$-coloring of $H_0$, so $\tau$ extends to an $L$-coloring $\tau'$ of $V(P\cup H_0)$. Since $x_0x_1\not\in E(G)$, it follows from our choice of color for $p_1$ that there is an $r\in L_{\tau'}(u_t)$ and that $(\tau'(x_0), s, r)$ is a proper $L$-coloring of $x_0zx_1$. By our choice of color for $z$, this $L$-coloring of $x_0zx_1$ extends to an $L$-coloring of $M$, and thus $\tau'$ extends to an $L$-coloring of $G$, a contradiction. 

\textbf{Case 2:} $|L(p_0)|=|L(p_1)|=2$

In this case, each of $K_0$ and $K_1$ is a triangle, so $x_0=u_1$ and $x_1=u_t$. Now, for each $i\in\{0,1\}$, we let $S_i$ be the set of colors $s\in L(z)$ such that $L(x_i)\setminus\{s\}\neq L(p_i)$. For each $I=0,1$, we have $|S_i|\geq 4$, ands since $|L(z)|=5$, we have $|S_0\cap S_1|\geq 3$, so there is an $s\in S_0\cap S_1$ such that any $L$-coloring of $x_0zx_1$ using $s$ on $z$ extends to an $L$-coloring of $G$. Since $|L(p_0)|=|L(p_1)|=2$, and since $G$ contains neither of the edges $zp_0, zp_1$, it follows from our choice of color for $z$ that there is an $L$-coloring $\sigma$ of $\{p_0, z, p_1\}$ with $\phi(z)=s$, where, for each $i=0,1$, $|L_{\phi}(x_i)|\geq 2$. By Claim \ref{AnyLColMidEndExtPCL11}, $\phi$ extends to an $L$-coloring $\phi'$ of $V(P)$ such that $\phi'$ does not extend to an $L$-coloring of $G$. But since each of $x_0, x_1$ has an $L_{\phi'}$-list of size at least one, there exist colors $r_0\in L_{\phi'}(x_0)$ and $r_1\in L_{\phi'}(x_1)$. Possibly $r_0=r_1$, but, in any case, since $x_0x_1\not\in E(G)$, it follows from our choice of color for $z$ that the $L$-coloring $(r_0, s, r_1)$ of $x_0zx_1$ extends to an $L$-coloring of $M$. Since each of $K_0$ and $K_1$ is a triangle, $\phi'$ extends to an $L$-coloring of $G$, a contradiction.  \end{claimproof}\end{addmargin}

We now let $\ell_0, \ell_1$ be indices in $\{1,\ldots, t\}$, where $x_0=u_{\ell_0}$ and $x_1=u_{\ell_1}$. As $M$ is not a triangle, $\ell_0<\ell_1-1$. 

\vspace*{-8mm}
\begin{addmargin}[2em]{0em}
\begin{subclaim}\label{EachIndexNeighborColorSubCLMTriangle} For each index $k\in\{\ell_0+1, \ldots, \ell_1-1\}$, both of the following hold.

\begin{enumerate}[label=\arabic*)]
\itemsep-0.1em
\item If $|L(p_1)|\geq 2$, then $|L(u_k)\cap L(u_{k+1})|\geq |L(p_0)|$; AND
\item If $|L(p_0)|\geq 2$, then $|L(u_{k-1})\cap L(u_k)|\geq |L(p_1)|$. 
\end{enumerate}

 \end{subclaim}

\begin{claimproof}  Let $k\in\{\ell_0, \ldots, \ell_1-1\}$. 1) and 2) are symmetric, so it suffices to prove that $k$ satisfies 1). Suppose that $|L(p_1)|\geq 2$  and suppose toward a contradiction that $|L(u_k)\cap L(u_{k+1})|<|L(p_0)|$. Let $\overleftarrow{G}$ be the subgraph of $G$ bounded by outer cycle $p_0u_1\ldots u_kzq_0$ and let $\overrightarrow{G}$ be the subgraph of $G$ bounded by outer cycle $u_k\ldots u_1p_1q_1z$. Note that $G=\overleftarrow{G}\cup\overrightarrow{G}$ and $\overleftarrow{G}\cap\overrightarrow{G}=u_kz$. Since $|L(u_k)\cap L(u_{k+1})|<|L(p_0)|$ and $|L(u_k)|=3$, we have $|L(u_k)\setminus L(u_{k+1})|>3-|L(p_0)|$, so $|L(u_k)\setminus L(u_{k+1})|\geq 4-|L(p_0)|$. Thus, we get $|L(u_k)\setminus L(u_{k+1})|+|L(p_0)|\geq 4$. The outer cycle of $\overrightarrow{G}$ contains the 3-path $R=p_0q_0zu_k$. Since $|L(u_k)\setminus L(u_{k+1})|+|L(p_0)|\geq 4$ and $|L(z)|\geq 5$, it follows from Theorem \ref{CornerColoringMainRes} that there exist a $c\in L(p_0)$, a $d\in L(u_k)\setminus L(u_{k+1})$, and a pair of distinct $L$-colorings $\phi_0, \phi_1$ of $\{p_0, u_k, z\}$ such that both of the following hold:
\begin{enumerate}[label=\roman*)]
\item $\phi_0, \phi_1$ both use $c, d$ on the respective vertices $p_0, u_k$, but use different colors on $z$; \emph{AND}
\item For each $i\in\{0,1\}$, any extension of $\phi_i$ to an $L$-coloring of $V(R)$ also extends to $L$-color all of $\overleftarrow{G}$. 
\end{enumerate}
Since $|L(p_1)|\geq 2$, $K_1$ is a triangle, so $\ell_1=t$ and $x_1=u_t$. Since $\{\phi_0(z), \phi_1(z)\}|\geq 2$, there exists a $f\in L(p_0)$ and a $k\in\{0,1\}$ such that $L(u_t)\setminus\{f, \phi_k(z)\}|\geq 2$. Thus, $\phi_k$ extends to an $L$-coloring $\psi$ of $\{p_0, z, u_k, p_1\}$ with $|L_{\psi}(u_t)|\geq 2$. Since $u_k$ is an internal vertex of the path $M-z$, $u_k$ is not adjacent to either of $q_0, q_1$, so each of $q_0, q_1$ has an $L_{\psi}$-list of size at least three. Since $G$ is a counterexample, there is an extension of $\psi$ to an $L$-coloring $\psi^*$ of $V(P)\cup\{u_k\}$ such that $\psi^*$ does not extend to an $L$-coloring of $G$. Since the $L$-coloring $(c, \psi^*(q_0), \psi^*(z), d)$ of the 3-path $p_0q_0zu_k$ extends to an $L$-coloring of $H_0$, it follows that the $L$-coloring $(d, \psi^*(z), \psi^*(q_1), \psi^*(p_1))$ of $u_kzq_1p_1$ does not extend to an $L$-coloring of $H_1$. Possibly $k=t-1$, but, in any case, since $d\not\in L(u_{k+1})$, it follows from our chice of colors for $z, p_1$ that $|L_{\psi^*}(u_t)|\geq 1$, so there is an $r\in L_{\psi^*}(u_t)$, and the $L$-coloring $(c, \psi^*(z), r)$ of the 2-path $u_kzu_t$ is a proper $L$-coloring of its domain. Since $c\not\in L(u_k)$, this $L$-coloring extends to an $L$-coloring of the broken wheel $M\setminus\{u_1, \ldots, u_{k-1}\}$ (in particular, if $k=t-1$, then it is already an $L$-coloring of this broken wheel). But then, the $L$-coloring $(d, \psi^*(z), \psi^*(q_1), \psi^*(p_1))$ of  $u_kzq_1p_1$ extends to an $L$-coloring of $H_1$, a contradiction.  \end{claimproof}\end{addmargin}

We now deal with the case where each of $p_0, p_1$ has a list of size two. 

\vspace*{-8mm}
\begin{addmargin}[2em]{0em}
\begin{subclaim}\label{1+3CaseLeftAfter2+2Case} For some $i\in\{0,1\}$, $|L(p_i)|\neq 2$ \end{subclaim}

\begin{claimproof} Suppose not. Thus, $|L(p_0)|=|L(p_1)|=2$, so each of $K_0$ and $K_1$ is a triangle. In particular, $\ell_0=1$ and $\ell_1=t$. We apply a similar argument to that of Subclaim \ref{MBWheelButNotTriangle}. For each $I=0,1$, we let $S_i$ be the set of $c\in L(z)$ such that $L(z)\setminus\{c\}\neq L(p_i)$. Each of $S_0, S_1$ has size at least four, and, since $|L(z)|=5$, $|S_0\cap S_1|\geq 4$. Consider the following cases. 

\textbf{Case 1:} There exists a $k\in\{2, \ldots, t-1\}$ such that $L(u_k)\neq S_0\cap S_1$

In this case, since $|S_0\cap S_1|\geq 3$, there is a $c\in S_0\cap S_1$ with $c\not\in L(u_k)$. Since $c\in S_0\cap S_1$ and $|L(p_0)|=|L(p_1)|=2$, there is an $L$-coloring $\phi$ of $\{p_0, z, p_1\}$ with $\phi(z)=c$, where $|L_{\phi}(u_1)|\geq 2$ and $|L_{\phi}(u_t)|\geq 2$. By Claim \ref{AnyLColMidEndExtPCL11}, $\phi$ extends to an $L$-coloring $\phi'$ of $V(P)$, where $\phi'$ does not extend to an $L$-coloring of $G$. By our choice of $\phi$, we have $|L_{\phi'}(u_1)|\geq 1$ and $|L_{\phi'}(u_t)|\geq 1$. Since $x_0x_1\not\in E(G)$, it follows that there is an extension of $\phi'$ to an $L$-coloring $\phi''$ of $V(P)\cup\{u_1, u_t\}$. Since $c\not\in L(u_k)$, $\phi''$ extends to $L$-color $M$ as well, so $\phi$ extends to an $L$-coloring of $G$, a contradiction.

\textbf{Case 2:} For each $k\in\{2, \ldots, t-1\}$, $L(u_k)=S_0\cap S_1$

In this case, $|S_0\cap S_1|=3$ and, since $|S_i|\geq 4$ for each $i\in\{0,1\}$, it follows that $|S_0|=|S_1|=4$ and $S_0\neq S_1$. Thus, there is a color $r\in L(z)\setminus S_1$. Since $r\not\in S_1$, we have $L(u_t)\setminus\{r\}=L(p_1)$, so $L(u_t)=L(p_1)\cup\{r\}$ as a disjoint union. By Subclaim \ref{EachIndexNeighborColorSubCLMTriangle}, since $|L(p_0)|=|L(p_1)|=2$, we have $|L(u_{t-1})\cap L(u_t)|\geq 2$.  Since $L(u_{t-1})=S_0\cap S_1$, we have $r\not\in L(u_{t-1})$, so $L(u_{t-1})\cap L(u_t)=L(p_1)$. In particular, there is an $s\in L(u_{t-1})\cap L(p_1)$, and we have $s\in S_0\cap S_1$. The trick now is to color $u_t$. Since $s\in S_0$ and $|L(p_0)|=2$, and since $u_1u_t, zp_1\not\in E(G)$, it follows there is an $L$-coloring $\phi$ of $\{p_0, z, u_t, p_1\}$ with $|L_{\phi}(u_1)|\geq 2$, where $\phi$ uses the colors $s,r,s$ on the respective vertices $z, u_t, p_1$. Since $\phi$ uses the same color on $z, p_1$, we have $|L_{\phi}(q_1)|\geq 3$. Since $\textnormal{dom}(\phi)\cap N(q_0)=\{z, p_0\}$, we have $|L_{\phi}(q_0)|\geq 3$ as well. Since $G$ is a counterexample, there is an extension of $\phi$ to an $L$-coloring $\phi'$ of $V(P)\cup\{u_t\}$, where $\phi'$ does not extend to an $L$-coloring of $G$. Since $|L_{\phi'}(u_1)|\geq 1$, there is an $r'\in L_{\phi'}(u_1)$. Possibly $r=r'$, but in any case, since $M$ is not a triangle, $(r', s, r)$ is a proper $L$-coloring of $u_1zu_t$. Since $L(u_{t-1})=S_0\cap S_1$, $r\not\in L(u_{t-1})$, so this $L$-coloring of $u_1zu_t$ extends to $L$-color $M$. Thus, $\phi'$ extends to an $L$-coloring of $G$, a contradiction. \end{claimproof}\end{addmargin}

Applying Subclaim \ref{1+3CaseLeftAfter2+2Case}, we suppose without loss of generality that $|L(p_0)|=1$ and $|L(p_1)|=3$. Thus, $K_1$ is a triangle, so $\ell_1=t$ and $x_1=u_t$. Let $a$ be the lone color of $L(p_0)$. Let $A\subseteq L(z)$ be the set of colors $c$ such that the $L$-coloring of $\{p_0, z\}$ using $a,c$ on the respective vertices $p_0, z$ ins an element of $\textnormal{End}(p_0q_0z, H_0)$. Since $|L(z)|=5$, it follows from Theorem \ref{SumTo4For2PathColorEnds} that $|A|\geq 3$. 

\vspace*{-8mm}
\begin{addmargin}[2em]{0em}
\begin{subclaim}\label{SubCLMainAllListsSameA} For each internal vertex $u$ of $M-z$, $L(u)=A$. \end{subclaim}

\begin{claimproof} Suppose there is an internal vertex $u$ of $M-z$ such that $L(u)\neq A$. Since $|L(u)|=3$, there is a $c\in A\setminus L(u)$. Since $|L(p_1)|=3$, there is an $f\in L(p_1)$ such that $L(u_t)\setminus\{c, f\}|\geq 2$. By Claim \ref{AnyLColMidEndExtPCL11}, there is an $L$-coloring $\sigma$ of $V(P)$ which does not extend to an $L$-coloring of $G$, where $\sigma$ uses $a,c$ on the respective vertices $p_0, z$. Since $c\in A$, $\sigma$ extends to an $L$-coloring $\sigma'$ of $V(P\cup H_0)$. Since $M$ is not a triangle, $x_0u_t\not\in E(G)$, so $|L_{\sigma'}(u_t)|\geq 1$ by our choice of color for $p_1$. Letting $r\in L_{\sigma'(u_t)}$, $(\sigma'(x_{\ell_0}), c, r)$ is a proper $L$-coloring of the 2-path $x_{\ell_0}zu_t$, and since $c\not\in L(u)$, this $L$-coloring extends to $L$-color $M$ as well. Thus, $\sigma'$ extends to an $L$-coloring of $G$, so $\sigma$ extends to an $L$-coloring of $G$, a contradiction. \end{claimproof}\end{addmargin}

Now consider the set $L(z)\setminus L(u_{\ell_0})$. For any $d\in L(z)\setminus L(u_{\ell_0})$, there is an $L$-coloring $\phi$ of $\{p_0, z\}$ using $a,d$ on the respective vertices $p_0, z$, including $d=a$, since $zp_0\not\in E(G)$. It follows from 2) of Theorem \ref{EitherBWheelOrAtMostOneColThm} that, for any such $\phi$ and any extension of $\phi$ to an $L$-coloring $\phi'$ of $\{p_0, q_0, z\}$, $\phi'$ extends to an $L$-coloring of $G$. That is, we have $L(z)\setminus L(u_{\ell_0})\subseteq A$.  By Subclaim \ref{SubCLMainAllListsSameA}, $L(u_{\ell_0+1})|=A$. Since $|L(z)\setminus L(u_{\ell_0})|\geq 2$, it follows that $|L(u_{\ell_0})\cap L(u_{\ell_0+1})|\leq 1$, so there is an $r\in A$ such that $L(u_{\ell_0})\cap L(u_{\ell_0+1})\subseteq\{r\}$. Possibly $L(u_{\ell_0})\cap L(u_{\ell_0+1})=\varnothing$. As $|L(p_1)|\geq 3$, there is an $f\in L(p_1)$ with $|L(u_t)\setminus\{r, f\}|\geq 2$. By Claim \ref{AnyLColMidEndExtPCL11}, there is an $L$-coloring $\tau$ of $V(P)$ which does not extend to an $L$-coloring of $G$, where $\tau$ uses $a, r, f$ on the respective vertices $p_0, z, p_1$. Since $r\in A$, the $L$-coloring $(a, \tau(q_0), r)$ of $p_0q_0z$ extends to an $L$-coloring of $H_0$, so $\tau$ extends to an $L$-coloring $\tau'$ of $V(P\cup H_0)$. Since $x_0u_t\not\in E(G)$, it follows from our choice of $f$ that $|L_{\tau'}(u_t)|\geq 1$, so there is an $r'\in L_{\tau'}(u_t)$, and $(\tau'(u_{\ell_0}), r, r')$ is a proper $L$-coloring of $u_{\ell_0}zu_t$. Since $L(u_{\ell_0})\cap L(u_{\ell_0+1})\subseteq\{r\}$, we have $\tau'(u_{\ell_0})\not\in L(u_{\ell_0+1})$, so $\tau'$ extends to $L$-color $M$ as well, and thus $\tau$ extends to an $L$-coloring of $G$, a contradiction. This proves Claim \ref{MTriangleClaimRuleOutZChord}. \end{claimproof}

\subsection{Dealing with the case where $M$ is a triangle and $|L(x_0)\cap L(x_1)|\leq 1$}\label{TriangleAndLx0x1InserAtMost1}

In this subsection, we show that, if $M$ is a triangle, then $|L(x_0)\cap L(x_1)|\geq 2$. In Subsection \ref{MTriangleInterSectAtLeast2X0X1}, we then rule out the possibility that $M$ is a triangle and conclude that $z$ is incident to at most one chord of $C$.

\begin{Claim}\label{MTriangleIntersectionX0X1AtLeast2} If $z$ is incident to at least one chord of $C$, then either $|L(p_0)|=|L(p_1)|=2$ or $|L(x_0)\cap L(x_1)|\geq 2$. \end{Claim}

\begin{claimproof} Suppose $z$ is incident to at least one chord of $C$. By Claim \ref{MTriangleClaimRuleOutZChord}, $M$ is a triangle, so $z$ is incident to precisely two chords of $C$ and $x_0x_1$ is an edge of the path $u_1\ldots u_t$. By Claim \ref{CompressClaimsFromSec912}, $G$ contains each of the edges $q_0x_0$ and $q_1x_1$. Suppose toward a contradiction that $|L(x_0)\cap L(x_1)|<2$ and there is an endpoint of $P$ with an $L$-list of size not equal to two. Suppose without loss of generality that $|L(p_0)|=1$ and $|L(p_1)|=3$. By 2) of Claim \ref{UnivAndAlmostUnivColorClaims}, $L(p_0)\subseteq L(u_1)$ and $L(p_1)\subseteq L(u_t)$, and furthermore, $K_1$ is a broken wheel with principal path $Q_1$, where $|V(K_1)|\leq 4$. Let $a$ be the lone color of $p_0$ and $L(p_1)=\{b_0, b_1, b_2\}$. We now fix an $L$-coloring $\psi$ of $\{p_0, x_0\}$ with $\psi(p_0)=a$, where any extension of $\psi$ to an $L$-coloring of the 2-path $p_0q_0x_0$ extends to $L$-color all of $K_0$. By Theorem \ref{SumTo4For2PathColorEnds}, such a $\psi$ exists, since $|\{a\}|+|L(x_0)|=4$. We also fix an $s\in L(z)$ such that $L(q_0)\setminus\{a,s\}|\geq 4$. Such an $s$ exists, as $|L(z)|=5$. 

\vspace*{-8mm}
\begin{addmargin}[2em]{0em}
\begin{subclaim}\label{AugColorForX0DistinctS}  $\psi(x_0)\neq s$. \end{subclaim}

\begin{claimproof} Suppose toward a contradiction that $\psi(x_0)=s$. Now we simply choose a color $r\in L(z)\setminus (L(x_1)\cup\{s\})$. Such an $r$ exists, since $|L(x_1)\cup\{s\}|\leq 4$. Since $p_0z p_1z\not\in E(G)$, it follows that, for each $j=0,1,2$, there is an extension of $\psi$ to an $L$-coloring $\phi^j$ of $\{p_0, x_0, z, p_1\}$, where $\phi^j$ uses $r, b_j$ on the respective vertices $z, p_1$. Since $x_0$ is colored with $s$ in each case, we have $L_{\phi^j}(q_0)|\geq 3$ for each $j=0,1,2$. Since $x_1$ is uncolored, we have $|L_{\phi^j}(q_1)|\geq 3$ for each $j=0,1,2$. Since $G$ is a counterexample, it follows that, for each $j=0,1,2$, $\phi^j$ extends to an $L$-coloring $\tau^j$ of $V(P)\cup\{x_0\}$, where $\tau^j$ dos not extend to an $L$-coloring of $G$. For each $j=0,1,2$, since $p_0, x_0$ are already colored by $\phi^j$ and $\phi^j$ restricts to an element of $\textnormal{End}(p_0q_0x_0, K_0)$, it follows that $\Lambda_{K_1}(\bullet, \tau^j(q_0), b_j)\subseteq\{s, r\}$, or else $\tau^j$ extends to an $L$-coloring of $G$. Since $r\not\in L(x_1)$, it follows that, for each $j=0,1,2$, we have $\Lambda^{x_1q_1p_1}_{K_1}(\bullet, \tau^j(q_0), b_j)=\{s\}$. This contradicts 1c) of Theorem \ref{BWheelMainRevListThm2}. \end{claimproof}\end{addmargin}

\vspace*{-8mm}
\begin{addmargin}[2em]{0em}
\begin{subclaim}\label{ContainmentX1sNotinx0} $\{\psi(x_0), s\}\subseteq L(x_1)$ and $s\not\in L(x_0)$. In particular, $L(x_0)\cap L(x_1)=\{\psi(x_0)\}$. 
\end{subclaim}

\begin{claimproof} By Subclaim \ref{AugColorForX0DistinctS}, $\psi(x_0)\neq s$. Since $p_0z, p_1z\not\in E(G)$, it follows that, for each $j=0,1,2$, there is an extension of $\psi$ to an $L$-coloring $\sigma^j$ of $\{p_0, x_0, z, p_1\}$, where $\sigma^j$ uses $s, b_j$ on the respective vertices $z, p_1$. As above, for each $j=0,1,2$, since $s$ is used by $z$, we have $|L_{\sigma^j}(q_0)|\geq 3$, and, since $x_1$ is uncolored, we have $|L_{\sigma^j}(q_1)|\geq 3$. Since $G$ is a counterexample, it follows that, for each $j=0,1,2$, $\sigma^j$ extends to an $L$-coloring $\pi^j$ of $V(P)\cup\{x_0\}$, where $\pi^j$ does not extend to an $L$-coloring of $G$. For each $j=0,1,2$, since $p_0, x_0$ are already colored and $\pi^j$ restricts to an element of $\textnormal{End}(p_0q_0x_0, K_0)$, it follows that $\Lambda_{K_1}(\bullet, \pi^j(p_1), )\subseteq\{\psi(x_0), s\}$, or else $\pi^j$ extends to an $L$-coloring of $G$. If $s\not\in L(x_1)$, then, for each $j=01,2,$, we have $\Lambda_{K_1}(\bullet, \pi^j(q_1), b_j)=\{\psi(x_0)\}$. Since $b_0, b_1, b_2$ are distinct, this contradicts 1c) of Theorem \ref{BWheelMainRevListThm2}, as in Case 1 above. Thus, $s\in L(x_1)$. Likewise, if $\psi(x_0)\not\in L(x_1)$, then, for each $j=0,1,2$, we have $\Lambda_{K_1}(\bullet, \pi^j(q_1), b_j)=\{s\}$, which again contradicts 1c) of Theorem \ref{BWheelMainRevListThm2}. We conclude that $\{s, \psi(x_0)\}\subseteq L(x_1)$. By assumption, $|L(x_0)\cap L(x_1)|\leq 1$, so we have $L(x_0)\cap L(x_1)=\{\psi(x_0)\}$, and $s\not\in L(x_0)$.\end{claimproof}\end{addmargin}

\vspace*{-8mm}
\begin{addmargin}[2em]{0em}
\begin{subclaim} For any $\sigma\in\textnormal{End}(p_1q_1x_1, K_1)$ with $\sigma(x_1)\neq\psi(x_0)$, we have $L(q_1)\setminus\{\sigma(x_1), \sigma(p_1)\}|=3$. \end{subclaim}

\begin{claimproof} Let $\sigma\in\textnormal{End}(p_1q_1x_1, K_1)$, where $\sigma(x_1)\neq\psi(x_0)$ and suppose toward a contradiction that $L(q_1)\setminus\{\sigma(x_1), \sigma(p_1)\}|\neq 3$. Thus, $L(q_1)\setminus\{\sigma(x_1), \sigma(p_1)\}|\geq 4$. The idea now is to simply leave $x_0$ uncolored. Since $|L(z)|\geq 5$, there is an $r\in L(z)\setminus (L(x_0)\cup\{\sigma(x_1)\})$, and since $p_1x_1\not\in E(G)$, $\sigma$ extends to an $L$-coloring $\sigma'$ of $\{p_0, z, x_1, p_0\}$, where $\sigma'(z)=r$. Since $L(q_1)\setminus\{\sigma(x_1), \sigma(p_1)\}|\geq 4$, we have $|L_{\sigma'}(q_1)|\geq 3$. Since $x_0$ is uncolored, we have $|L_{\sigma'}(q_0)|\geq 3$ as well. Since $G$ is a counterexample, $\sigma'$ extends to an $L$-coloring $\tau$ of $V(P)\cup\{x_1\}$, where $\tau$ does not extend to an $L$-coloring of $G$. Since $r\not\in L(x_0)$, it follows from 2) of Theorem \ref{EitherBWheelOrAtMostOneColThm} that the $L$-coloring $(\tau(p_0), \tau(q_0), r)$ of $p_0q_0z$ extends to an $L$-coloring of $H_0$. By assumption, $\tau(x_1)=\sigma(x_1)\neq\psi(x_0)$, and since $\{\psi(x_0)\}=L(x_0)\cap L(x_1)$ by Subclaim \ref{ContainmentX1sNotinx0}, we have $\tau(x_1)\not\in L(x_0)$. Thus, $\tau$ extends ot an $L$-coloring $\tau'$ of $V(P\cup H_0)\cup\{x_1\}$. Since $\tau'$ restricts to an element of $\textnormal{End}(x_1q_1p_1, K_1)$, the $L$-coloring $(\tau'(x_1), \tau'(q_1), \tau'(p_1))$ of $x_1q_1p_1$ extends to $L$-color $K_1$, so $\tau'$ extends to an $L$-coloring of $G$. Thus, $\tau$ extends to an $L$-coloring of $G$, a contradiction. \end{claimproof}\end{addmargin}

We now let $A=\textnormal{Col}(\textnormal{End}(p_0q_0z, H_0)\mid z)$. Each element of $\textnormal{End}(p_0q_0z, H_0)$ uses $a$ on $p_0$, since $L(p_0)=\{a\}$. By Theorem \ref{SumTo4For2PathColorEnds}, we have $|A|\geq 3$. 

\vspace*{-8mm}
\begin{addmargin}[2em]{0em}
\begin{subclaim}\label{DistinctPsiX0AndASub} $\psi(x_0)\not\in A$. In particular, $A\setminus L(x_1)\neq\varnothing$. \end{subclaim}

\begin{claimproof} Suppose that $\psi(x_0)\in A$. Since $|L(p_0)|=3$, it follows from Theorem \ref{SumTo4For2PathColorEnds} that there is a $d\in L(p_1)$ such that any $L$-coloring of $\{z, p_1\}$ using $\psi(x_0), d$ on the respective vertices $z, p_1$ extends to an $L$-coloring of $H_1$. By Claim \ref{AnyLColMidEndExtPCL11}, there is an $L$-coloring $\tau$ of $V(P)$ which does not extend to an $L$-coloring of $G$, where $\tau$ uses $a, \psi(x_0), d$ on the respective vertices $p_0, z, p_1$. Since $\psi(x_0)\in A$, the $L$-coloring $(\tau(p_0), \tau(q_0), \tau(z))$ of $p_0q_0z$ extends to an $L$-coloring $\tau_0$ of $V(H_0)$. Likewise, by our choice of color for $p_1$, the $L$-coloring $(\tau(z), \tau(q_1), \tau(p_1))$ of $zq_1p_1$ extend to an $L$-coloring $\tau_1$ of $V(H_1)$. By Subclaim \ref{ContainmentX1sNotinx0}, we have $L(x_0)\cap L(x_1)=\{\psi(x_0)\}$. Since $z$ is colored by $\psi(x_0)$, it follows that $\tau_0(x_0)\neq\tau_1(x_1)$, so the union $\tau_0\cup\tau_1$ is an extension of $\tau$ to a proper $L$-coloring of $G$, contradicting the fact that $\tau$ does not extend to an $L$-coloring of $G$. We conclude that $\psi(x_0)\not\in A$. Since $\psi(x_0)\in L(x_1)$, we have $A\neq L(x_1)$. Since $|L(x_1)|=3$ and $|A|\geq 3$, we have $A\setminus L(x_1)\neq\varnothing$. \end{claimproof}\end{addmargin}

\vspace*{-8mm}
\begin{addmargin}[2em]{0em}
\begin{subclaim}\label{H1NotTriangleMTriangleSubClS}  $|V(K_1)|=4$. \end{subclaim}

\begin{claimproof} Suppose not. Thus, $K_1$ is a triangle, and $x_1=u_t$. By Subclaim \ref{ContainmentX1sNotinx0}, $s\in L(x_1)$. Now consider the following cases:

\textbf{Case 1:} $s\in L(p_1)$

In this case, we let $f$ be the lone color of $L(x_1)\setminus\{s, \psi(x_0)\}$. Since $p_1z\not\in E(G)$, there is an $L$-coloring $\sigma$ of $\{p_0, z, x_0, x_1, p_1\}$ which uses $s$ on both of $z, p_1$ and colors the respective vertices $p_0, x_0, x_1$ with $a, \psi(x_0), f$. Since the same color is used on $z, p_1$, we have $|L_{\sigma}(q_1)|\geq 3$. By our choice of $s$, we have $|L_{\sigma}(q_0)|\geq 3$ as well. Since $\sigma$ restricts to an element of $\textnormal{End}(p_0q_0z, K_0)$, it follows that $\sigma$ is $(P, G)$-sufficient, contradicting our assumption that $G$ is a counterexample. 

\textbf{Case 2:} $s\not\in L(p_1)$

Since $s\not\in L(p_1)$ by assumption, and $|L(p_1)|=3$, it follows that there is a $d\in L(p_1)\setminus L(u_t)$. By Subclaim \ref{DistinctPsiX0AndASub}, there is an $r\in A\setminus L(x_1)$. By Claim \ref{AnyLColMidEndExtPCL11}, there is an $L$-coloring $\tau$ of $V(P)$ which does not extend to an $L$-coloring of $G$, where $\tau$ uses the colors $a, r, d$ on the respective vertices $p_0, z, p_1$. Since $r\in A$ and $q_1, p_1$ have no neighbors in $H_0-z$, $\tau$ extends to an $L$-coloring $\tau'$ of $V(P)\cup\{q_1, p_1\}$. Since $H_1$ is a triangle and $r,d\not\in L(u_t)$, we have $|L_{\tau'}(u_t)|\geq 1$, so $\tau'$ extends to an $L$-coloring of $G$, a contradiction.  \end{claimproof}\end{addmargin}

\vspace*{-8mm}
\begin{addmargin}[2em]{0em}
\begin{subclaim}\label{InAddForSizeK134Sub} $|L(x_1)\cap L(p_1)|\leq 1$ \end{subclaim}

\begin{claimproof} Suppose toward a contradiction that $|L(x_1)\cap L(p_1)|\geq 2$. By Subclaim \ref{H1NotTriangleMTriangleSubClS}, $|V(K_1)|=4$ and $K_1-q_1=u_{t-1}u_tp_1$. In particular, $x_1=u_{t-1}$ and $p_1x_1\not\in E(G)$. Since $|L(x_1)\cap L(p_1)|\geq 2$, there is an $r\in L(x_1)\cap L(p_1)$ with $r\neq\psi(x_0)$. By Subclaim \ref{DistinctPsiX0AndASub}, there is an $r'\in A\setminus L(x_1)$. In particular, $r'\neq r$, and, since $p_1x_1\not\in E(G)$, there is an $L$-coloring $\sigma$ of $\{p_0, z, x_1, p_1\}$, where $\sigma$ uses $a, r'$ on the respective vertices $p_0, z$ and uses $r$ on both of $x_1, p_1$. Since $x_0$ is uncolored, we have $|L_{\sigma}(q_0)|\geq 3$, and since $\sigma$ uses the same color on $x_1, p_1$, we have $|L_{\sigma}(q_1)|\geq 3$. Since $G$ is a counterexample, $\sigma$ extends to an $L$-coloring $\sigma'$ of $V(P)\cup\{x_1\}$, where $\sigma'$ does not extend to an $L$-coloring of $G$. Since $K_1-q_1$ is a 2-path and $\sigma$ uses the same color on $p_1, x_1$, it follows that $\sigma'$ extends to an $L$-coloring $\sigma'_1$ of $V(P\cup H_1)$. Likewise, since $r'\in A$, the $L$-coloring $(\sigma'(p_0), \sigma'(q_0), \sigma'(z))$ of $p_0q_0z$ extends to an $L$-coloring $\sigma'_0$ of $H_0$. By Subclaim \ref{ContainmentX1sNotinx0}, $L(x_0)\cap L(x_1)=\{\psi(x_0)\}$. Since $x_1$ is colored with $r$ and $r\neq\psi(x_0)$, it follows that $\sigma'_0(x_0)\neq r$, so $\sigma'_0\cup\sigma'_1$ is a proper $L$-coloring of its domain, which is $G$, contradicting the fact that $\sigma'$ does not extend to an $L$-coloring of $G$. \end{claimproof}\end{addmargin}

Recall now that, since $|L(p_1)|=3$ and $L(u_t)\subseteq L(p_1)$, we have $L(p_1)=L(u_t)$. Since $|V(K_1)|=4$, so $x_1=u_{t-1}$, and, by Subclaim \ref{InAddForSizeK134Sub}, $|L(x_1)\cap L(p_1)|\leq 1$. As $L(p_1)=L(u_t)$, we have $|L(u_t)\cap L(u_{t-1})|\leq 1$, and there is a $c\in L(p_1)$ with $(L(u_t)\setminus\{c\})\cap L(u_{t-1})=\varnothing$. Thus, any $L$-coloring of $\{x_1, q_1, p_1\}$ using $c$ on $p_1$ extends to an $L$-coloring of $K_1$. Since $K_1$ is not a triangle, and $p_1x_1\not\in E(G)$, $c$ is a $K_1$-universal color of $L(p_1)$, contradicting 2) of Claim \ref{UnivAndAlmostUnivColorClaims}. This completes the proof of Claim \ref{MTriangleIntersectionX0X1AtLeast2}. \end{claimproof}

With Claim \ref{MTriangleIntersectionX0X1AtLeast2} in hand, we now entirely rule out the possibility that $L(x_0)$ and $L(x_1)$ share fewer than two colors. We just need to deal with the case where $|L(p_0)|=|L(p_1)|=2$. This is easier than the case above.

\begin{Claim}\label{22CaseForp0p1x0x1InsersectAtLeast2} If $z$ is incident to at least one chord of $C$, then, $|L(x_0)\cap L(x_1)|\geq 2$. \end{Claim}

\begin{claimproof} Suppose that $z$ is incident to at least one chord of $C$. By Claim \ref{MTriangleClaimRuleOutZChord}, $M$ is a triangle, so $z$ is incident to precisely two chords of $C$ and $x_0x_1$ is an edge of the path $u_1\ldots u_t$. Suppose toward a contradiction that $|L(x_0)\cap L(x_1)|<2$. By Claim \ref{MTriangleIntersectionX0X1AtLeast2}, we have $|L(x_0)|=|L(x_1)|=2$. By Claim \ref{CompressClaimsFromSec912}, $G$ contains each of the edges $q_0x_0$ and $q_1x_1$. By 2) of Claim \ref{UnivAndAlmostUnivColorClaims}, we have $L(p_0)\subseteq L(u_1)$ and $L(p_1)\subseteq L(u_t)$, and for each $i=0,1$, $K_i$ is a  broken wheel with principal path $Q_i$, where $|V(K_i)|\leq 4$, so $H_i$ is also broken wheel. In particular, $V(G)=V(C)$, i.e $G$ is outerplanar, and $|V(G)|\leq 9$. Since $|L(p_0)|=|L(p_1)|=2$, we let $L(p_0)=\{a_0, a_1\}$ and $L(p_1)=\{b_0, b_1\}$.

\vspace*{-8mm}
\begin{addmargin}[2em]{0em}
\begin{subclaim}\label{K0K1TriangleAndS0S1InX0X1}
 For each $i=0,1$, $K_i$ is a triangle and $L(p_i)=L(x_i)\setminus L(x_{1-i})$.
\end{subclaim}

\begin{claimproof} Since $|L(x_0)\cap L(x_1)|<2$, there exist sets $S_0, S_1$ where $S_0\cap S_1=\varnothing$ and, for each $i=0,1$, $S_i\subseteq L(x_i)$ and $|S_i|=2$. Since $|S_0|=|S_1|=2$, there exists an $s\in L(z)\setminus (S_0\cup S_1)$. Since $p_0z, p_1z\not\in E(G)$, it follows from Claim \ref{AnyLColMidEndExtPCL11} that, for each pair of indices $0\leq j,k\leq 1$, there is an $L$-coloring $\phi^{jk}$ of $V(P)$ which does not extend to an $L$-coloring of $G$, where $\phi^{jk}$ uses $a_j, s, b_k$ on the respective vertices $p_0, z, p_1$. Now,, $M$ is a triangle and $|L(x_0)\cap L(x_1)|<2$. Since $s\not\in S_0\cup S_1$ and $|L(x_0)|=|L(x_1)|=3$, it follows that there is a lone color $r\in L(x_0)\cap L(x_1)$ (where possibly $r=s$) such that the following holds.
\begin{equation}\tag{$\dagger$}\Lambda_{K_0}(a_j, \phi^{jk}(q_0), \bullet)=\Lambda_{K_1}(\bullet, \phi^{jk}(q_1), b_k)=\{r\}\ \textnormal{for each}\ 0\leq j,k\leq 1\end{equation}
It now follows from 1c) of Theorem \ref{BWheelMainRevListThm2} applied to each of the broken wheels $K_0$ and $K_1$ that, for each pair $0\leq j,k\leq 1$, we have $\phi^{jk}(q)=a_{1-j}$ and $\phi^{jk}(q_1)=b_{1-k}$. In particular, letting $\ell\in\{1, \ldots, t\}$ be the unique index such that $x_0=u_{\ell}$ and $x_1=u_{\ell+1}$, it follows from $(\dagger)$ that $|\Lambda_{K_0}(a_j, a_{1-j}, \bullet)|=|\Lambda_{K_1}(\bullet, b_{1-k}, b_k)|=1$ for each $0\leq j,k\leq 1$, and, in particular, we get both of the following containments.
\begin{center}
$\{a_0, a_1\}\subseteq L(u_1)\cap\ldots\cap L(u_{\ell-1})\cap (L(u_{\ell})\setminus\{r\})$\\
$\{b_0, b_1\}\subseteq (L(u_{\ell+1})\setminus\{r\})\cap L(u_{\ell_2})\cap\ldots L(u_t)$
\end{center}
Since $\{a_0, a_1\}=L(p_0)$ and $\{b_0, b_1\}=L(p_1)$, it follows that each of $K_0$ and $K_1$ is a triangle, or else we contradict Claim \ref{KiBWheelSizeG}. Furthermore, for each $i=0,1$, we have $L(x_i)=L(p_i)\cup\{r\}$ as a disjoint union, and $S_i=L(p_i)$. In particular, since $|L(x_0)\cap L(x_1)|\leq 1$, we have $L(x_0)\cap L(x_1)=\{r\}$, so $L(p_0)\cap L(p_1)=\varnothing$ and, for each $i=0,1$, we have $L(p_i)=L(x_i)\setminus L(x_{1-i})$. This proves Subclaim \ref{K0K1TriangleAndS0S1InX0X1}. \end{claimproof}\end{addmargin}

As $M$ is a triangle, Subclaim \ref{K0K1TriangleAndS0S1InX0X1} implies that $t=2$ and $|V(G)|=7$, so $G$ consists of the graph in Figure \ref{FigureFor22CaseMK0K1Triangles}. 

\begin{center}\begin{tikzpicture}
\node[shape=circle,draw=black] (p0) at (0,0) {$p_0$};
\node[shape=circle,draw=black] (u1) at (2, 0) {$u_1$};
\node[shape=circle,draw=black]  (u2) at (4, 0) {$u_2$};
\node[shape=circle,draw=black]  (p1) at (6, 0) {$p_1$};
\node[shape=circle,draw=black] (q0) at (0,2) {$q_0$};
\node[shape=circle,draw=black] (q1) at (6,2) {$q_1$};
\node[shape=circle,draw=black] (z) at (3,3) {$z$};

 \draw[-] (p0) to (q0) to (z) to (q1) to (p1) to (u2) to (u1) to (p0);
\draw[-] (q0) to (u1) to (z) to (u2) to (q1);
\end{tikzpicture}\captionof{figure}{}\label{FigureFor22CaseMK0K1Triangles}\end{center}

\vspace*{-8mm}
\begin{addmargin}[2em]{0em}
\begin{subclaim}\label{P0P1ConShareNoColorsZSub} For each $i=0,1$, $L(z)\cap L(p_i)=\varnothing$. \end{subclaim}

\begin{claimproof} Suppose not, and suppose toward a contradiction that $L(z)\cap L(p_0)\neq\varnothing$, and, in particular, $a_0\in L(z)\cap L(p_0)$. Let $b\in L(p_1)$. By Subclaim \ref{K0K1TriangleAndS0S1InX0X1}, $a_1\in L(u_1)$. Let $\phi$ be an $L$-coloring of $V(G)\setminus\{q_0, q_1\}$, where $\phi$ uses $a_1, b$ on the respective vertices $u_1, p_1$ and $\phi$ colors both of $p_0$ and $z$ with $a_0$. Since $u_2$ is uncolored, we have $|L_{\phi}(q_1)|\geq 3$. Since $\phi$ uses the same color on $z$ and $p_0$, we have $|L_{\phi}(q_0)|\geq 3$. By Subclaim \ref{K0K1TriangleAndS0S1InX0X1}, neither $a_0$ nor $a_1$ lies in $L(u_2)$, so $|L_{\phi}(u_2)|\geq 2$. Thus, any extension of $\phi$ to an $L$-coloring of $V(P)\cup\{u_1\}$ extends to an $L$-coloring of $G$, contradicting our assumption that $G$ is a counterexample.  \end{claimproof}\end{addmargin}

\vspace*{-8mm}
\begin{addmargin}[2em]{0em}
\begin{subclaim}\label{Q0Q1ContainColorsWithZSubCL} For each $i=0,1$,  $L(p_i)\subseteq L(q_i)$ and, in particular, $|L(z)\setminus L(q_i)|\geq 2$. \end{subclaim}

\begin{claimproof} Suppose there is an $i\in\{0,1\}$ and an $a\in L(p_i)$ with $|L(q_i)\setminus\{a\}|\geq 5$, say $i=0$ without loss of generality, and suppose (again without loss of generality) that $|L(q_0)\setminus\{a_0\}|\geq 5$. By Subclaim \ref{K0K1TriangleAndS0S1InX0X1}, $a_1\in L(u_1)\setminus L(u_2)$. Since $|L(z)=5$, there is an $s\in L(z)\setminus (\{a_1\}\cup L(u_2))$. Let $\phi$ be an $L$-coloring of $\{p_0, u_1, z, p_1\}$, where $\phi$ uses $a_0, a_1, s$ on the respective vertices $p_0, u_1, z$. Since $|L(q_0)\setminus\{a_0\}|\geq 5$, we have $|L_{\phi}(q_0)|\geq 3$. Since $u_2$ is uncolored, we have $|L_{\phi}(q_1)|\geq 3$. Since $a_1\not\in L(u_2)$,  it follows from our choice of $s$ that $|L_{\phi}(u_2)|\geq 2$. Thus, any extension of $\phi$ to an $L$-coloring of $V(P)\cup\{u_1\}$ extends to an $L$-coloring of $G$, contradicting our assumption that $G$ is a counterexample. We conclude that, for each $i=0,1$, $L(p_i)\subseteq L(q_i)$ and $|L(q_i)|=5$. Since $L(p_i)\cap L(z)=\varnothing$ by Subclaim \ref{P0P1ConShareNoColorsZSub}, we have $|L(z)\setminus L(q_i)|\geq 2$ for each $i=0,1$. \end{claimproof}\end{addmargin}

By Subclaim \ref{P0P1ConShareNoColorsZSub}, we have $|L(z)\setminus L(q_0)|\geq 2$. Since $|L(x_0)\cap L(x_1)|\leq 1$ by assumption, there is an $s\in L(z)\setminus L(q_0)$ with $s\not\in L(x_0)\cap L(x_1)$. By Subclaim \ref{K0K1TriangleAndS0S1InX0X1}, $L(p_1)=L(x_1)\setminus L(x_0)$ and, by Subclaim \ref{P0P1ConShareNoColorsZSub}, $L(p_1)\cap L(z)=\varnothing$, so $s\not\in L(x_1)$. Likewise, since $s\not\in L(q_0)$, we have $s\not\in L(p_0)$ by Subclaim \ref{Q0Q1ContainColorsWithZSubCL}, so $s\not\in\{a_0, a_1\}$. Now we simply let $\phi$ be an $L$-coloring of $\{p_0, u_1, z, p_1\}$, where $\phi$ uses $a_0, a_1, s$ on the respective vertices $p_0, u_1, s$. By our choice of color for $z$, we have $|L_{\phi}(q_0)|\geq 3$. Since $u_2$ is uncolored, we have $|L_{\phi}(q_1)|\geq 3$ as well. Since neither $a_1$ nor $s$ lies in $u_2$ ($=x_1$), we have $|L_{\phi}(u_2)|\geq 2$, so any extension of $\phi$ to an $L$-coloring of $V(P)\cup\{u_1\}$ extend to an $L$-coloring of $G$, contradicting our assumption that $G$ is a counterexample. \end{claimproof}

\subsection{Dealing with the case where $M$ is a triangle and $|L(x_0)\cap L(x_1)|\geq 2$}\label{MTriangleInterSectAtLeast2X0X1}

With Claim \ref{22CaseForp0p1x0x1InsersectAtLeast2} in hand, we can now rule out the possibility that there are any chords of $C$ incident to $z$. 

\begin{Claim}\label{IfZTwoChordsContainmentBothTriangleK0K1} If $z$ is incident to at least one chord of $C$, then all three of A), B), and C) hold.
\begin{enumerate}[label=\Alph*)] 
\itemsep-0.1em
\item For each $i\in\{0,1\}$ with $|L(p_i)|\geq 2$, $K_i$ is a triangle; AND
\item For each $i\in\{0,1\}$, $L(p_i)\cap L(z)\subseteq L(q_{1-i})\setminus L(p_{1-i})$; 
\item $|L(p_0)|=|L(p_1)|=2$. 
\end{enumerate}
\end{Claim}

\begin{claimproof} By Claim \ref{CompressClaimsFromSec912}, $G$ contains each of the edges $q_0x_0$ and $q_1x_1$. By Claim \ref{MTriangleClaimRuleOutZChord}, $M$ is a triangle, so $z$ is incident to precisely two chords of $C$ and $x_0x_1$ is an edge of the path $u_1\ldots u_t$. By Claim \ref{22CaseForp0p1x0x1InsersectAtLeast2}, we have $|L(x_0)\cap L(x_1)|\geq 2$. Since $|L(z)|=5$ and $|L(x_0)\cap L(x_1)|\geq 2$, we have $L(z)\setminus (L(x_0)\cup L(x_1))\neq\varnothing$. By 2) of Claim \ref{UnivAndAlmostUnivColorClaims}, for each $i\in\{0,1\}$, no color is $K_i$-universal and $L(p_0)\subseteq L(u_1)$ and $L(p_1)\subseteq L(u_t)$.  

Now we prove A). Let $i\in\{0,1\}$ with $|L(p_i)|\geq 2$, say $i=1$ without loss of generality. Again by 2) of Claim \ref{UnivAndAlmostUnivColorClaims}, $K_i$ is a broken wheel with principal path $Q_i$ and $|V(K_i)\leq 4$. We just need to show that $|V(K_1)|\neq 4$. Suppose toward a contradiction that $|V(K_1)|=4$. By Claim \ref{KiBWheelSizeG}, we have $L(p_1)\not\subseteq L(u_{t-1})$. Since $L(p_1)\subseteq L(u_t)$ and $|L(p_1)|\geq 2$, there is a $c\in L(p_1)$ with $L(u_t)\setminus\{c\}\not\subseteq L(u_{t-1})$. Since $L(z)\setminus (L(x_0)\cup L(x_1))\neq\varnothing$, there is an $s\in L(z)$ with $s\not\in L(x_0)\cup L(x_1)$. By Claim \ref{AnyLColMidEndExtPCL11}, there is an $L$-coloring $\tau$ of $V(P)$ which does not extend to an $L$-coloring of $G$, where $\tau$ uses $s, c$ on the respective vertices $z, p_1$. Since $s\not\in L(x_0)$, it follows from 2) of Theorem \ref{EitherBWheelOrAtMostOneColThm} that the $L$-coloring $(\tau(p_0), \tau(q_0), \tau(z))$ of $p_0q_0z$ extends to an $L$-coloring $\tau_0(z)$ of $H_0$. Let $f=\tau_0(x_0)$. Now, $s\not\in L(x_1)$, and since $M$ is a triangle and $\tau$ does not extend to an $L$-coloring of $H_1$, we have $\Lambda_{K_1}(\bullet, \tau(q_0), c)=\{f\}$. Since $|V(K_1)|=4$, we have $\tau(q_0)\in L(u_{t-1})\cap L(u_t)$ and $f\in L(u_{t-1})$. Thus, $L(u_t)=\{c, f, \tau(q_0)\}$. But since $L(u_t)\setminus\{c\}\not\subseteq L(u_{t-1})$ and $\tau(q_0)\in L(u_{t-1})$, we have $f\not\in L(u_{t-1})$, which is false, since $\Lambda_{K_1}(\bullet, \tau(q_0), c)=\{f\}$. This proves A) of Claim \ref{IfZTwoChordsContainmentBothTriangleK0K1}. Now we prove B) and C) together. We first prove the following intermediate result.

\vspace*{-8mm}
\begin{addmargin}[2em]{0em}
\begin{subclaim}\label{ForRepBSize7Intermed} Either $|V(G)|>7$ or, for each $i\in\{0,1\}$, $L(p_i)\cap L(z)\subseteq L(q_{1-i})\setminus L(p_{1-i})$ \end{subclaim}

\begin{claimproof} Suppose toward a contradiction that Subclaim \ref{ForRepBSize7Intermed} does not hold. In particular, $|V(G)|\leq 7$. Since $|V(G)|\leq 7$, each of $K_0$ and $K_1$ is a triangle, and $|V(G)|=7$. Thus, $x_0=u_1$ and $x_1=u_t=u_2$, and we have the same graph as in Figure \ref{FigureFor22CaseMK0K1Triangles}, and $V(C)=V(G)$. Consider the following cases.

\textbf{Case 1:}  $L(p_0)\cap L(z)\cap L(p_1)\neq\varnothing$ 

In this case, we let $a$ be a color of $L(p_0)\cap L(z)\cap L(p_1)$. Since $\{p_0, z, p_1\}$ is an independent subset of $V(G)$, there is an $L$-coloring $\phi$ of $\{p_0, z, p_1\}$ using $a$ on all three of $p_0, z, p_1$. Furthermore, $|L_{\phi}(u_1)|\geq 2$ and $|L_{\phi}(u_2)|\geq 2$, since the same color is used on all of $p_0, z, p_1$. Thus, $\phi$ extends to an $L$-coloring $\phi^*$ of $V(C)\setminus\{q_0, q_1\}$. Since the same color is used on all of $p_0, z, p_1$, we have $|L_{\phi^*}(q_0)|\geq 3$ and $|L_{\phi^*}(q_1)|\geq 3$. Since $V(G)=V(C)$ and $\phi^*$ already colors all of $V(G)\setminus\{q_0, q_1\}$, we contradict our assumption that $G$ is a counterexample.

\textbf{Case 2:}  $L(p_0)\cap L(z)\cap L(p_1)=\varnothing$

In this case, since Subclaim \ref{ForRepBSize7Intermed} does not hold, there is an $i\in\{0,1\}$ and an $a\in L(p_i)\cap L(z)$ with $a\not\in L(q_{1-i})$, say $i=0$ without loss of generality. Since $p_0z\not\in E(G)$, there is an $L$-coloring $\phi$ of $V(C)\setminus\{q_0, q_1\}$ which uses $a$ on both of the $p_0, z$ (note that this is true even if $|L(p_1)|=1$). Since $p_0, z$ are colored with the same color, we have $|L_{\phi}(q_0)|\geq 3$, and since $a\not\in L(q_1)$, we have $|L_{\phi}(q_1)|\geq 3$ as well. Since $V(G)=V(C)$ and $\phi$ already colors all of $V(G)\setminus\{q_0, q_1\}$, we contradict our assumption that $G$ is a counterexample. \end{claimproof}\end{addmargin}

Now, if $|L(p_0)|=|L(p_1)|=2$, then, by A), each of $K_0$ and $K_1$ is a triangle and thus $|V(G)|=7$. If we prove C), then, by Subclaim \ref{ForRepBSize7Intermed}, B) also holds, so to finish the proof of Claim \ref{IfZTwoChordsContainmentBothTriangleK0K1}, it suffices to prove that C) holds. Suppose toward a contradiction that C) does not hold, and suppose without loss of generality that $|L(p_0)|=1$ and $|L(p_1)|=3$. By A), $K_1$ is a triangle, and, in particular, $x_0=u_{t-1}$ and $x_1=u_t$ and $L(p_1)=L(u_t)$. Let $a$ be the lone color of $p_0$ and let $L(p_1)=\{b_0, b_1, b_2\}$ for some colors $b_0, b_1, b_2$. Since $L(z)\setminus (L(x_0)\cup L(x_1))\neq\varnothing$, there is an $s\in L(z)\setminus (L(x_0)\cup L(x_1))$. By Claim \ref{AnyLColMidEndExtPCL11}, for each $k=0,1,2$, there is an $L$-coloring $\psi^k$ of $V(P)$ which does not extend to an $L$-coloring of $G$, where $\psi^k$ uses $a, s, b_k$ on the respective vertices $p_0, z, p_1$. Since each of $M$ and $K_1$ is a triangle and $s\not\in L(x_0)\cup L(x_1)$, we immediately have the following.

\vspace*{-8mm}
\begin{addmargin}[2em]{0em}
\begin{subclaim}\label{ForEachKColorMaptoOneLambda} For each $k\in\{0,1,2\}$, $|\Lambda_{K_0}(a, \psi^k(q_0), \bullet)|=1$ and furthermore, $\Lambda_{K_0}(a, \psi^k(q_0), \bullet)=L(u_t)\setminus\{b_k, \psi^k(p_1)\}$. \end{subclaim}\end{addmargin}

By 2) of Corollary \ref{CorMainEitherBWheelAtM1ColCor}, $K_0$ is a broken wheen with principal path $p_0q_0x_0$. Since each of $\psi^0, \psi^1, \psi^2$ uses the same color on $p_0$, it follows from 2) of Theorem \ref{BWheelMainRevListThm2} that there exist distinct $j,k\in\{0,1,2\}$ with $\psi^j(q_0)=\psi^k(p)$, or else there exists a $k\in\{0,1,2\}$ such that $|\Lambda_{K_0}(a, \psi^k(q_0), \bullet)|>1$ and we contradict Subclaim \ref{ForEachKColorMaptoOneLambda} above. Thus, we suppose without loss of generality that $\psi^0(q_0)=\psi^1(q_0)=r$ for some color $r$ and let $r'$ be the lone color of $\Lambda_{K_0}(a, r, \bullet)$. We have $\{r'\}=L(u_t)\setminus\{b_0, \psi^0(q_1)\}=\{b_1, \psi^1(q_0)\}$, and since $L(u_t)=L(p_1)$, we have $r'=b_2$, and furthermore, $\psi^0(q_1)=b_1$ and $\psi^1(q_1)=b_0$. 

\vspace*{-8mm}
\begin{addmargin}[2em]{0em}
\begin{subclaim}\label{ColorContainmentU1Ut-1SubCLMTri} $L(u_1)=\{a, b_0, b_1\}$ and $L(u_{t-1})=\{b_0, b_1, b_2\}$. \end{subclaim}

\begin{claimproof} As shown above, for each $k=0,1$, we have $\psi^k(q_0)=b_{1-k}$, and since $|\Lambda_{K_0}(a, \psi^k(q_0), \bullet)|=1$ for each $k=0,1$, it follows that $L(u_1)=\{a, b_0, b_1\}$ and, in particular, $\{b_0, b_1\}\subseteq (L(u_1)\setminus\{a\})\cap L(u_2)\cap\ldots\cap L(u_{t-1})$. Since $r'=b_2$, we have $L(u_{t-1})=\{b_0, b_1, b_2\}$. \end{claimproof}\end{addmargin}

Now consider $\psi^2$. By Subclaim \ref{ColorContainmentU1Ut-1SubCLMTri}, we have $L(u_{t-1})=\{b_0, b_1, b_2\}$, i.e $L(u_{t-1})=L(u_t)=L(p_1)$. Since $\psi^2(p_1)=b_2$, it follows from Subclaim \ref{ForEachKColorMaptoOneLambda} that there is a $k\in\{0,1\}$ such that $\Lambda_{K_0}(a, \psi^2(q_0), \bullet)=\{b_k\}$. Since $r'=b_2$, we have $\Lambda_{K_0}(a, r, \bullet)=\{b_2\}$, so $r\neq\psi^2(q_0)$. Applying 1b) of Theorem \ref{BWheelMainRevListThm2} to the two $L$-colorings $(a, r)$ and $(a, \psi^2(q_0))$ of the edge $p_0q_0$, it follows that $\psi^2(q_0)=b_2$. Since $\psi^2(q_0)=b_2$ and $\psi^2(q_0)\neq a$, it follows from Subclaim \ref{ColorContainmentU1Ut-1SubCLMTri} that $\psi^2(q_0)\not\in L(u_1)$, and thus $|\Lambda_{K_0}(a, \psi^2(q_0), \bullet)|\geq 2$, which is false, since $\Lambda_{K_0}(a, \psi^2(q_0), \bullet)=\{b_k\}$, as indicated above. This proves C) and completes the proof of Claim \ref{IfZTwoChordsContainmentBothTriangleK0K1}.  \end{claimproof}

With Claim \ref{IfZTwoChordsContainmentBothTriangleK0K1} in hand, we now finally rule out the possibility that $z$ is incident to a chord of $C$. 

\begin{Claim}\label{noChordzOrAtLeast2Rule1Out} $z$ is incident to no chords of $C$. \end{Claim}

\begin{claimproof} Suppose toward a contradiction that $z$ is incident to a chord of $C$. By Claim \ref{CompressClaimsFromSec912}, $G$ contains each of the edges $q_0x_0$ and $x_1q_1$. By Claim \ref{MTriangleClaimRuleOutZChord}, $M$ is a triangle, so $z$ is incident to precisely two chords of $C$ and $x_0x_1$ is an edge of $u_1\ldots u_t$. By Claim \ref{22CaseForp0p1x0x1InsersectAtLeast2}, $|L(x_0)\cap L(x_1)|\geq 2$. Since $|L(z)|=5$ and $|L(x_0)\cap L(x_1)|\geq 2$, we have $L(z)\setminus (L(x_0)\cup L(x_1))\neq\varnothing$. By C) of Claim \ref{IfZTwoChordsContainmentBothTriangleK0K1}, $|L(p_0)|=|L(p_1)|=2$. By A) of Claim \ref{IfZTwoChordsContainmentBothTriangleK0K1}, each of $K_0, K_1$ is a triangle. In particular, $|V(G)|=7$, and we again have the same graph as in Figure \ref{FigureFor22CaseMK0K1Triangles}, so $x_0=u_1$ and $x_1=u_2$.

\vspace*{-8mm}
\begin{addmargin}[2em]{0em}
\begin{subclaim}\label{EitherContainFromPtoQOrListSize3Sub} For each $i=0,1$ $L(p_i)\subseteq L(q_i)$ and $|L(q_i)|=5$. \end{subclaim}

\begin{claimproof} Suppose not, and suppose without loss of generality that there is a $c\in L(p_0)$ with $|L(q_0)\setminus\{c\}|\geq 5$. No generality is lost in this assumption, since $|L(p_0)|=|L(p_1)|=2$. Since $|L(u_2)|=3$, there is an $f\in L(u_2)\setminus L(p_1)$. Let $\phi$ be an arbitrary $L$-coloring of the 2-path $p_0x_0x_1$, where $\phi(p_0)=c$ and $\phi(x_1)=f$. Such a $\phi$ exists, since $|L(x_0)|=3$. By our choice of $f$, we have $|L_{\phi}(p_1)|\geq 2$, and we also have $|L_{\phi}(z)|\geq 3$. Since $|L_{\phi}(q_1)|\geq 4$ and $|L_{\phi}(z)|+|L_{\phi}(p_1)|\geq 5$. Thus, $\phi$ extends to an $L$-coloring $\phi^*$ of $V(C)\setminus\{q_0, q_1\}$ such that $|L_{\phi^*}(q_1)|\geq 3$. By our choice of color $c$, we also have $|L_{\phi^*}(q_1)|\geq 3$. Since $\phi^*$ already colors $V(G)\setminus\{q_0, q_1\}$, this contradicts our assumption that $G$ is a counterexample. \end{claimproof}\end{addmargin}

\vspace*{-8mm}
\begin{addmargin}[2em]{0em}
\begin{subclaim}\label{xicontainedinqiForEachi=01} For each $i=0,1$, $L(x_i)\subseteq L(q_i)$. \end{subclaim}

\begin{claimproof} Suppose not and suppose without loss of generality that there is an $s\in L(x_0)\setminus L(q_0)$. Now consider the following three cases.

\textbf{Case 1:} $L(z)\cap L(p_1)=\varnothing$

By Subclaim \ref{EitherContainFromPtoQOrListSize3Sub}, $|L(q_1)|=|L(z)|=5$, and, since $L(p_1)\subseteq L(q_1)$, we have $|L(z)\setminus L(q_1)|\geq 2$, so there is an $r\in L(z)\setminus L(q_1)$ with $r\neq s$. Since $|L(x_1)|=3$, there is an $L$-coloring $\phi$ of the triangle $M$ using $r,s$ on the vertices $x_0, z$. Since $|L(p_0)|=|L(p_1)|=2$ and $\phi(x_1)\neq r$, there is an extension of $\phi$ to an $L$-coloring $\phi^*$ of $V(C)\setminus\{q_0, q_1\}$, where $\phi^*(p_1)=r$. By our choice of $s$, we have $|L_{\phi^*}(p_1)|\geq 3$. By our choice of $r$, we have $|L_{\phi^*}(q_1)|\geq 3$ as well. Since $\phi^*$ is already an $L$-coloring of $G\setminus\{q_0, q_1\}$, we contradict our assumption that $G$ is a counterexample.  

\textbf{Case 2:} $L(z)\cap L(p_1)=\{s\}$

In this case, we color $z$ and $p_1$ with $s$. Since $|L(p_0)|=2$, there is an $L$-coloring $\psi$ of $V(C)\setminus\{q_0, q_1\}$ using $s$ on each of $z, p_1$. Thus, $|L_{\psi}(q_1)|\geq 3$. By our choice of $s$, $|L_{\psi}(q_0)|\geq 3$ as well. Since $\psi$ is already an $L$-coloring of $G\setminus\{q_0, q_1\}$, we contradict our assumption that $G$ is a counterexample.  

\textbf{Case 3:} $L(z)\cap L(p_1)\not\subseteq\{s\}$

In this case, we let $r\in L(z)\cap L(p_1)$, where $r\neq s$. Since $|L(p_0)|=2$ and $|L(x_1)\setminus\{r, s\}|\geq 1$, there is an $L$-coloring $\psi$ of $V(C)\setminus\{q_0, q_1\}$ which uses $r$ on each of $z, p_1$ and uses $s$ on $x_0$. As above, since $p_1, z$ are colored with the same color, we have $|L_{\psi}(q_1)|\geq 3$, and, by our choice of $s$, we have $|L_{\psi}(q_0)|\geq 3$ as well.  Since $\psi$ is already an $L$-coloring of $G\setminus\{q_0, q_1\}$, we contradict our assumption that $G$ is a counterexample.  \end{claimproof}\end{addmargin}

\vspace*{-8mm}
\begin{addmargin}[2em]{0em}
\begin{subclaim}\label{IntersectPreciselyTwoP0P1SubCLMT} $L(p_0)=L(p_1)=L(x_0)\cap L(x_1)$. \end{subclaim}

\begin{claimproof} Recall that $L(p_0)\subseteq L(x_0)$ and $L(p_1)\subseteq L(x_1)$. Since $|L(x_0)\cap L(x_1)|\geq 2$, we have $L(p_0)\cap L(p_1)\neq\varnothing$. Thus, there exists a color $b$ such that $L(p_0)=\{a_0, b\}$ and $L(p_1)=\{a_1, b\}$.  By Subclaim \ref{xicontainedinqiForEachi=01}, we have $L(p_0)\subseteq L(x_0)\subseteq L(q_0)$. We have $b\not\in L(z)$ by B) of Claim \ref{IfZTwoChordsContainmentBothTriangleK0K1}. By Subclaim \ref{EitherContainFromPtoQOrListSize3Sub}, $|L(q_0)|=|L(z)|=5$. Thus, $|L(z)\setminus L(q_0)|\geq 1$, so let $s$ be a color of $L(z)\setminus L(q_0)$. Since $L(x_0)\subseteq L(q_0)$, we have $s\not\in L(x_0)$. Now, in order to show that Subclaim \ref{IntersectPreciselyTwoP0P1SubCLMT} holds, it suffices to show the following.

\begin{enumerate}[label=\roman*)]
\itemsep-0.1em
\item $s\in L(x_1)$; \emph{AND}
\item $a_0=a_1$
\end{enumerate}

If i) and ii) hold, then, since $s\not\in L(x_0)$, we have $L(p_0)=L(p_1)=L(x_0)\cap L(x_1)$. We first show that $s\in L(x_1)$. Suppose not. Since $L(p_1)\subseteq L(x_1)$, we have $b\in L(x_1)$, so $s\neq b$. Since $b\in L(x_0)$ as well, there is an $L$-coloring $\psi$ of $V(C)\setminus\{q_0, q_1, x_1\}$, where $\psi$ uses $a_0, s$ on the respective vertices $p_0, z$ and $\psi$ colors both of $x_0, p_1$ with $b$. Since $x_1$ is uncolored, we have $|L_{\psi}(q_1)|\geq 3$, and, by our choice of $s$, $|L_{\psi}(q_0)|\geq 3$ as well. Since $s\not\in L(x_1)$ and $x_0, p_1$ are colored with the same color, we have $|L_{\psi}(x_1)|\geq 2$. Thus, any extension of $\psi$ to an $L$-coloring of $\textnormal{dom}(\psi)\cup\{q_0, q_1\}$ (which is $V(G-x_1)$) extends to $x_1$ as well and thus extends to an $L$-coloring of $G$, contradicting our assumption that $G$ is a counterexample. We conclude that $s\in L(x_1)$. This proves i).

We now prove ii). We first note that $s\not\in L(p_1)$. To see this, recall that $s$ has been chosen so that $s\in L(z)\setminus L(q_0)$, so if $s\in L(p_1)$, then we contradict B) of Claim \ref{IfZTwoChordsContainmentBothTriangleK0K1}. Thus, $s\not\in L(p_1)$. Since $s\in L(x_1)$, we have $L(x_1)=L(p_1)\cup\{s\}$ as a disjoint union. Since $s\not\in L(x_0)$ and $|L(x_0)\cap L(x_1)|\geq 2$, we have $L(x_0)\cap L(x_1)=L(p_1)=\{a_1, b\}$. Now suppose toward a contradiction that $a_0\neq a_1$. Since $a_0$ is also distinct from each of $s, b$, we have $a_0\not\in L(x_1)$, and $L(x_0)=\{a_0, a_1, b\}$. Since $L(x_0)\subseteq L(q_0)$, there is an $L$-coloring $\sigma$ of $V(C-x_1)$, where $\sigma$ colors $p_0$ and $z$ with $a_0$ and $\sigma$ colors $x_0, p_1$ with $a_1$. Since $\sigma$ uses the same color on $p_0, z$, we have $|L_{\sigma}(q_0)|\geq 3$. Since $x_1$ is uncolored, we also have $|L_{\sigma}(q_1)|\geq 3$.  Since $a_0\not\in L(x_1)$ and $\sigma$ uses the same color on $x_0, p_1$, we have $|L_{\sigma}(x_1)|\geq 2$.  Thus, any extension of $\sigma$ to an $L$-coloring of $\textnormal{dom}(\sigma)\cup\{q_0, q_1\}$ (which is $V(G-x_1)$) extends to $x_1$ as well and thus extends to an $L$-coloring of $G$, contradicting our assumption that $G$ is a counterexample. We conclude that $a_0=a_1$. This completes the proof of Subclaim \ref{IntersectPreciselyTwoP0P1SubCLMT}. 
\end{claimproof}\end{addmargin}

By Subclaim \ref{IntersectPreciselyTwoP0P1SubCLMT}, there is a set $S$ with $|S|=2$ such that $L(p_0)=L(p_1)=S$ and $L(x_0)\cap L(x_1)=S$. We have $L(z)\cap S=\varnothing$ by B) of Claim \ref{IfZTwoChordsContainmentBothTriangleK0K1}. Let $S'=L(z)\setminus L(q_0)$. Recall that $S\subseteq L(q_0)$, so $|S'|\geq 2$ and $S\cap S'=\varnothing$. Since $|S|=2$ and $L(x_0)\cap L(x_1)=S$, there is an $r\in L(x_0)\setminus L(x_1)$. By Subclaim \ref{xicontainedinqiForEachi=01}, we also have $r\in L(q_0)$, so $r\not\in S'$. Since $|S'|\geq 2$ and $S\cap S'=\varnothing$, there is an $s'\in S'\setminus L(x_1)$. Since $r\not\in S'$, we have $r\neq s'$. Thus, there is an $L$-coloring $\psi$ of $V(C)\setminus\{q_0, q_1, x_1\}$, where $\psi$ uses $r, s'$ on the respective vertices $x_0, z$. Since $s'\in S'$, we have $|L_{\psi}(q_0)|\geq 3$. Since $x_1$ is uncolored, we have $|L_{\psi}(q_1)|\geq 3$ as well. By our choice of $r, s'$, we have $|L_{\psi}(x_1)|\geq 2$, so any extension of $\psi$ to an $L$-coloring of $\textnormal{dom}(\psi)\cup\{q_0, q_1\}$ extends to $x_1$ as well and thus extends to an $L$-coloring of $G$, contradicting our assumption that $G$ is a counterexample. This completes the proof of Claim \ref{noChordzOrAtLeast2Rule1Out}. \end{claimproof}

\subsection{Ruling Out A common neighbor of $v_0, v_1$ in $G\setminus C$}\label{RuleOutCommNbrToV0V1InG-C}

With Claims \ref{2ChordIncidentZVI} and \ref{noChordzOrAtLeast2Rule1Out} in hand, we can now dispense with the notation of Definition \ref{SubgraphsH0H1MForEdgezOfC}. To deal with the last few cases which make up this paper, we introduce the following definition.

\begin{defn} \emph{We let $G_{\pentagon}$ be the subgraph of $G$ bounded by outer cycle $C_{\pentagon}=(v_0(C\setminus\mathring{P})v_1)q_1zq_0$, and we let $P_{\pentagon}$ be the 4-path $v_0q_0zq_1v_1$.} \end{defn}

By Claim \ref{NoCommNbrQ0Q1OnPathu1Tout}, $v_0\neq v_1$, so $P_{\pentagon}$ is indeed a 4-path. Since $z$ is incident to no chords of $C$, every chord of $C$ is incident to one of $q_0, q_1$. By Claim \ref{Q0AndQ1AreNotAdjacentCL}, $q_0q_1\not\in E(G)$, so the outer cycle of $G_{\pentagon}$ is induced. 

\begin{Claim}\label{EdgeV0V1NotinEG} $v_0v_1\not\in E(G)$. In particular, $C_{\pentagon}$ has length at least six. \end{Claim}

\begin{claimproof} Suppose toward a contradiction that $v_0v_1\in E(G)$. Since the outer cycle of $G_{\pentagon}$ has no chords, it follows that $v_0v_1$ is an edge of $C\setminus\mathring{P}$, and $C_{\pentagon}$ has length precisely five. 

\vspace*{-8mm}
\begin{addmargin}[2em]{0em}
\begin{subclaim}\label{CPentColExtGPent} Any $L$-coloring of $V(C_{\pentagon})$ extends to $L$-color all of $G_{\pentagon}$ \end{subclaim}

\begin{claimproof} Suppose there is an $L$-coloring $\phi$ of $V(C_{\pentagon})$ which does not extend to $L$-color $G_{\pentagon}$. It follows from Theorem \ref{BohmePaper5CycleCorList}, there is a vertex adjacent to all five vertices of $\{q_0, z, q_1\}\cup\{v_0, v_1\}$, contradicting Claim \ref{2ChordIncidentZVI}. \end{claimproof}\end{addmargin}

\vspace*{-8mm}
\begin{addmargin}[2em]{0em}
\begin{subclaim} $|L(p_0)|=|L(p_1)|=2$. \end{subclaim}

\begin{claimproof} Suppose not, and suppose without loss of generality that $|L(p_0)|=1$ and $|L(p_1)|=3$. Note that there is a $(q_0, K_0)$-sufficient $L$-coloring $\phi$ of $Q_0-q_0$. This is immediate if $K_0$ is an edge, and otherwise it just follows from Theorem \ref{SumTo4For2PathColorEnds}. Since $|L_{\phi}(q_0)|\geq 3$ and no chord of $C$ is incident to $z$, it follows that, for each $c\in L(p_1)$, there is an extension of $\phi$ to an $L$-coloring $\sigma_c$ of $\{p_0, v_0, z, p_1\}$ with $|L_{\sigma_c}(q_0)|\geq 3$. Since $v_0q_1\not\in E(G)$, $|L_{\sigma_c}(q_1)|\geq 3$ as well and, since $G$ is a counterexample, $\sigma_c$ extends to an $L$-coloring $\sigma_c^*$ of $V(P)\cup\{v_0\}$ which does not extend to $L$-color $G$. If $K_1$ is an edge, then we just choose a $c\in L(p_1)\setminus\{\phi(v_0)\}$, and then it follows from Subclaim \ref{CPentColExtGPent} that $\sigma_c^*$ extends to an $L$-coloring of $G$, a contradiction. Thus, $K_1$ is not just an edge. For each $c\in L(p_1)$, it follows from Subclaim \ref{CPentColExtGPent} that $\sigma_*^c$ does not extend to an $L$-color $G\setminus V(G_{\pentagon}\setminus C_{\pentagon})$, and thus $\Lambda_{K_1}(\bullet, \sigma_c^*(q_1), c)=\{\phi(v_0)\}$. As $|L(p_1)|=3$, we contradict 1c) of Theorem \ref{BWheelMainRevListThm2}.  \end{claimproof}\end{addmargin}

\vspace*{-8mm}
\begin{addmargin}[2em]{0em}
\begin{subclaim}\label{K0K1EdgeCaseRuleSb} Neither $K_0$ nor $K_1$ is an edge. \end{subclaim}

\begin{claimproof} Suppose not. Since $|L(p_0)|=|L(p_1)|=2$, we suppose without loss of generality that $K_0$ is an edge. Thus, $K_1$ is not just an edge, or else we contradict Claim \ref{VCoutFaceLenGr5}. In particular, $u_1=v_1$. Applying Theorem \ref{SumTo4For2PathColorEnds}, we fix a $\phi\in\textnormal{End}(Q_1, K_1)$. Since $|L(p_0)\setminus\{\phi(u_1)\}|\geq 1$ and $z$ is incident to no chords of $C$, $\phi$ extends to an $L$-coloring $\phi'$ of $\{p_0, u_1, p_1, z\}$ with $|L_{\phi'}(q_1)|\geq 3$. Since $K_0$ is an edge, $q_0$ also has an $L_{phi'}$-list of size at least three, and since $G$ is a counterexample, $\phi'$ extends to an $L$-coloring $\psi$ of $V(P)\cup\{u_1\}$ which does not extend to $L$-color $G$. But since $\psi$ restricts to a proper $L$-coloring of $V(C_{\pentagon})$ and also restricts to an element of $\textnormal{End}(Q_1, K_1)$, it follows from Subclaim \ref{CPentColExtGPent} that $\psi$ extends to an $L$-coloring of $G$, a contradiction. \end{claimproof}\end{addmargin}

\vspace*{-8mm}
\begin{addmargin}[2em]{0em}
\begin{subclaim} Each of $K_0, K_1$ is a triangle. \end{subclaim}

\begin{claimproof} Suppose not. Since $|L(p_0)|=|L(p_1)|=2$, we suppose without loss of generality that $K_0$ is not a triangle. By 2) of Claim \ref{KiSizeAtMost4UnivColorBlock}, there is an $a\in L(p_0)$ which is almost $K_0$-universal. By Claim \ref{AnyLColMidEndExtPCL11}, there is an $L$-coloring $\sigma$ of $V(P)$ which does not extend to an $L$-coloring of $G$, where $\sigma(p_0)=a$. We now fix a $b\in\Lambda_{K_1}(\bullet, \sigma(q_1), \sigma(p_1))$. Since $a$ is almost $K_0$-universal, there is a color of $\Lambda_{K_0}(a, \sigma(q_0), \bullet)$ which is distinct from $b$, and since there are no chords of $C_{\pentagon}$, it follows from Subclaim \ref{CPentColExtGPent} that $\sigma$ extends to an $L$-coloring of $G$, a contradiction. \end{claimproof}\end{addmargin}

Since each of $K_0, K_1$ is a triangle, we have $v_0=u_1$ and $v_1=u_2$, and $C$ is a cycle of length seven. Since $|L(p_0)|=|L(p_1)|=2$, there is an $L$-coloring $\phi$ of $\{p_0, u_2, p_1\}$ with $|L_{\phi}(u_1)|\geq 2$. Since $|L_{\phi}(q_1)|\geq 3$ and no chord of $C$ is incident to $z$, there is an extension of $\phi$ to an $L$-coloring $\phi'$ of $\textnormal{dom}(\phi)\cup\{z\}$ with $|L_{\phi'}(q_1)|\geq 3$. Since $u_1$ is uncolored, $q_0$ also has an $L_{\phi'}$-list of size at least three. Since $G$ is a counterexample $\phi'$ extends to an $L$-coloring $\psi$ of $V(P)\cup\{u_2\}$ which does not extend to $L$-color $G$. But since there is a color left over for $u_1$, it follows from Subclaim \ref{CPentColExtGPent} that $\psi$ extends to an $L$-coloring of $G$, a contradiction. \end{claimproof}

Before proving the main result of this subsection, we prove the following two useful observations.

\begin{Claim}\label{AnLColPK0K1Ext} For any $i\in\{0,1\}$ and any $a\in L(p_{1-i})$,  
\begin{enumerate}[label=\arabic*)] 
\itemsep-0.1em
\item any $L$-coloring of either $V(P)$ or $V(P\cup K_i)$ extends to an $L$-coloring of $V(P\cup K_0\cup K_1)$ which uses $a$ on $p_{1-i}$; AND
\item For any $(q_i, K_i)$-sufficient $L$-coloring $\phi$ of $Q-q_i$ and any $s\in L(z)$ with $|L_{\phi}(q_i)\setminus\{s\}|\geq 3$, $\phi$ extends to an $L$-coloring of $V(P\cup K_0\cup K_1)$ which uses $a,s$ on the respective vertices $p_{1-i}, z$ and does not extend to an $L$-coloring of $G$. 
\end{enumerate}
 \end{Claim}

\begin{claimproof} By Claim \ref{EdgeV0V1NotinEG}, $v_0v_1\not\in E(G)$, and, since every chord of $C$ is incident to one of $q_0, q_1$, Claim \ref{EdgeV0V1NotinEG} also implies that, for each $i\in\{0,1\}$, $p_iv_{1-i}\not\in E(G)$, so 1) is immediate. Now we prove 2). Let $i\in\{0,1\}$ and $a\in L(p_{1-i})$ and let $\phi, s$ be as in 2). Since $p_{1-i}$ is not adjacent to either of $p_i, v_i, z$, $\phi$ extends to an $L$-coloring $\phi'$ of $\{p_{1-i}, z, v_i, p_i\}$ using $a,s$ on the respective vertices $p_{1-i}, z$. Since $v_0\neq v_1$, each of $q_0, q_1$ has an $L_{\phi'}$-list of size at least three. Since $G$ is a counterexample and $\phi$ is a $(q_i, K_i)$-sufficient $L$-coloring $\phi$ of $Q-q_i$, it follows that $\phi'$ extends to an $L$-coloring of $V(P\cup K_i)$ which does not extend to an $L$-coloring of $G$. By 1), $\phi'$ extends to an $L$-coloring of $V(P\cup K_0\cup K_1)$ which does not extend to an $L$-coloring of $G$. \end{claimproof}

\begin{Claim}\label{YCommNeighborReplace} For each $i\in\{0,1\}$, there exists an $L$-coloring $\psi$ of $V(P\cup K_0\cup K_1)$, where $\psi$ does not extend to an $L$-coloring of $G$ such that, for any $y\in V(G\setminus C)$ which is a common neighbor to $z, v_i$, $|L_{\psi}(y)|\geq 3$.  \end{Claim}

\begin{claimproof} Suppose without loss of generality that $i=1$. If there is no $y\in V(G\setminus C)$ adjacent to both of $z, v_i$, then, since $G$ is a counterexample, we just choose an arbitrary $L$-coloring $\psi$ of $V(P)$ which does not extend to an $L$-coloring of $G$. By 1) of Claim \ref{AnLColPK0K1Ext}, $\psi$ extends to an $L$-coloring of $V(P\cup K_0\cup K_1)$, so we are done in that case. Now suppose that there is such a $y$. Since no chords of $C$ are incident to $z$, and $G$ has no induced 4-cycles, we have $yq_1\in E(G)$. Furthermore, since $K_{2,3}$-free and $q_1\in N(z)\cap N(v_1)$, $y$ is the unique common neighbor of $z, p_1$ in $G\setminus C$. Recall that, by Claim \ref{NoCommNbrQ0Q1OnPathu1Tout}, $v_0\neq v_1$, so, by Claim \ref{2ChordIncidentZVI}, $y$ is not adjacent to either $q_0$ or $v_0$. By Claim \ref{piAndzHavenoNbrOutsideOuterCycleHlpCLMa}, $K_1$ is not an edge. Applying Theorem \ref{SumTo4For2PathColorEnds}, we fix a $\phi\in\textnormal{End}(Q_1, K_1)$. Let $T=\{s\in L(z): |L_{\phi}(q_1)\setminus\{s\}|\geq 3\}$. Note that $|T|\geq 2$. As indicated above, $p_0$ is not adjacent to either of $v_1, p_1$. Now consider the following cases. 

\textbf{Case 1:} Either $\phi(v_1)\in L(z)$ or $\phi(v_1)\not\in L(y)$

In this case, $\phi$ extends to an $L$-coloring $\phi'$ of $\{z, v_1, p_1\}$ with $|L_{\phi}(q_1)|\geq 3$ and $|L_{\phi}(y)|\geq 4$, since we either color $z, v_1$ with the same color, or, if $\phi(v_1)\not\in L(y)$, then we just choose any $s\in T$ and use that color on $z$ instead. By 2) of Claim \ref{AnLColPK0K1Ext}, $\phi'$ extends to an $L$-coloring $\psi$ of $V(P\cup K_0\cup K_1)$ which does not extend to $L$-color $G$. Note that $|L_{\psi}(y)|\geq 3$, as $y$ is not adjacent to either $q_0$ or $v_0$, so we are done. 

\textbf{Case 2:} $\phi(v_1)\in L(y)\setminus L(z)$

In this case, we avoid precoloring $v_1$. Since $|L(z)|=|L(y)|=5$, there is a $c\in L(z)\setminus L(y)$. By Claim \ref{AnyLColMidEndExtPCL11}, there is an $L$-coloring $\sigma$ of $V(P)$ with $\phi(z)=c$, where $\sigma$ does not extend to an $L$-coloring of $G$. By 1) of Claim \ref{AnLColPK0K1Ext}, $\sigma$ extends to an $L$-coloring $\psi$ of $V(P\cup K_0\cup K_1)$. As$y$ is not adjacent to either $q_0$ or $v_0$, $|L_{\psi}(y)|\geq 3$, we are done. \end{claimproof}

With the intermediate results above in hand, we prove the main result of this subsection.

\begin{Claim}\label{Eachi01IndexV0V1QINotCom} $v_0, v_1$ have no common neighbor in $G\setminus C$. \end{Claim}

\begin{claimproof} Suppose toward a contradiction that $v_0, v_1$ have a common neighbor $w\in V(G\setminus C)$. Thus $G$ contains the 6-cycle $D:=zq_1v_1wv_0q_0$. Let $H^{\downarrow}=\textnormal{Ext}_G(D)\cap G_{\pentagon}$ and $H_{\uparrow}=\textnormal{Int}_G(D)$. This is illustrated in Figure \ref{V0V1CommonNbrHUpFigure}, where $D$ is in bold. It follows from Claim \ref{EdgeV0V1NotinEG} that any chord of $D$ lies in $H^{\uparrow}$. We now show that $H^{\uparrow}$ contains a chord of $D$.

\begin{center}\begin{tikzpicture}
\node[shape=circle,draw=black] (p0) at (-4,0) {$p_0$};
\node[shape=circle,draw=black] (u1) at (-3, 0) {$u_1$};
\node[shape=circle,draw=white] (u1+) at (-2, 0) {$\ldots$};
\node[shape=circle,draw=black] (v0) at (-1, 0) {$v_0$};
\node[shape=circle,draw=white] (mid) at (0, 0) {$\ldots$};
\node[shape=circle,draw=black] (v1) at (1, 0) {$v_1$};
\node[shape=circle,draw=white] (un+) at (2, 0) {$\ldots$};
\node[shape=circle,draw=black] (ut) at (3, 0) {$u_t$};
\node[shape=circle,draw=black] (p1) at (4, 0) {$p_1$};
\node[shape=circle,draw=black] (q0) at (-3,2) {$q_0$};
\node[shape=circle,draw=black] (q1) at (3,2) {$q_1$};
\node[shape=circle,draw=black] (z) at (0,4) {$z$};
\node[shape=circle,draw=black] (w) at (0, 1.5) {$w$};
\node[shape=circle,draw=white] (K0) at (-2.8, 1) {$K_0$};
\node[shape=circle,draw=white] (K1) at (2.8, 1) {$K_1$};
\node[shape=circle,draw=white] (Hdown) at (0, 0.6) {$H_{\downarrow}$};
\node[shape=circle,draw=white] (Hdown) at (0, 2.5) {$H^{\uparrow}$};
 \draw[-] (p1) to (ut);
 \draw[-, line width=1.8pt] (v0) to (q0) to (z) to (q1) to (v1);
\draw[-] (q1) to (p1);
\draw[-] (p0) to (u1) to (u1+) to (v0) to (mid) to (v1) to (un+) to (ut);
\draw[-] (p0) to (q0);
\draw[-, line width=1.8pt] (v0) to (w) to (v1);

\end{tikzpicture}\captionof{figure}{}\label{V0V1CommonNbrHUpFigure}\end{center}

\vspace*{-8mm}
\begin{addmargin}[2em]{0em}
\begin{subclaim}\label{DInducedOrPK0K1Col} Either $D$ is not an induced cycle, or any $L$-coloring of $V(P\cup K_0\cup K_1)$ extends to an $L$-coloring of $\textnormal{Ext}_G(D)$. \end{subclaim}

\begin{claimproof} Suppose that $D$ is induced and let $\tau$ be an $L$-coloring of $V(P\cup K_0\cup K_1)$. Since $D$ is induced, $|L_{\tau}(z)|\geq 3$, so it follows from 3) of Corollary \ref{CorMainEitherBWheelAtM1ColCor} that $L_{\tau}(z)\cap\Lambda_{H_{\downarrow}}^{v_0wv_1}(\tau(v_0), \bullet, \tau(v_1))\neq\varnothing$. Thus, $\tau$ extends to $L$-color $H_{\downarrow}$ as well, i.e $\tau$ extends to an $L$-coloring of $\textnormal{Ext}_G(D)$. \end{claimproof}\end{addmargin}

\vspace*{-8mm}
\begin{addmargin}[2em]{0em}
\begin{subclaim}\label{DIndThenHUpArrowTriang} If $D$ is an induced cycle, then
\begin{enumerate}[label=\arabic*)] 
\itemsep-0.1em
\item For each $s\in L(z)$, there is an $L$-coloring of $D$ using $s$ on $z$ which does not extend to $L$-color $H^{\uparrow}$; AND
\item $H^{\uparrow}\setminus D$ is a triangle in which each vertex is adjacent to a subpath of length precisely two, and one of these 2-paths has $z$ as its midpoint. 
\end{enumerate}
 \end{subclaim}

\begin{claimproof} Suppose that $D$ is induced. For each $s\in L(z)$, since $G$ is a counteexample, there is an $L$-coloring of $V(P)$ which uses $s$ on $z$ and does not extend to an $L$-coloring of $G$. Thus, it follows from 1) of Claim \ref{AnLColPK0K1Ext} that, for each $s\in L(z)$, there is an $L$-coloring of $V(P\cup K_0\cup K_1)$ which does not extend to $L$-color $G_{\pentagon}$. Applying Subclaim \ref{DInducedOrPK0K1Col}, we conclude that, for each $s\in L(z)$, there is an $L$-coloring $\phi_s$ of $V(D)$ which does not extend to $L$-color the interior of $D$.  This proves 1). Now we prove 2). Suppose toward a contradiction that 2) does not hold. We apply Theorem \ref{BohmePaper5CycleCorList}. If there is a lone vertex of $H^{\uparrow}\setminus D$ adjacent to all six vertices of $D$, then we contradict Claim \ref{2ChordIncidentZVI}, so $H^{\uparrow}\setminus D$ is either an edge or a triangle. Furthermore, if $H^{\uparrow}\setminus D$ is a triangle, then each vertex of $H^{\uparrow}\setminus D$ is adjacent to a subpath of $D$ of length two, and, if $H^{\uparrow}\setminus D$ is a triangle, then each vertex of $H^{\uparrow}\setminus D$ is adjacent to a subpath of $D$ of length three. Consider the following cases.

\textbf{Case 1:} There exists a $y\in V(H^{\uparrow}\setminus D)$ such that $z$ is an internal vertex of the path $G[N(y)\cap V(D)]$

In this case, since 2) does not hold, $H^{\uparrow}\setminus D$ is an edge, each endpoint of which is adjacent to a subpath of $D$ of length three. It follows from the assumption of Case 1 that there is an $i\in\{0,1\}$ such that $y$ is adjacent to the four consecutive vertices $q_izq_{1-i}v_{1-i}$ of $D$, contradicting Claim \ref{2ChordIncidentZVI}. 

\textbf{Case 2:} There is no $y\in V(H^{\uparrow}\setminus D)$ such that $z$ is an internal vertex of the path $G[N(y)\cap V(D)]$

In this case, $H^{\uparrow}\setminus D$ is either an edge or a triangle, and, in either case, there exist vertices $y_0, y_1\in V(H^{\uparrow}\setminus D)$ such that, for each $i=0,1$, $G[N(y_i)\cap V(D)]$ is a subpath of $D$ of length either two or three with $zq_iv_i$ as a subpath. By Claim \ref{YCommNeighborReplace}, there is an $L$-coloring $\psi$ of $V(P\cup K_0\cup K_1)$ which does not extend to an $L$-coloring of $G$, where $|L_{\psi}(y_1)|\geq 3$. By \ref{DInducedOrPK0K1Col} $\psi$ extends to an $L$-coloring $\psi'$ of $\textnormal{Ext}_G(D)$. Suppose now that $H^{\uparrow}\setminus D$ is an edge. In this case,  each vertex of $H^{\uparrow}\setminus D$ is adjacent to a subpath of $D$ of length precisely three. Since $|L_{\psi}(y_1)|\geq 3$, we have $|L_{\psi'}(y_1)|\geq 2$. Since $psi'$ does not extend to an $L$-coloring of $H^{\uparrow}$, the two vertices of $H^{\uparrow}\setminus D$ have the same $L_{\psi'}$-list of size one, a contradiction.  Now suppose that $H^{\uparrow}\setminus D$ is a triangle. Thus, each vertex of $H^{\uparrow}\setminus D$ is adjacent to a subpath of $D$ of length precisely two, and $L_{\psi'}(y_1)=L_{\psi}(y_1)$, so $|L_{\psi'}(y_1)|\geq 3$. But since $\psi'$ does not extend to $L$-color $H^{\uparrow}$, each vertex of $H^{\uparrow}\setminus D$ has the same 2-list, a contradiction. \end{claimproof}\end{addmargin}

\vspace*{-8mm}
\begin{addmargin}[2em]{0em}
\begin{subclaim}\label{DNotInducedSubCL} $D$ is not an induced cycle.  \end{subclaim}

\begin{claimproof} Suppose that $D$ is induced. By 2) of Subclaim \ref{DIndThenHUpArrowTriang}, $H^{\uparrow}\setminus D$ is a triangle $y_0y_1y_2$, each vertex of which is adjacent to a subpath of $D$ of length precisely two, where one of these 2-paths has $z$ as its midpoint, so, for the sake of definiteness, we suppose that $G[N(y_2)\cap V(D)]$ is the path $q_0zq_1$ and, for each $i=0,1$, $G[N(y_i)\cap V(D)]$ is the path $q_iv_iw$. By 1) of  Subclaim \ref{DIndThenHUpArrowTriang}, for each $s\in L(z)$, there is an $L$-coloring $\phi_s$ of $V(D)$ which does not extend to an $L$-coloring of $H^{\uparrow}$. In particular, for each $s\in L(z)$, the three vertices of $y_0y_1y_2$ have the same 2-list. It follows that $s\in L(y_2)$, and, for each $i\in\{0,1\}$, $\phi_s$ uses $s$ on a vertex of $N(y_i)\cap V(D)$. Since this holds for each $s\in L(z)$, we have $L(z)=L(y_0)=L(y_1)=L(y_2)$. Neither $K_0$ nor $K_1$ is an edge, or else we contradict Claim \ref{piAndzHavenoNbrOutsideOuterCycleHlpCLMa}. Applying Theorem \ref{SumTo4For2PathColorEnds}, we fix a $\sigma\in\textnormal{End}(Q_1, K_1)$. Consider the following cases.

\textbf{Case 1:} $\sigma(v_1)\not\in L(y_1)$

In this case, since $|L_{\sigma}(q_1)|\geq 3$, we just choose an arbitrary $s\in L(z)$ such that $|L_{\sigma}(q_1)|\geq 3$. By 2) of Claim \ref{AnLColPK0K1Ext}, $\sigma$ extends $L$-coloring $\sigma'$ of $V(P\cup K_0\cup K_1)$ which does not extend to an $L$-coloring of $G$, where $\sigma'(z)=s$. By Subclaim \ref{DInducedOrPK0K1Col}, $\sigma'$ extends to an $L$-coloring $\sigma^*$ of $\textnormal{Ext}_G(D)$. But then, since $|L_{\sigma^*}(y_1)|\geq 3$, $\sigma^*$ extends to an $L$-coloring of $G$, which is false.

\textbf{Case 2:} $\sigma(v_1)\in L(y_1)$.

In this case, we let $c=\sigma(v_1)$ and we color both of $z, v_1$ with $c$. By 2) of Claim \ref{AnLColPK0K1Ext}, $\sigma$ to an $L$-coloring $\sigma'$ of $V(P\cup K_0\cup K_1)$ which does not extend to an $L$-coloring of $G$, where $\sigma'$ uses $c$ on $z$ as well. By Subclaim \ref{DInducedOrPK0K1Col}, $\sigma'$ extends to an $L$-coloring $\sigma^*$ of $\textnormal{Ext}_G(D)$, and since $\sigma^*$ does not extend to an $L$-coloring of $G$, the three vertices of $H^{\uparrow}\setminus D$ have the same $L_{\sigma^*}$-list of size two. Since $L(y_k)=L(z)$ for all $k=0,1,2$, it follows that $\sigma'(v_0)=c$ as well. We now let $S=L_{\sigma'}(w)\cap\Lambda_{H_{\downarrow}}^{v_0wv_1}(c, \bullet, c)$. Since $D$ is an induced cycle and the same color is used on $v_0, v_1$, we have $|L_{\sigma'}(w)|\geq 4$. It follows from 3) of Corollary \ref{CorMainEitherBWheelAtM1ColCor} that at most two colors of $L_{\sigma'}(w)$ lie outside of $\Lambda_{H_{\downarrow}}^{v_0wv_1}(c, \bullet, c)$, so $|S|\geq 2$. Note that, since $\sigma'$ does not color $w$, we have $|L_{\sigma'}(y_1)|\geq 3$, so there is an $r\in L_{\sigma'}(y_1)$ with $|S\setminus\{r\}|\geq 2$. By first coloring $y_1$ with $r$, extend $\sigma'$ to an $L$-coloring of $\textnormal{Ext}_G(C^{H_{\downarrow}})$ which uses a color of $S$ on $w$, so $\sigma'$ extends to an $L$-coloring of $G$, which is false. \end{claimproof}\end{addmargin}

Applying Subclaim \ref{DNotInducedSubCL}, we have the following. 

\vspace*{-8mm}
\begin{addmargin}[2em]{0em}
\begin{subclaim}\label{wzEdgeNotInEGSubC} There is precisely one chord of $D$, and this chord is $q_iw$ for some $i\in\{0,1\}$. \end{subclaim}

\begin{claimproof} We first show that $wz\not\in E(G)$. Suppose that $wz\in E(G)$. Since $G$ contains no induced 4-cycles and there are no chords of $C$ incident to $z$, it follows that $w$ is adjacent to all three of $q_0, w, q_1$. Since $w$ is adjacent to each of $q_0, q_1$ and $v_0\neq v_1$, we contradict Claim \ref{2ChordIncidentZVI}. To finish, it suffices to prove that there is precisely one chord of $D$. Suppose not. By Subclaim \ref{DNotInducedSubCL}, there are at least two chords of $D$. Since $wz\not\in E(G)$ and $G$ contains no induced 4-cycles, it follows that all three of $q_0q_1$ and $q_0w, q_1w$ are chords of $D$, contradicting Claim \ref{Q0AndQ1AreNotAdjacentCL}. \end{claimproof}\end{addmargin}

\vspace*{-8mm}
\begin{addmargin}[2em]{0em}
\begin{subclaim}\label{AnyDCyclePrecolorToHup} Any $L$-coloring of $V(D)$ extends to $L$-color all of $H^{\uparrow}$. \end{subclaim}

\begin{claimproof} Suppose toward a contradiction that there is an $L$-coloring $\tau$ of $V(D)$ which does not extend to $L$-color $H^{\uparrow}$. Let $D'$ be the induced unique 5-cycle whose vertices lie in $V(D)$. Since $\tau$ does not extend to an $L$-coloring of $G$,  it follows from Theorem \ref{BohmePaper5CycleCorList} that there is a $y\in V(H^{\uparrow})\setminus D$ adjacent to all five vertices of $D'$. Since $D'=wq_1zq_0v_0$ and $y$ is adjacent to all five of these vertices, we contradict Claim \ref{2ChordIncidentZVI}.\end{claimproof}\end{addmargin}

Let $H^{\ast}=(H_{\downarrow}+K_1)+wq_1$. The outer cycle of $H^{\ast}$ contains the 3-path $P^{\ast}=v_0wq_1p_1$. Let $X$ be the set of $L$-colorings $\psi$ of $\{v_0, p_1\}$ such that $\psi$ extends to two different elements of $\textnormal{End}(w, P^{\ast}, H^{\ast})$. As $p_1w\not\in E(G)$, and $|L(v_1)|=3$, it follows from 1) of Claim \ref{CornerColoringMainRes} that there are $|L(p_1)|$ different elements of $X$, each using a different color on $L(v_1)$. By Corollary \ref{GlueAugFromKHCor}, there is a $(q_0, K_0)$-sufficient $L$-coloring $\phi$ of $Q-q_0$ and a $\psi\in X$ with $\phi(v_0)=\psi(v_0)$. Let $s\in L(z)$ with $|L_{\phi}(q_0)\setminus\{s\}|\geq 3$. Since $q_0v_1\not\in E(G)$ and $G$ is a counterexample, $\phi\cup\psi$ extends to an $L$-coloring of $V(P)\cup\{v_0\}$ which does not extend to $L$-color $G$. It follows from our choice of $\phi$ that $\phi\cup\psi$ extends to an $L$-coloring $\sigma$ of $V(P\cup K_0)$ which does not extend to an $L$-coloring of $G$. Since $\psi\in X$, there exist two colors $c_0, c_1\in L_{\psi}(w)$ such that, for each $k=0,1$, $\psi$ extends to an element of $\textnormal{End}(w, P^{\ast}, H^{\ast})$ using $c_k$ on $w$. Since $wq_1$ is the lone chord of $D$, we have $L_{\sigma}(w)\cap\{c_0, c_1\}\neq\varnothing$. By Subclaim \ref{AnyDCyclePrecolorToHup}, $\sigma$ extends to an $L$-coloring of $V(P\cup K_0\cup H^{\uparrow})$ using one of $c_0, c_1$ on $w$, so $\sigma$ extends to an $L$-coloring of $G$, a contradiction. This completes the proof of Claim \ref{Eachi01IndexV0V1QINotCom}. \end{claimproof}

\subsection{Completing the proof of Theorem \ref{MainHolepunchPaperResulThm}}\label{CompSubSecTh}

\begin{Claim}\label{NoCommToBothQ0Q1ExcZ} There is no vertex of $G$ adajcent to both of $q_0, q_1$, except for $z$. \end{Claim}

\begin{claimproof} Suppose toward a contradiction that there is a $w\in N(q_0)\cap N(q_1)$ with $w\neq z$. We first note the following. 

\vspace*{-8mm}
\begin{addmargin}[2em]{0em}
\begin{subclaim}\label{NeighborsOfWOnlyQ0ZQ1} $N(w)\cap V(C)=\{q_0, z, q_1\}$ \end{subclaim}

\begin{claimproof} Since $z$ is incident to no chords of $C$, $w\in V(G\setminus C)$. By Claim \ref{Q0AndQ1AreNotAdjacentCL}, we have $q_0q_1\not\in E(G)$, and since there are no induced 4-cycles in $G$, it follows that $z\in N(w)$. SInce $z$ is also adjacent to both of $q_0, q_1$, and $v_0\neq v_1$, it follows from Claim Claim \ref{2ChordIncidentZVI} that $w$ has no neighbors in $C\setminus\mathring{P}$, so $N(w)\cap V(C)=\{q_0, z, q_1\}$. \end{claimproof}\end{addmargin}

Let $G'$ be the subgraph of $G$ bounded by outer cycle $(p_0(C\setminus\mathring{P})v_1)q_1wq_0$. That is, $G'=(G\setminus\{w\})\setminus (K_1\setminus\{v_1, q_1\})$. Now, the outer cycle of $G'$ contains the 4-path $P'=p_0q_0wq_1v_1$. We now apply Lemma \ref{EndLinked4PathBoxLemmaState} to $G'$ and $P'$. Since $V(C^{G'})\subseteq V(C)$, it follows from Subclaim \ref{NeighborsOfWOnlyQ0ZQ1} that the outer cycle of $G'$ has no chords incident to $w$. Let $\mathcal{F}$ be the set of partial $L$-colorings $\phi$ of $V(C')\setminus\{q_0, w, q_1\}$ such that
\begin{enumerate}[label=\arabic*)] 
\itemsep-0.1em
\item $\{p_0, v_1\}\subseteq\textnormal{dom}(\phi)\subseteq N(q_0)\cup N(q_1)$, and furthermore, $N(q_0)\cap\textnormal{dom}(\phi)\subseteq\{p_0, v_1\}$; \emph{AND}
\item For any two (not necessarily distinct) extensions of $\phi$ to $L$-colorings $\psi, \psi'$ of $\textnormal{dom}(\phi)\cup\{q_0, q_1\}$, at least one of the following holds
\begin{enumerate}[label=\roman*)]
\itemsep-0.1em
\item For some $c\in L(w)$, each of $\psi, \psi'$ extends to an $L$-coloring of $G'$ using $c$ on $w$; \emph{OR}
\item There is at an element of $\{\psi, \psi'\}$ which extends to two $L$-colorings of $G'$ using different colors on $w$; \emph{OR}
\item There is an element of $\{\psi, \psi'\}$ which does not extend to an $L$-coloring of $G'$
\end{enumerate}
\end{enumerate}

Recall that  $v_0\neq v_1$, so $q_0v_1\not\in E(G)$ and $N(q_0)\cap\textnormal{dom}(\phi)=\{p_0\}$. Furthermore, $p_0\not\in N(q_1)$ and no chords of the outer cycle of $G'$ are incident to $q_1$, so $N(q_1)\cap\textnormal{dom}(\phi)=\{v_1\}$, i.e each element of $\mathcal{F}$ is an $L$-coloring of $p_0, v_1$. By Corollary \ref{GlueAugFromKHCor}, there is a $\phi\in\mathcal{F}$ and a $(q_1, K_1)$-sufficient $L$-coloring $\psi$ of $Q-q_1$ such that $\phi(v_1)=\psi(v_1)$. Thus, $\phi\cup\psi$ is an $L$-coloring of $\{p_0, v_1, q_1\}$. Note that $|L_{\phi\cup\psi}(q_0)|\geq 4$ and $|L_{\phi\cup\psi}(q_1)|\geq 3$. As $|L(z)|=5$, there exist distinct $s_0, s_1\in L(z)$, where $|L_{\phi\cup\psi}(q_1)\setminus\{s_k\}|\geq 3$ for each $k=0,1$. Since $|L_{\phi}(q_0)|\geq 4$ and no chords of $C$ are incident to $z$, and since $G$ is a counterexample, it follows that, for each $k\in\{0,1\}$, there is an extension of $\phi\cup\psi$ to an $L$-coloring $\sigma_k$ of $V(P)\cup\{v_1\}$ which does not extend to an $L$-coloring of $G$, where $\sigma_k(z)=s_k$. It follows from our choice of $\psi$ that, for each $k=0,1$, $\sigma_k$ extends to an $L$-coloring $\tau_k$ of $V(P\cup K_1)$. For each $k=0,1$, we let $\tau_k'$ be the restriction of $\tau_k$ to $\{p_0, q_0, q_1, v_1\}$, i.e the restriction of $\tau_k$ to $V(P'-w)$. 

\vspace*{-8mm}
\begin{addmargin}[2em]{0em}
\begin{subclaim}\label{RestrictExtendsSubForTauPrime} For each $k=0,1$, $\tau_k'$ extends to an $L$-coloring of $G'$. \end{subclaim}

\begin{claimproof} Suppose that there is a $k\in\{0,1\}$ such that $\tau_k'$ does not extend to $L$-color $G'-w$. It follows from 1) of Claim \ref{AnLColPK0K1Ext} that $\tau_k'$ extends to an $L$-coloring $\pi$ of $V(K_0\cup K_1)$, so this $L$-coloring of $V(K_0\cup K_1)$ does not extend to $L$-color $G_{\pentagon}-z$. Now, $G_{\pentagon}-z$ is bounded by outer cycle $(v_0(C\setminus\mathring{P})v_1)q_1wq_0$, and, by Subclaim \ref{NeighborsOfWOnlyQ0ZQ1}, no chord of the outer cycle of $G_{\pentagon}-z$ is incident to $w$. Since the outer cycle of $G_{\pentagon}$ is induced, the outer cycle of $G_{\pentagon}-z$ is also induced. In particular, $|L_{\pi}(w)|\geq 3$. Since $\pi$ does not extend to $L$-color $G_{\pentagon}-z$, it follows from Lemma \ref{PartialPathColoringExtCL0} that there is a vertex $x$ of $G_{\pentagon}-z$ which does not lie on the outer cycle of $G_{\pentagon}-z$ and is adjacent to at least three of $\{v_0, v_1, q_0, q_1\}$. If $x$ is adjacent to both of $q_0, q_1$, then, since $x\neq z,w$, we contradict the fact that $G$ is $K_{2,3}$-free, so $x$ is adjacent to both of $v_0, v_1$, contradicting Claim \ref{Eachi01IndexV0V1QINotCom}. \end{claimproof}\end{addmargin}

For each $k=0,1$, since $\tau_k$ does not extend to an $L$-coloring of $G$, it follows that any extension of $\tau_k'$ to an $L$-coloring of $G'$ uses $s_k$ on $w$. By Subclaim \ref{RestrictExtendsSubForTauPrime}, each of $\tau_0', \tau_1'$ extends to at least one $L$-coloring of $G'-w$, and since $\phi\in\mathcal{F}$ and $s_0\neq s_1$, there is at least one $k\in\{0,1\}$ such that $\tau_k'$ extends to an $L$-coloring of $G'-w$ using a color other than $s_k$ on $w$, a contradiction. \end{claimproof}

We now introduce one final definition. For each $i\in\{0,1\}$, we let $\textnormal{Ob}(v_i)$ be the set of $y\in V(G\setminus C)$ such that $N(y)\cap V(P_{\pentagon})=\{v_i, q_i, z\}$.

\begin{Claim}\label{EachObViNonempty} For each $i=0,1$, $\textnormal{Ob}(v_i)|=1$. In particular, neither $K_0$ nor $K_1$ is an edge, and furthermore, every vertex of $G_{\pentagon}\setminus C_{\pentagon}$ with at least three neighbors in $P_{\pentagon}$ lies in $\textnormal{Ob}(v_0)\cup\textnormal{Ob}(v_1)$. \end{Claim}

\begin{claimproof} Note that, for each $i\in\{0,1\}$, we have $|\textnormal{Ob}(v_i)|\leq 1$, as $G$ is $K_{2,3}$-free, and furthermore, if $\textnormal{Ob}(v_i)\neq\varnothing$, then $K_i$ is not just an edge, or else we contradict Claim \ref{piAndzHavenoNbrOutsideOuterCycleHlpCLMa}.

\vspace*{-8mm}
\begin{addmargin}[2em]{0em}
\begin{subclaim}\label{ObV0ObV1AllObstructions} For any $y\in V(G_{\pentagon}\setminus C_{\pentagon})$ with at least three neighbors in $P_{\pentagon}$, $w\in\textnormal{Ob}(v_0)\cup\textnormal{Ob}(v_1)$. \end{subclaim}

\begin{claimproof} Suppose that $y\not\in\textnormal{Ob}(v_0)\cup\textnormal{Ob}(v_1)$. By Claim \ref{NoCommToBothQ0Q1ExcZ}, there is an $i\in\{0,1\}$ with $y\not\in N(q_i)$, say $i=0$ without loss of generality. As $y\not\in\textnormal{Ob}(v_1)$, either $N(y)\cap V(P_{\pentagon})=\{z, q_1, v_0\}$ or $N(y)\cap V(P_{\pentagon})=\{q_0, v_0, v_1\}$. In the former case, as $v_0\neq v_1$, we contradict Claim \ref{2ChordIncidentZVI}. In the latter case, we contradict Claim \ref{Eachi01IndexV0V1QINotCom}. \end{claimproof}\end{addmargin}

Now suppose that Claim \ref{EachObViNonempty} does not hold. Thus, there is a $i\in\{0,1\}$ with $\textnormal{Ob}(v_i)=\varnothing$, say $i=0$ without loss of generality. By Subclaim \ref{ObV0ObV1AllObstructions}, every vertex of $G_{\pentagon}\setminus C_{\pentagon}$ with at least three neighbors on $P_{\pentagon}$ lies in $\textnormal{Ob}(v_1)$. By Claim \ref{YCommNeighborReplace}, there is an $L$-coloring $\psi$ of $V(P\cup K_0\cup K_1)$ which does not extend to $L$-color $G$, where every vertex of $\textnormal{Ob}(v_1)$ has an $L_{\psi}$-list of size at least three. In particular, $\psi$ does not extend to $L$-color $G_{\pentagon}$. The outer cycle of $G_{\pentagon}$ is induced, so, by Lemma \ref{PartialPathColoringExtCL0}, there is a $y\in V(G_{\pentagon}\setminus C_{\pentagon})$ with $|L_{\psi}(y)|<3$. Thus, $y\in\textnormal{Ob}(v_1)$, a contradiction. \end{claimproof}

Applying Claim \ref{EachObViNonempty}, for each $i=0,1$, we let $y_i$ be the lone vertex of $\textnormal{Ob}(v_i)$. This is illustrated in Figure \ref{Y0Y1ObstructionVertexFigure}, where the path $P_{\pentagon}$ on the outer cycle of $G_{\pentagon}$ is indicated in bold. By Claim \ref{YCommNeighborReplace}, there is an $L$-coloring $\psi$ of $V(P\cup K_0\cup K_1)$ which does not extend to an $L$-coloring of $G$, where $|L_{\psi}(y_0)|\geq 3$. Note that $|L_{\psi}(y_1)|\geq 2$. 

\begin{center}\begin{tikzpicture}
\node[shape=circle,draw=black] (p0) at (-4,0) {$p_0$};
\node[shape=circle,draw=black] (u1) at (-3, 0) {$u_1$};
\node[shape=circle,draw=white] (u1+) at (-2, 0) {$\ldots$};
\node[shape=circle,draw=black] (v0) at (-1, 0) {$v_0$};
\node[shape=circle,draw=white] (mid) at (0, 0) {$\ldots$};
\node[shape=circle,draw=black] (v1) at (1, 0) {$v_1$};
\node[shape=circle,draw=white] (un+) at (2, 0) {$\ldots$};
\node[shape=circle,draw=black] (ut) at (3, 0) {$u_t$};
\node[shape=circle,draw=black] (p1) at (4, 0) {$p_1$};
\node[shape=circle,draw=black] (q0) at (-3,2) {$q_0$};
\node[shape=circle,draw=black] (q1) at (3,2) {$q_1$};
\node[shape=circle,draw=black] (z) at (0,4) {$z$};
\node[shape=circle,draw=black] (y0) at (-1.5,2) {$y_0$};
\node[shape=circle,draw=black] (y1) at (1.5,2) {$y_1$};
\node[shape=circle,draw=white] (Gpen) at (0, 2) {$G_{\pentagon}$};
\node[shape=circle,draw=white] (K0) at (-2.8, 1) {$K_0$};
\node[shape=circle,draw=white] (K1) at (2.8, 1) {$K_1$};
 \draw[-] (p1) to (ut);
 \draw[-, line width=1.8pt] (v0) to (q0) to (z) to (q1) to (v1);
\draw[-] (q1) to (p1);
\draw[-] (p0) to (u1) to (u1+) to (v0) to (mid) to (v1) to (un+) to (ut);
\draw[-] (p0) to (q0);
\draw[-] (q0) to (y0);
\draw[-] (q1) to (y1);
\draw[-] (z) to (y1) to (v1);
\draw[-] (z) to (y0) to (v0);
\end{tikzpicture}\captionof{figure}{}\label{Y0Y1ObstructionVertexFigure}\end{center}

\begin{Claim}\label{EdgeY0Y1tPresent} $y_0y_1\in E(G)$. \end{Claim}

\begin{claimproof} Suppose not. Now, $G_{\pentagon}-q_1$ is bounded by outer cycle $(v_0(C\setminus P)v_1)y_1zq_0$. By Claim \ref{2ChordIncidentZVI}, $y_1$ has no neighbors on the the path $(v_0(C\setminus P)v_1)$, except for $v_1$. Since the outer cycle of $G_{\pentagon}$ is induced and $y_0y_1\not\in E(G)$, the outer cycle of $G_{\pentagon}-q_1$ is also induced. Since $|L_{\psi}(y_1)|\geq 2$, there is an extension of $\psi$ to an $L$-coloring $\psi^*$ of $\textnormal{dom}(\psi)\cup\{y_1\}$. Since $y_0y_1\not\in E(G)$, we have $|L_{\psi^*}(y_1)|\geq 3$. Since $\psi^*$ does not extend to an $L$-coloring of $G$, it follows from Lemma \ref{PartialPathColoringExtCL0} that there is a $w\in V(G_{\pentagon}-q_1)$ with at least three neighbors among $\{v_0, z, y_1, v_1\}$, where $w$ does not lie on the outer cycle of $G_{\pentagon}-q_1$ and $w\neq y_0$. Thus, $w$ is adjacent to at least one of $v_0, v_1$, so $w\not\in N(z)$, as $G$ is short-separation-free. It follows that $w$ is adajcent to all three of $v_0, v_1, y_1$, contradicting Claim \ref{Eachi01IndexV0V1QINotCom}.  \end{claimproof} 

Since $y_0y_1\in E(G)$, we let $H_{\pentagon}$ be the subgraph of $G_{\pentagon}$ bounded by outer cycle $(v_0(C\setminus\mathring{P})v_1)y_1y_0$. The outer cycle of $H_{\pentagon}$ contains the 3-path $R=v_0y_0y_0v_1$, and, since $G$ is short-separation-free, $V(G)=V(K_0\cup K_1\cup H_{\pentagon})\cup\{z\}$.  Since every chord of $C$ is incident to one of $q_0, q_1$, it follows from Claim \ref{2ChordIncidentZVI} that the outer cycle of $H_{\pentagon}$ is induced. Since $\psi$ extends to an $L$-coloring of $\textnormal{dom}(\psi)\cup\{y_0, y_1\}$, but $\psi$ does not extend to $L$-color $G$, it follows from Lemma \ref{PartialPathColoringExtCL0} that there is a vertex $V(H_{\pentagon})$ which does not lie on the outer face of $H_{\pentagon}$, where $w$ has at least three neighbors in $R$.  By Claim \ref{Eachi01IndexV0V1QINotCom}, $w$ is adjacent to at most one of $v_0, v_1$, so there exists a $k\in\{0,1\}$ such that $N(w)\cap V(R)=\{v_k, y_0, y_1\}$. Since the outer cycle of $H_{\pentagon}$ is induced, every chord of the outer cycle of $H_{\pentagon}-y_k$ is incident to $w$.

\begin{Claim} $k=1$. \end{Claim}

\begin{claimproof} Suppose not. Thus, $k=0$ and $N(w)\cap V(R)=\{v_0, y_0, y_1\}$, and $y_0$ is the central vertex of a wheel with six vertices, i.e $N(y_0)=\{v_0, q_0, z, w, y_1\}$.  Let $c$ be an arbitrary color of $L_{\psi}(y_1)$ and let $\sigma$ be an $L$-coloring of $\{y_1, v_0, v_1\}$ in which $y_1, v_0, v_1$ are colored with the respective colors $c, \psi(v_0), \psi(v_1)$. Since every chord of the outer face of $H_{\pentagon}-y_0$ is incident to $w$ and $w$ is not adjacent to $v_1$, it follows from Lemma \ref{PartialPathColoringExtCL0} that $\sigma$ extends to an $L$-coloring of $H_{\pentagon}-y_0$, so $\psi$ extends to an $L$-coloring of $G-y_0$. Since $|L_{\psi}(y_0)|\geq 3$ and $y_0$ only has five neighbors, there is a color left over for $y_0$, so $\psi$ extends to an $L$-coloring of $G$, which is false. \end{claimproof}

Since $k=1$, we have $N(w)\cap V(R)=\{v_1, y_0, y_1\}$, and $y_1$ is the central vertex of a wheel with six vertices, i.e $N(y_1)=\{v_1, q_1, z, w, y_0\}$. Since $|L_{\psi}(y_0)|\geq 3$, there is a $c\in L_{\psi}(y_0)$ with $|L_{\psi}(y_1)\setminus\{c\}|\geq 2$. Let $\sigma$ be an $L$-coloring of $\{y_0, v_0, v_1\}$ which uses $c, \psi(v_0), \psi(v_1)$ on the respective vertices $y_0, v_0, v_1$. As every chord of the outer face of $H_{\pentagon}-y_1$ is incident to $w$, and $w$ is not adjacent to $v_0$, it follows from Lemma \ref{PartialPathColoringExtCL0} that $\sigma$ extends to an $L$-coloring of $H_{\pentagon}-y_1$, so $\psi$ extends to an $L$-coloring of $G-y_1$ using $c$ on $y_0$. Since $y_1$ only has five neighbors, it follows from our choice of $c$ that $\psi$ extends to an $L$-coloring of $G$, which is false. This completes the proof of Theorem \ref{MainHolepunchPaperResulThm}.  \end{proof}

\end{document}